\documentclass[11pt,twoside]{article} 

\usepackage{eqnarray,amsmath}
\usepackage[utf8]{inputenc} 
\usepackage[T1]{fontenc}    
\usepackage{booktabs}       
\usepackage{amsfonts}       
\usepackage{nicefrac}       
\usepackage{microtype}      
\usepackage{xcolor}         
\usepackage{subcaption}
\expandafter\def\csname 
ver@subfig.sty\endcsname{}
\usepackage{eqnarray,amsmath}
\usepackage{epsf}
\usepackage{epsfig}
\usepackage{fancyhdr}
\usepackage{graphics}
\usepackage{graphicx}
\usepackage{psfrag}
\usepackage{fullpage}
\usepackage{pdfpages}
\usepackage{amssymb}
\usepackage{natbib}

\usepackage{url}

\usepackage{color}

\usepackage{amsthm,amsmath,amsfonts,amssymb,mathrsfs,booktabs}
\usepackage[skip=0pt]{caption}  

\usepackage{textcomp}
\usepackage{siunitx}
\usepackage{wrapfig}
\usepackage{algorithm}

\usepackage{multirow}
\usepackage{multicol}
\usepackage{makecell}
\usepackage{colortbl} 
\usepackage{changepage}

\definecolor{customyellow}{HTML}{fedf8a}
\usepackage{eqnarray,amsmath}
\usepackage{epsf}
\usepackage{epsfig}
\usepackage{fancyhdr}
\usepackage{graphics}
\usepackage{graphicx}
\usepackage{psfrag}
\usepackage{fullpage}
\usepackage{pdfpages}
\usepackage{ragged2e}

\newtheorem{assumption}{Assumption}

\newtheorem{lemma}{Lemma}

\newtheorem{theorem}{Theorem}
\newtheorem{proposition}{Proposition}
\newtheorem{definition}{Definition}

\usepackage{eqnarray,amsmath}
\usepackage{booktabs}       
\usepackage{nicefrac}       
\renewcommand{\arraystretch}{1.35}

\usepackage{subfig}
\usepackage{epsf}
\usepackage{epsfig}
\usepackage{fancyhdr}
\usepackage{graphics}
\usepackage{graphicx}
\usepackage{epstopdf}
\usepackage{psfrag}
\usepackage{fullpage}
\usepackage{pdfpages}

\usepackage{enumerate}   
\usepackage{multirow}
\usepackage{bm}

\usepackage{mathtools}

\usepackage{url}
\usepackage[colorlinks,linkcolor=black,citecolor=blue, pagebackref=true]{hyperref}
\renewcommand*{\backrefalt}[4]{%
    \ifcase #1 \footnotesize{(Not cited.)}%
    \or        \footnotesize{(Cited on page~#2.)}%
    \else      \footnotesize{(Cited on pages~#2.)}%
    \fi}

\usepackage{color}

\usepackage{amsthm}
\usepackage{amsmath}
\usepackage{amssymb,bbm}
\usepackage{caption}
\usepackage{algorithmic}
\usepackage{algorithm}
\usepackage{textcomp}
\usepackage{siunitx}
\usepackage{wrapfig}
\usepackage{algorithmic}
\usepackage{algorithm}
\usepackage{multirow}
\usepackage{multicol}
\usepackage{mathtools}

\def\1{\bm{1}}

\DeclareMathAlphabet{\mathsfit}{\encodingdefault}{\sfdefault}{m}{sl}
\SetMathAlphabet{\mathsfit}{bold}{\encodingdefault}{\sfdefault}{bx}{n}

\usepackage{eqnarray,amsmath}

\usepackage{amsthm}
\usepackage{amsmath}
\usepackage{amssymb,bbm}

\allowdisplaybreaks

\begin{document}

\begin{center}

{\bf{\LARGE{Convergence Rates for Latent Mixing Measures in Infinite Homoscedastic Location-Scale Mixture Models}}}
  
\vspace*{.2in}
{\large{
\begin{tabular}{cccc}
Nicola Bariletto\footnotemark & Dung Le\footnotemark[1] & Alessandro Rinaldo & Nhat Ho
\end{tabular}
}}

\footnotetext{Equal contribution.}

\vspace*{.2in}

\begin{tabular}{c}
The University of Texas at Austin
\end{tabular}

\vspace*{.2in}
\today


\begin{abstract}
    We study posterior contraction rates for mixing measures in homoscedastic location-scale mixture models with infinitely many components. While posterior convergence at the level of densities is  well understood, ensuring convergence of the latent mixing measure is more challenging and has remained an open problem in settings where both location and scale parameters are unknown. We address this by deriving novel lower-bounds that connect the $L^1$ distance between mixture densities to discrepancies, based on the Wasserstein distances and the operator norm, between the underlying mixing measures and scale matrices. Our approach combines the dual formulation of the $W_1$ distance with functional-analytic approximation techniques. This leads to general inequalities, whose strength is determined (i) by the smoothness of the mixture kernel via the rate of decay of its characteristic function, and (ii) by a key lower-bound on the $L^1$ metric involving the operator norm discrepancy between scale parameters. Moreover, a novel PDE inversion condition yields a sharper inequality for important ordinary-smooth cases. We specialize these bounds to popular mixtures based on multivariate Gaussian, Cauchy, and Laplace kernels. As a consequence, we obtain first-of-their-kind contraction rates in the context of Dirichlet process mixtures with an unknown scale parameter shared across components. As a byproduct of our inequalities, we can distinguish the convergence behavior of the location mixing measure from that of the scale parameter across a range of kernel choices, leading to nuanced insights into their respective rates.
\end{abstract}
\end{center}


\section{Introduction}\label{sec:introduction}

Mixture modeling is a central tool in modern statistical analysis, due to its flexibility and the interpretability of the implied generative process. Concretely, one models data $X_1,\ldots,X_n\in\mathbb R^d$ as arising from a density of the form
\begin{equation*}
    p_G(x):=\int_{\Theta} f(x\mid \theta) G( d\theta),
\end{equation*}
where $G$, a probability measure on the parameter space $\Theta$, is known as the \emph{mixing measure}, while $\{f(\cdot\mid \theta) : \theta\in\Theta\}$ is a family of probability density kernels indexed by $\Theta$. The statistical problem, then, consists in estimating $G$ and the resulting density $p_G$. A popular methodological choice involves modeling $G$ as discrete, i.e., $G=\sum_{j=1}^k w_j \delta_{\theta_j}$ with $k\in\mathbb N\cup\{\infty\}$ and $(w_j)_{j=1}^k$ a sequence of positive weights summing to one. This formulation is appealing for two reasons. First, it is highly expressive: for instance, any density can be approximated in the $L^1$ sense by discrete mixtures with suitable kernels. Second, from a generative perspective, it induces a latent partition of the data into at most $k$ groups $j=1,\ldots,k$, according to which component $f(\cdot\mid\theta_j)$ generates each observation, which provides a natural basis for clustering.

Mixture models have received substantial attention in statistics and machine learning over the last decades, especially within the Bayesian community \citep{mclachlan2000finite, fraley2002model} following the introduction of Bayesian nonparametric models that place flexible priors, such as the Dirichlet process \citep{Ferguson}, on discrete mixing measures with $k=\infty$ \citep{ferguson1983bayesian,lo1984class,neal2000markov}. The appeal of the nonparametric approach is that it allows for flexible density estimation and clustering schemes, often making both the data partition and the underlying mixing measure the two primary objects of interest \citep{Nguyen-13,wade2023bayesian}. This flexibility has also led to computational methods for handling the resulting infinite-dimensional structure \citep{escobar1995bayesian,maceachern1998estimating,walker2007sampling}, as well as to extensive theoretical analysis. This article offers several new contributions to such line of work, which we now briefly review.

\subsection{Related literature}
Much of the existing theory has focused on the asymptotic behavior of Bayesian mixtures for density estimation. To introduce the problem, define the $L^1$ distance $\|p-p'\|_1:=\int_{\mathbb R^d}|p(x)-p'(x)|dx$ between densities $p$ and $p'$ with respect to the Lebesgue measure on $\mathbb R^d$, and suppose that $X_i\overset{\text{i.i.d.}}\sim p_{G_0}$ for some unknown mixing measure $G_0$. A natural question is whether, given a prior $\Pi$ supported on a suitable space of mixing measures $\mathscr G$ with $G_0\in\mathscr G$, the posterior distribution
\begin{equation*}
    \Pi(A\mid X_1,\ldots,X_n) := \frac{\int_A \prod_{i=1}^n p_G(X_i)\Pi( d G)}{\int_{\mathscr{G}} \prod_{i=1}^n p_{G'}(X_i)\Pi( d G')}
\end{equation*}
concentrates in $L^1$ neighborhoods of $p_{G_0}$ as $n\to\infty$. This question has been studied extensively, both for fixed but arbitrarily small neighborhoods \citep[leading to $L^1$ posterior consistency,][]{Ghosal-1999,Barron-Shervish-Wasserman-99,walker2001,walker2004squarerootsum,lijoi2005consistency,tokdar2006posterior,wu2010l1,chae2017,bariletto2025posterior,ghosal2017fundamentals}, and for shrinking neighborhoods, yielding $L^1$ posterior contraction rates $\omega_n$ satisfying
\begin{equation}\label{eq:L1_contraction_rate}
    \Pi\left(G : \|p_G - p_{G_0}\|_1\geq K\omega_n\mid X_1,\ldots,X_n\right) \to 0
\end{equation}
in $P_{G_0}^n$-probability for all large enough $K>0$ as $n\to\infty$; here, $P_G^n$ denotes the $n$-fold product measure induced by $P_G( dx):=p_G(x)dx$, while $(\omega_n)_{n\in\mathbb N}$ is a positive vanishing sequence. Investigating sufficient conditions for posterior behavior like in Equation~\eqref{eq:L1_contraction_rate}, as done by \citep{Ghosal-2000,genovese2000rates,shen2001rates,ghosal2001entropies,walker2007rates,ghosal2007noniid,Ghosal-2007,gine2011rates,shen2013adaptive,scricciolo2011posterior,canale2017posterior,ghosal2017fundamentals,castillo2024bnp}, provides useful information on the asymptotic rate at which the posterior concentrates around the ground truth $p_{G_0}$, and yields guidance on model choices that support efficient density estimation.

However, $L^1$ contraction rates do not fully characterize posterior behavior in terms of clustering and contraction to the true mixing measure, since the evolution of the mixture components is a more refined phenomenon than $L^1$ convergence of the density. Several works have addressed this problem, focusing for instance on the recovery of $k_0$, the true number of components \citep{miller2013simple,miller2014inconsistency,ascolani2023clustering}, and on the allocation of data to redundant components in overfitted mixture models \citep{rousseau2011asymptotic}. A prominent strand of literature, pioneered by \cite{Nguyen-13}, studies posterior contraction towards $G_0$ using the Wasserstein distances to measure discrepancy between mixing measures \citep[see also][]{guha2021Bernoulli,Ho-Nguyen-Ann-16,Ho-Nguyen-EJS-16,rousseau2024wasserstein}. Concretely, if $\rho$ is a metric on $\Theta$, the corresponding \emph{Wasserstein distance of order $r\in[1,\infty)$} is defined as
\begin{equation*}
    W_{r}(G,G'):=\left(\inf_{\pi\in\mathcal P(G,G')}\int_{\Theta\times \Theta}\rho(\theta,\theta')^r\pi( d\theta, d\theta')\right)^{1/r},
\end{equation*}
where $\mathcal P(G,G')$ is the set of probability measures $\pi$ on $\Theta\times \Theta$ with marginals $G$ and $G'$. Besides their solid grounding in the rich field of optimal transport \citep{Villani-09}, the appeal of the $W_r$ metrics in the study of mixture models lies in their ability to capture discrepancies between mixing measures that may be discrete and have disjoint supports, a feature not shared by many classical probability metrics. These features have motivated significant recent interest in optimal transport metrics within the Bayesian nonparametric community, for instance to measure dependence between random measures \citep{catalano2021measuring,catalano2024wasserstein,catalano2025measures}.

Convergence analyses for mixing measures with infinitely many atoms, such as those arising from a Dirichlet process prior, have been presented only for a restricted class of mixtures, namely \emph{location mixtures}, where $f(x\mid \theta)= g(x-\theta)$ for $\theta\in\Theta\subseteq\mathbb R^d$ and some density kernel $g$. Location mixtures have important applications, for instance in deconvolution problems with additive noise of known scale, as studied in a Bayesian nonparametric framework by \cite{Nguyen-13,Gao2016Posterior,scricciolo2018bayes,rousseau2024wasserstein}; see also \cite{Caillerie-etal-11,dedecker2013minimax,dedecker2015improved} for non-Bayesian analyses using Wasserstein distances. However, restricting variation to a location parameter severely limits flexibility in learning the underlying density and mixing structure, often motivating the adoption of more complex model designs.

\subsection{Contributions of the article}
A more flexible and  popular alternative to location mixtures are \emph{location-scale mixtures} associated to families of kernels of the form
\begin{equation*}
    f(x\mid\theta,\Sigma)=|\Sigma|^{-1/2} g\left((x-\theta)^\top\Sigma^{-1}(x-\theta)\right),
\end{equation*}
where $\theta\in\mathcal X\subseteq\mathbb R^d$ is a location parameter, $\Sigma\in\mathbb R^{d\times d}$ is a positive-definite symmetric scale matrix (with $|\Sigma|$ denoting its determinant), and $g$ is a typically symmetric kernel. Despite their prominence in applications, the posterior behavior of nonparametric mixing measures with a prior on both the location and scale parameters has, to the best of our knowledge, never been studied before. Our goal is to provide, for the first time, such an analysis for priors supported on mixing measures of the form
\begin{equation*}
    G(d\theta,dS) = P(d\theta)\times\delta_{\Sigma}(dS),
\end{equation*}
that is, \emph{homoscedastic location-scale mixtures} in which the matrix $\Sigma$ is shared across components and, unlike in previous work, estimated from data rather than fixed at the unknown value $\Sigma_0$. 
Concretely, we aim to establish posterior contraction rates of the form
\begin{equation}\label{eq:W1_contraction_rate}
    \Pi\left(G : \mathcal D(G, G_0)\geq K\varepsilon_n\mid X_1,\ldots,X_n\right) \to 0
\end{equation}
in $P_{G_0}^n$-probability for all large enough $K>0$ as $n\to\infty$, where $G_0=P_0\times\delta_{\Sigma_0}$ and $\mathcal D$ is a meaningful discrepancy between mixing measures $G$ and $G_0$. To achieve this, after verifying an $L^1$ contraction rate $\omega_n$ as in Equation~\eqref{eq:L1_contraction_rate}, a key step is to derive a lower-bound of the form $\|p_G-p_{G_0}\|_1\gtrsim \mathcal D(G,G_0)$.\footnote{Here and in the rest of the paper, $a(x) \lesssim b(x)$ and $a(x) \gtrsim b(x)$ denote inequalities up to a multiplicative constant, while $a(x) \asymp b(x)$ means $a(x) \lesssim b(x) \lesssim a(x)$. The variables $x$ that these constants are independent of will be clear from the context or explicitly spelled out case by case.} Following standard practice, we will design $\mathcal D$ to incorporate $W_1$ as a discrepancy between $P$ and $P_0$ (using the Euclidean ground metric on $x,y\in\mathcal X$) and the operator norm
$$\|S\|_{\mathrm{op}}:=\max_{v\in\mathbb R^d\,:\,\|v\|\leq 1}|v^\top S v|$$
to measure the distance between $\Sigma$ and $\Sigma_0$.

As mentioned above, virtually all existing work is limited to location mixtures only. In our notation, this is equivalent to fixing $\Sigma=\Sigma_0$---hence assuming it known---and restricting attention to the location mixing measure $P$, so that $\mathcal D(G,G_0)\equiv W_1(P,P_0)$ and the inequality one seeks is $\|p_G-p_{G_0}\|_1\gtrsim \varphi(W_1(P,P_0))$ for some increasing function $\varphi:[0,\infty)\to[0,\infty)$. Existing approaches, pioneered by \cite{Nguyen-13}, establish this kind of inequality by first convolving $P$ and $P_0$ with a suitable smoothing kernel $K_\delta$, and then exploiting the triangle inequality
\begin{equation}\label{eq:long_convolution}
    W_1(P,P_0) \leq W_1(P,P*K_\delta) + W_1(P*K_\delta, P_0*K_\delta) + W_1(P_0*K_\delta, P_0)
\end{equation}
to bound $W_1(P,P_0)$ by handling each term separately \citep{Nguyen-13,Gao2016Posterior,scricciolo2018bayes,rousseau2024wasserstein}.\footnote{Here $*$ denotes the convolution operator, that is $P*K:=\int_\mathcal X K(x-y)P(dy)$.} While this strategy represents the state of the art in the location mixture setting, it fails in the more general location-scale case, as the notion of convolution has no obvious extension to bound discrepancies of the form $W_1(P\times\delta_{\Sigma},P_0\times \delta_{\Sigma_0})$ where different scales $\Sigma$ and $\Sigma_0$ are involved.

To fill this gap, our first contribution is to develop new proof techniques that approach the problem from a functional analytic perspective. Our starting point is the observation that, for many location-scale mixtures of interest, an inequality of the form
\begin{equation}\label{eq:L1_Sigma_inequality}
    \|p_{G} - p_{G_0}\|_1\gtrsim \Psi(\|\Sigma-\Sigma_0\|_{\mathrm{op}}),
\end{equation}
for some strictly increasing function $\Psi:[0,\infty)\to[0,\infty)$, is available under mild parameter boundedness conditions. Moreover, the well-known Kantorovich--Rubinstein duality theorem \citep{Villani-09} allows to rewrite
\begin{equation*}
    W_1(P,P_0)=\sup_{h\in\mathrm{Lip}_1(\mathcal X)}\int_{\mathcal X} h(\theta)\big(P( d\theta) - P_0( d\theta)\big),
\end{equation*}
where $\mathrm{Lip}_1(\mathcal X):=\{h\in\mathbb R^\mathcal X \,:\, \forall \theta,\theta'\in\mathcal X,\; |h(\theta)-h(\theta')|\leq \|\theta-\theta'\|\}$ and the supremum is attained at some $h^*\in\mathrm{Lip}_1(\mathcal X)$. Our strategy, then, is to develop a functional approximation of $h^*$ that enables the use of Fourier transforms and the inversion theorem to analyze $W_1(P,P_0)$. In fact, since $\mathcal X$ will be assumed to be a compact subset of $\mathbb R^d$, leveraging Fourier inversion requires to first extend $h^*$ to a function $\widetilde h_{\mathcal X}^*$ defined on the whole $\mathbb R^d$, and then to approximate $\widetilde h_{\mathcal X}^*$ with another function $\widetilde h_{\mathcal X,\Lambda}^*$, where $\Lambda>0$ controls the approximation accuracy, that guarantees sufficient regularity. This constitutes a key technical step in our analysis.

With these approximations in place, we rewrite
\begin{align*}
    W_1(P,P_0) & = \int_{\mathbb R^d} \widetilde h_{\mathcal X,\Lambda}^* (\theta)\big(P( d\theta) - P_0( d\theta)\big)\\
    & \quad + \int_{\mathbb R^d}  \left(\widetilde h_{\mathcal X}^* (\theta)-\widetilde h_{\mathcal X,\Lambda}^* (\theta)\right)\big(P( d\theta) - P_0( d\theta)\big).
\end{align*}
The rest of our argument then proceeds in two steps:
\begin{enumerate}
    \item we show that the difference $\int_{\mathbb R^d}  \big(\widetilde h_{\mathcal X}^* (\theta)-\widetilde h_{\mathcal X,\Lambda}^* (\theta)\big)\big(P( d\theta) - P_0( d\theta)\big)$ represents a controlled approximation error of order $\Lambda^{-1}$, then
    
    \item we bound the integral $\int_{\mathbb R^d} \widetilde h_{\mathcal X,\Lambda}^* (\theta)\big(P( d\theta) - P_0( d\theta)\big)$, which is amenable to a Fourier-analytic treatment owing to the regularity of $\widetilde h_{\mathcal X,\Lambda}^*$. In particular, a careful application of the Fourier inversion theorem, together with Equation~\eqref{eq:L1_Sigma_inequality}, yields a general inequality of the form
\begin{equation}\label{eq:general_inequality}
    \|p_G-p_{G_0}\|_1 \gtrsim \varphi(W_1(P,P_0)) + \Psi(\| \Sigma - \Sigma_0\|_{\text{op}}),
\end{equation}
for a certain increasing function $\varphi$.
\end{enumerate}

While the specific form of the function $\Psi$ is a direct consequence of the requirement in Equation~\eqref{eq:L1_Sigma_inequality}, $\varphi$ can be characterized in terms of the properties of the density kernel $f$. In particular, $\varphi$ is related to the smoothness of $f$ measured by the rate of decay of its characteristic function. Our findings on this point, which we present in Section~\ref{sec:general_inequalities} and summarize in Table~\ref{table1}, are as follows:
\begin{enumerate}
    \item[(i)] In the \emph{super-smooth case of order $\alpha>0$}, that is, when the characteristic function $\Phi_\Sigma(\xi)$ of $f(\cdot\mid0,\Sigma)$ satisfies $\log \Phi_\Sigma(\xi)\asymp -\|\xi\|^\alpha$, we obtain
    $$\varphi(t)=\min\{\exp(-C/t^\alpha),(\Psi\circ \Xi^{-1})(\exp(-C/t^\alpha))\}$$
    for some constant $C>0$, where $\Xi:[0,\infty)\to[0,\infty)$ is another function (introduced later) determined by $f$. This result is dimension-free and applies to important kernel families such as the Gaussian and Cauchy ones.
    
    \item[(ii)] In the \emph{ordinary-smooth case of order $\beta>0$}, where $\log \Phi_\Sigma(\xi)\asymp -\log(1+\|\xi\|^\beta)$, we get that
    $$\varphi(t)=\min\{t^{d+\beta+1},(\Psi\circ \Xi^{-1})(C t^{d+p+1})\},$$
    for some constant $C>0$, where $p\geq 1$ is determined by $f$ and introduced later. In this case, the dependence on $d$ and $p$ reflects the mild regularity imposed by the ordinary-smoothness assumption and may be suboptimal for certain kernels such as the Laplace one, motivating the next result.
    
    \item[(iii)] In the presence of a \emph{partial differential equation (PDE) inversion condition of order $\beta\in\mathbb N$} (involving the existence of an appropriate differential operator $\mathcal{T}_\Sigma$ such that $\mathcal{T}_\Sigma p_G=P$ for all $G=P\times\delta_\Sigma\in\mathscr G$), we obtain that
    $$\varphi(t)=\min\{t^{\beta},\Psi(C t^{\beta})\},$$
    for some constant $C>0$. This novel PDE-based condition is notable because (a) it implies that the kernel $f$ is ordinary-smooth of order $\beta$ provided that $\beta$ is even (Proposition~\ref{pro:PDE_implies_ordinary}), and (b) it yields a dimension-free inequality, replacing the $\beta+d+1$ and $\beta+p+1$ exponents of scenario (ii) with the more favorable exponent $\beta$. We exploit this condition explicitly in the analysis of Laplace mixtures, while noting that it also applies to other important mixture classes, such as those arising from shifted exponential and fixed-shape Gamma kernels.
\end{enumerate}

\begin{table}
\caption{Lower-bounds on the $L^1$ distance under different kernel smoothness settings. See Assumption~\ref{ass2} for a definition of $\Xi$ and $p$.}
\label{table1}
\centering
\renewcommand{\arraystretch}{2} 
\begin{tabular}{@{}l@{\hspace{3em}}l@{}}
\toprule
\textbf{Kernel smoothness} & $\big\|p_G-p_{G_0}\big\|_1\gtrsim$\\
\toprule
\begin{tabular}{@{}l@{}}Super-smooth ($\alpha>0$) \\[-1.5ex] (Theorem~\ref{theorem:supersmooth_rate})\end{tabular} & 
$\begin{aligned}
  &\min\big\{\exp(-C/W_1(P,P_0)^\alpha), (\Psi\circ \Xi^{-1})(\exp(-C/W_1(P,P_0)^\alpha))\big\} \\
  &\quad + \Psi(\| \Sigma - \Sigma_0\|_{\text{op}})
\end{aligned}$ \\
\midrule

\begin{tabular}{@{}l@{}}Ordinary-smooth ($\beta>0$) \\[-1.5ex] (Theorem~\ref{theorem:ordinary_rate})\end{tabular} & 
$\begin{aligned}
  &\min\big\{W_1(P,P_0)^{d+\beta+1}, (\Psi\circ \Xi^{-1})(C W_1(P,P_0)^{d+p+1})\big\} \\
  &\quad + \Psi(\| \Sigma - \Sigma_0\|_{\text{op}})
\end{aligned}$ \\
\midrule

\begin{tabular}{@{}l@{}} PDE inversion ($\beta\in\mathbb N$) \\ [-1.5ex] (Theorem~\ref{theorem:sharpen_ordinary_rate})\end{tabular} & 
$\begin{aligned}
  &\min\big\{W_1(P,P_0)^{\beta}, \Psi(C W_1(P,P_0)^{\beta})\big\} \\
  &\quad + \Psi(\| \Sigma - \Sigma_0\|_{\text{op}})
\end{aligned}$ \\
\bottomrule
\end{tabular}
\end{table}

The different forms of smoothness indicated above are verified in important and popular choices of the kernel; in Section~\ref{sec:examples}, we verify these conditions for the Gaussian, Cauchy, and Laplace cases, with the resulting bounds summarized in Table~\ref{table2}. We also show that, in the Gaussian mixture setting, a sharper inequality is obtained with an isotropic scale  $\Sigma=\sigma^2I_d$ (where $I_d$ denotes the $d\times d$ identity matrix), deriving from the unique semi-group structure of Gaussian kernel families. In turn, when combined with $L^1$ contraction results of the form in Equation~\eqref{eq:L1_contraction_rate}, these bounds lead to posterior contraction rates for mixing measures and scale parameters in the corresponding infinite mixture models. A key feature of our approach is that, once an $L^1$ rate is available, the structure of the inequality in Equation~\eqref{eq:general_inequality} allows one to derive a $W_1$ rate for the mixing measure $P$ and a separate $\|\cdot\|_{\mathrm{op}}$ rate for the scale parameter $\Sigma$. This is made explicit in Propositions~\ref{pro:rate_multivariate_normal} and \ref{pro:rate_univariate_laplace} in Section~\ref{sec:posterior_rates}, which treat multivariate Gaussian and univariate Laplace mixtures under suitable Dirichlet process priors. For the Cauchy and multivariate Laplace kernels, instead, no explicit $L^1$ rate $\omega_n$ is currently available to the best of our knowledge, so that our inequalities yield rates for recovering $P_0$ and $\Sigma_0$ that are defined implicitly in terms of the unknown sequence $\omega_n$ (assuming the latter exists). We refer the reader to Section~\ref{sec:posterior_rates} for a detailed discussion of this issue for infinite Cauchy mixtures. 

\begin{table}
\caption{Lower-bounds on the $L^1$ distance for different kernel types.}
\label{table2}
\centering
\renewcommand{\arraystretch}{2} 
\begin{tabular}{@{}l@{\hspace{2em}}l@{}}
\toprule
\textbf{Kernel type} & $\big\|p_G-p_{G_0}\big\|_1\gtrsim$\\
\toprule
Gaussian (general scale, Theorem~\ref{thm:inequality_gaussian})& 
$\exp\left( - \frac{C\exp\big(\tilde C/W(P, P_0)^2\big)}{W_1(P, P_0)^2} \right)  + \exp\left(-M \frac{\log(1/\|\Sigma - \Sigma_0\|_{\mathrm{op}})}{\|\Sigma - \Sigma_0\|_{\mathrm{op}}} \right)$ \\
\midrule

Gaussian (isotropic scale, Theorem~\ref{thm:inequality_gaussian_isotropic})& 
$\exp\left( - \frac{C}{W_1(P, P_0)^2} \right)  + \exp\left(-M \frac{\log(1/\|\Sigma - \Sigma_0\|_{\mathrm{op}})}{\|\Sigma - \Sigma_0\|_{\mathrm{op}}} \right)$ \\
\midrule

Cauchy (Theorem~\ref{thm:inequality_cauchy}) & 
$\exp\left( - \frac{C}{W_1(P, P_0)} \right) + \|\Sigma - \Sigma_0\|_{\mathrm{op}}^2$ \\
\midrule

Laplace (Theorem~\ref{thm:inequality_laplace}) & 
$W_1(P, P_0)^{2} + \|\Sigma - \Sigma_0\|_{\mathrm{op}}$ \\
\bottomrule
\end{tabular}\vspace{1em}
\end{table}

Our posterior contraction results, summarized in Table~\ref{table3}, shed new light on two important points. First, across all of the kernel types we consider, there is significant heterogeneity in posterior contraction behavior between the location mixing measure $P$ and the scale parameter $\Sigma$, with the latter converging at substantially faster rates. Second, focusing on the kernel classes for which explicit $L^1$ rates are available, the comparison between Gaussian and Laplace mixtures shows that faster density estimation does not necessarily translate into faster recovery of the underlying parameters, as smoother kernels tend to blur more strongly the signal carried by $P_0$ and the variation in the scale parameter. This points to a practically relevant tradeoff: the choice of kernel should reflect whether the main objective is accurate density estimation or more structured tasks, such as clustering or parameter estimation, which depend directly on $P$ and $\Sigma$.

\begin{table}
\caption{$L^1$, $W_1$ and operator norm posterior contraction rates for Dirichlet process mixtures with different kernel types. For the Cauchy and multivariate Laplace cases, we express rates relative to an implicit $L^1$ rate $\omega_n$.}
\label{table3}
\centering
\renewcommand{\arraystretch}{2} 
\begin{tabular}{@{}l@{\hspace{1.5em}}@{\hspace{1.5em}}c@{\hspace{1.5em}}@{\hspace{1.5em}}c@{\hspace{1.5em}}@{\hspace{1.5em}}c@{}}
\toprule
\textbf{Kernel type} & $\|p_G-p_{G_0}\|_1$ & $W_1(P,P_0)$ & $\|\Sigma-\Sigma_0\|_{\text{op}}$ \\
\toprule
\begin{tabular}{@{}l@{}} Gaussian (general scale) \\[-1.5ex] (Proposition~\ref{pro:rate_multivariate_normal})\end{tabular} &
$n^{-1/2}(\log n)^{(d+1)/2}$ & 
$(\log\log n)^{-1/2}$ & 
$(\log n)^{-1}\log\log n$ \\
\midrule

\begin{tabular}{@{}l@{}} Gaussian (isotropic scale) \\[-1.5ex] (Proposition~\ref{pro:rate_multivariate_normal})\end{tabular} &
$n^{-1/2}(\log n)^{(d+1)/2}$ & 
$(\log n)^{-1/2}$ & 
$(\log n)^{-1}\log\log n$ \\
\midrule

\begin{tabular}{@{}l@{}}Cauchy \\[-1.5ex] (implicit in Theorem~\ref{thm:inequality_cauchy})\end{tabular} &
$\omega_{n}$ & 
$\left(\log(1/\omega_n)\right)^{-1}$ & 
$\omega_n^{1/2}$ \\
\midrule

\begin{tabular}{@{}l@{}} Laplace ($d=1$) \\[-1.5ex] (Proposition~\ref{pro:rate_univariate_laplace})\end{tabular} &
$n^{-1/4}(\log n)^{5/4}$ & 
$n^{-1/8}(\log n)^{5/8}$ & 
$n^{-1/4}(\log n)^{5/4}$ \\
\midrule

\begin{tabular}{@{}l@{}} Laplace ($d\geq2$)  \\[-1.5ex] (implicit in Theorem~\ref{thm:inequality_laplace})\end{tabular} &
$\omega_n$ & 
$\omega_n^{1/2}$ & 
$\omega_n$ \\
\bottomrule
\end{tabular}
\end{table}

To summarize, the main contributions of the article are as follows.
\begin{itemize}
    \item Using the dual formulation of the $W_1$ metric, together with functional-analytic approximation techniques, we develop new tools to lower-bound the $L^1$ distance between homoscedastic location-scale mixture densities in terms of the $W_1$ distance between the underlying location mixing measures and the operator norm distance between the underlying scales. The resulting inequalities depend on the smoothness of the mixture kernel and, in the ordinary-smooth case, can be significantly sharpened by verifying a novel and widely applicable PDE inversion condition;
    \item We specialize these abstract inequalities to common kernel choices, obtaining explicit bounds for multivariate Gaussian, Cauchy, and Laplace mixtures;
    \item Combining these results with existing $L^1$ contraction theory, for the first time in the literature we derive posterior convergence rates for mixing measures and scale parameters in infinite homoscedastic location-scale mixture models with Dirichlet process priors, focusing on the Gaussian and Laplace cases.
\end{itemize}

The rest of the article is organized as follows. Section~\ref{sec:notation} introduces notation and collects a few preliminary results. Section~\ref{sec:general_inequalities} develops abstract inequalities of the form in Equation~\eqref{eq:general_inequality} under different smoothness assumptions on the mixture kernel. Section~\ref{sec:examples} specializes these inequalities to commonly used kernels. Section~\ref{sec:posterior_rates} builds on the previous sections to derive posterior contraction rates for homoscedastic location-scale infinite mixture models with Dirichlet process priors. Section~\ref{sec:discussion} concludes. Complete proofs are provided in the appendices.

\section{Notation and preliminaries}\label{sec:notation}

Throughout the article, we assume that a sample of $\mathbb R^d$-valued observations $X_1,\ldots,X_n$ is available and drawn i.i.d. from an underlying mixture distribution $P_{G_0}$ (for $G_0=P_0\times\delta_{\Sigma_0}$) with density
\[
p_{G_0}(x)=\int_{\mathcal X} f(x\mid \theta,\Sigma_0)\,P_0(d\theta),\quad x\in\mathbb R^d,
\]
with respect to the Lebesgue measure  on $\mathbb R^d$ (endowed with its natural Borel $\sigma$-algebra $\mathscr B(\mathbb R^d)$). The specific form of the kernel $f$ will vary throughout, with
$$\Phi_\Sigma(\xi):=\int_{\mathbb R^d} e^{i\xi^\top x} f(x\mid 0,\Sigma)\,dx,\quad \xi\in\mathbb R^d,$$
representing the associated characteristic function; with a slight abuse of notation, we also write $\Phi_\mu(\xi):=\int_{\mathbb R^d}e^{i\xi^\top x}\mu(dx)$ for any probability measure $\mu$ on $\mathbb R^d$. Relatedly, for any integrable function  $h:\mathbb R^d\to \mathbb R$, 
 we denote its Fourier transform by \[
 \mathcal F[h](\xi):=\int_{\mathbb R^d}e^{-i\xi^\top x}h(x)dx, \quad \xi\in\mathbb R^d.
 \]

Prior and posterior distributions, which we have already defined and denoted respectively by $\Pi(\cdot)$ and $\Pi(\cdot\mid X_1,\ldots,X_n)$, are formally understood as probability measures on the space of absolutely continuous probabilities on $\mathbb R^d$, endowed with the topology of weak convergence and the corresponding Borel $\sigma$-algebra. However, since the priors we consider are supported on specific subclasses of mixture distributions, it is more convenient to view both the prior and the posterior as measures on the space of mixing measures generating such mixtures.

In particular, the priors $\Pi$ we consider concentrate on homoscedastic location-scale mixture densities arising from mixing measures of the form $G=P\times \delta_\Sigma$. For all such $G$ in the support of the prior (including the unknown $G_0$), we work under the standing assumption that there exist constants $R,\lambda_{\min},\lambda_{\max}>0$, with $\lambda_{\min}<\lambda_{\max}$, and a measurable set $\mathcal X\subseteq \bar B(0,R):=\{x\in\mathbb R^d : \|x\|\leq R\}$,\footnote{Denote by $B(0,R)$ the corresponding open ball.} such that
\[
\mathrm{supp}(P)\subseteq \mathcal X, 
\qquad 
\Sigma\in\mathbb S_d^+(\lambda_{\min},\lambda_{\max}),
\]
where $\mathrm{supp}(P)$ denotes the support of $P$, and $\mathbb S_d^+(\lambda_{\min},\lambda_{\max})$ is the set of symmetric positive-definite matrices with eigenvalues bounded below by $\lambda_{\min}$ and above by $\lambda_{\max}$. Therefore, the parameter space of interest is
$$\mathscr G:=\{G=P\times \delta_\Sigma \,:\, \mathrm{supp}(P)\subseteq \mathcal X, \,\Sigma\in\mathbb S_d^+(\lambda_{\min},\lambda_{\max})\},$$
and we write $\mathscr P:=\{p_G: G\in\mathscr G\}$. Throughout the text, we assume that the prior on $\mathscr G$ is of product form across its coordinates $P$ and $\Sigma$.

Our main interest is in the popular modeling strategy in which $P$ is assigned a Dirichlet process prior \citep{Ferguson}, written $P\sim \mathrm{DP}(a,H)$, where $a>0$ is a concentration parameter and $H$ is a diffuse base probability measure on $\mathbb R^d$; the scale $\Sigma$, instead, is modeled independently of $P$ with a prior having full support on $\mathbb S_d^+(\lambda_{\min},\lambda_{\max})$. In particular, for any finite measurable partition $A_1,\ldots,A_k$ of $\mathbb R^d$, this prior structure implies that the vector $(P(A_1),\ldots,P(A_k))$ has a Dirichlet distribution with parameters $(aH(A_1),\ldots,aH(A_k))$, and if $\mathrm{supp}(H)=\mathcal X$, then the support of $\Pi$ coincides with the set of mixing measures $G=P\times \delta_\Sigma$ with $\Sigma\in\mathbb S_d^+(\lambda_{\min},\lambda_{\max})$ and $\mathrm{supp}(P)\subseteq \mathcal X$ \citep{majumdar1992topological}.

A key feature of the Dirichlet process is the implied stick-breaking representation \citep{Sethuraman} for $P\sim\mathrm{DP}(a,H)$:
\[
P=\sum_{j\in\mathbb N} w_j\delta_{\theta_j},
\qquad 
w_1=B_1,\quad w_j=B_j\prod_{\ell=1}^{j-1}(1-B_\ell),\quad j\geq 2,
\]
where $B_j\overset{\text{i.i.d.}}{\sim}\mathrm{Beta}(1,a)$ and $\theta_j\overset{\text{i.i.d.}}{\sim}H$. This yields the generative mechanism
\[
X_1,\ldots,X_n \mid G \overset{\text{i.i.d.}}{\sim} p_G, 
\qquad 
p_G=\sum_{j\in\mathbb N} w_j f(\cdot\mid \theta_j,\Sigma),
\]
which can be interpreted as a two-stage procedure (once $G$ is fixed): first, latent mixture component assignments $c_i\in\mathbb N$ are drawn according to the discrete measure $\sum_{j\in\mathbb N}w_j\delta_j$ derived from $P$, and then $X_i$ is drawn from $f(\cdot\mid \theta_{c_i},\Sigma)$. This induces a natural partition of the data, where $X_i$ and $X_j$ are assigned to the same cluster whenever $c_i=c_j$. Clearly, this rich generative structure is governed by the mixing measure $G$, which is the main motivation driving the study of posterior contraction towards the unknown $G_0$.

As already discussed, a natural strategy is to deduce contraction around $G_0$ from $L^1$ contraction around $p_{G_0}$, leveraging the extensive literature on Bayesian nonparametric density estimation. To make this precise, we recall a general result (an immediate adaptation of Theorem~2.1 in \cite{Ghosal-2000}) giving sufficient conditions for $L^1$ contraction. Define
\[
V_1(p_{G_0},p_G):=\int_{\mathbb R^d}\log\left(\frac{p_{G_0}(x)}{p_G(x)}\right)p_{G_0}(x)\,dx,
\]
\[
V_2(p_{G_0},p_G):=\int_{\mathbb R^d}\left(\log\frac{p_{G_0}(x)}{p_G(x)}\right)^2 p_{G_0}(x)\,dx,
\]
and, for any $\omega > 0$ and $\mathscr A\subseteq \mathscr P$, let $D(\omega,\mathscr A,\|\cdot\|_1)$ denote the $L^1$ $\omega$-packing number of $\mathscr A$, i.e., the largest number of densities in $\mathscr A$ that are at least $\omega$ apart from each other in $L^1$ distance.

\begin{theorem}\label{thm:G-vdV}
Assume that, for a positive sequence $(\omega_n)_{n\in\mathbb N}$ with $\lim_{n\to\infty}\omega_n= 0$ and $\lim_{n\to\infty}n\omega_n^2=\infty$, there exist sets $\mathscr G_n\subseteq \mathscr G$ and a constant $C>0$ such that
\begin{align*}
\log D(\omega_n,\{p_G:G\in\mathscr G_n\},\|\cdot\|_1) &\leq n\omega_n^2,\\
\Pi(\mathscr G\setminus \mathscr G_n) &\leq \exp(-(C+4)n\omega_n^2),\\
\Pi\!\left(G\;:\; V_1(p_{G_0},p_G)\leq \omega_n^2,\;V_2(p_{G_0},p_G)\leq \omega_n^2\right) &\geq \exp(-Cn\omega_n^2)
\end{align*}
for all large enough $n\in\mathbb N$. Then
\[
\Pi(G\;:\;\|p_G-p_{G_0}\|_1\geq K\omega_n\mid X_1,\ldots,X_n)\to0
\]
in $P_{G_0}^n$-probability for all large enough $K>0$ as $n\to\infty$.
\end{theorem}

The conditions  of the theorem are standard: they ensure an $L^1$ rate by requiring that (i) complex densities receive negligible prior mass, and (ii) enough mass is placed near the true density in the Kullback--Leibler sense. On the other hand, as implied by Lemma 1 in \cite{Nguyen-13}, under mild continuity assumptions on $(\theta,\Sigma)\mapsto f(\cdot\mid\theta,\Sigma)$ one has
\[
\|p_G-p_{G'}\|_1 \lesssim W_1(G,G')
\]
for discrete mixing measures $G,G'\in\mathscr G$.\footnote{Using $\rho((x,\Sigma), (x',\Sigma'))=\|x-x'\| + \|\Sigma - \Sigma'\|_{\mathrm{op}}$ as a ground metric on the product space $\Theta=\mathcal X\times\mathbb S_d^+(\lambda_{\min},\lambda_{\max})$.} Thus, convergence of mixing measures implies convergence of densities, but the converse is not necessarily true. Establishing a reverse inequality
\[
\|p_G-p_{G'}\|_1 \gtrsim \mathcal D(G,G'),
\]
of the form in Equation~\eqref{eq:general_inequality} is however essential to deduce posterior contraction around $G_0$ from $L^1$ results such as Theorem~\ref{thm:G-vdV}, and identifying conditions under which such inverse bounds hold has been a central question since \cite{Nguyen-13}. The main goal of this article is to address this problem in the homoscedastic location-scale setting where $G=P\times\delta_\Sigma$ and $G'=P'\times\delta_{\Sigma'}$. We do so by first deriving general inverse bounds under different smoothness assumptions on $f$ (in Section~\ref{sec:general_inequalities}), while in Sections~\ref{sec:examples} and \ref{sec:posterior_rates} we specialize these results to Gaussian, Cauchy, and Laplace kernels, and combine them with existing $L^1$ contraction results (typically obtained via arguments related to Theorem~\ref{thm:G-vdV}) to derive corresponding $W_1$ posterior contraction rates.

\section{Inverse Bounds for Mixing Measure Discrepancies}\label{sec:general_inequalities}

In this section, we derive lower-bounds for the $L^1$ distance between mixture densities $p_G$ and $p_{G'}$, where $G=P\times\delta_{\Sigma}$ and $G'=P'\times\delta_{\Sigma'}$, in terms of suitable functions of the $W_1$ distance between the mixing measures $P$ and $P'$, and the operator norm distance between the scale matrices $\Sigma$ and $\Sigma'$.

We begin by stating two assumptions that will be used throughout.

\begin{assumption}\label{ass1}
    There exists a strictly increasing function $\Psi:[0,\infty)\to[0,\infty)$ such that, for any $G=P\times \delta_\Sigma$ and  $G'=P'\times \delta_{\Sigma'}$ in $\mathscr G$,
    \begin{align*}
        \|p_G - p_{G'}\|_1 \geq C \Psi(\|\Sigma - \Sigma'\|_{\mathrm{op}}),
    \end{align*}
    for a positive quantity $C$ that may depend on $R$, $\lambda_{\min}$, $\lambda_{\max}$, and $d$.
\end{assumption}
The specific form of $\Psi$ depends on the choice of kernel, but the requirement itself is mild: since the location space $\mathcal X$ is compact, a small $L^1$ distance between densities is expected to imply closeness of the corresponding scale parameters. This will be verified for several important kernel families in Section~\ref{sec:examples}.

\begin{assumption}\label{ass2}
    There exist an exponent $p \geq 1$ and a strictly increasing function $\Xi:[0,\infty)\to [0,\infty)$ such that, for all $\xi \in \mathbb{R}^d$ and $\Sigma, \Sigma' \in \mathbb{S}^{+}_{d}(\lambda_{\min}, \lambda_{\max})$,
    $$
    |\Phi_{\Sigma'}(\xi) - \Phi_\Sigma(\xi)| \leq C \Xi(\|\Sigma - \Sigma'\|_{\mathrm{op}}) \|\xi\|^p \max\{ |\Phi_{\Sigma'}(\xi)|,|\Phi_\Sigma(\xi)| \},
    $$
   for a positive quantity $C$ that may depend on $R$, $\lambda_{\min}$, $\lambda_{\max}$, and $d$.  
\end{assumption}

Assumption~\ref{ass2} requires the characteristic function associated with the kernel $f(\cdot\mid0,\Sigma)$ to vary in a controlled manner with respect to $\Sigma$. In particular, the dependence is governed by $\Xi$, with a degree-$p$ polynomial modulation in $\|\xi\|$ and a linear one in the magnitude of the characteristic function itself.

We are now in a position to state our main results. Towards that goal, we will consider three settings separately: when the kernel $f$ is super-smooth, when it is ordinary smooth, and when the mixture satisfies a certain PDE inversion condition. 

\subsection{Super-smooth kernel}

We first consider the case in which the mixture kernel $f$ is super-smooth. For the next definition, recall that $\Phi_\Sigma(\xi)$ denotes the characteristic function, evaluated at $\xi$, associated with the density $f(\cdot\mid 0,\Sigma)$.
\begin{definition}
    The density kernel $f$ is \emph{super-smooth of order $\alpha>0$} if
    $$\exp(-c_1 \|\xi\|^\alpha) \le |\Phi_\Sigma(\xi)| \le \exp(-c_2 \|\xi\|^\alpha)$$
    for all $\xi\in\mathbb R^d$ and $\Sigma\in\mathbb{S}^{+}_{d}(\lambda_{\min},\lambda_{\max})$, where $c_{1}$ and $c_{2}$ are constants depending only on $\lambda_{\min}$, $\lambda_{\max}$, and $d$.
\end{definition}
That is, super-smoothness requires an exponential decay of $\Phi_\Sigma(\xi)$ as a function of $\| \xi \|$. This condition is satisfied by several commonly used kernels, including the Gaussian and Cauchy kernels considered in Section~\ref{sec:examples}. Combined with Assumptions~\ref{ass1} and~\ref{ass2}, super-smoothness yields the following inverse bound.

\begin{theorem}\label{theorem:supersmooth_rate}
    Under Assumptions~\ref{ass1} and~\ref{ass2}, if the kernel $f$ is super-smooth of order $\alpha$, there exist positive quantities $C,\widetilde C, C'$, depending on $R$, $\lambda_{\min}$, $\lambda_{\max}$, $p$, and $d$, such that 
    \begin{align}
        \|p_G &- p_{G'}\|_1 \geq C' \bigg[ \Psi(\|\Sigma - \Sigma'\|_{\mathrm{op}})\nonumber\\
        & +\min\left\{ \exp\left( - \frac{C}{W_1(P, P')^\alpha} \right), \;(\Psi\circ \Xi^{-1})\left(\widetilde C\exp\left( - \frac{C}{W_1(P, P')^\alpha} \right)\right) \right\}\bigg] \label{eq:lower_bound_super_smooth}
    \end{align}
    for any $G=P\times \delta_\Sigma$ and  $G'=P'\times \delta_{\Sigma'}$ in $\mathscr G$.
\end{theorem}

The bound in Theorem~\ref{theorem:supersmooth_rate} depends on $\|\Sigma - \Sigma'\|_{\mathrm{op}}$ only through the function $\Psi$, which naturally aligns with Assumption~\ref{ass1}. In contrast, the dependence on $W_1(P,P')\equiv w$ is more complex, entering through terms of the form $\exp(-C/w^\alpha)$ that are possibly mediated by the composition $\Psi\circ \Xi^{-1}$. Furthermore, the lower-bound deteriorates as the smoothness level $\alpha$ increases. From a deconvolution perspective, this is intuitive: a smoother kernel makes the recovery of the underlying mixing measure statistically more challenging, as reflected by a looser bound.

Beyond the lower-bound results and their implications for statistical estimation of mixture models, a central contribution of this article is the development of new techniques to derive such bounds, based on functional approximations in the dual formulation of $W_1$ and tools from Fourier analysis and PDEs. Therefore, we now outline the main steps of the argument, deferring the full proof to the appendices.

\begin{proof}[Proof sketch of Theorem~\ref{theorem:supersmooth_rate}]
The proof consists of five steps. We first use the Kantorovich--Rubinstein duality theorem to represent $W_1(P,P')$ as an integral of an optimal $1$-Lipschitz function $h^*$, which we extend to all of $\mathbb{R}^d$ via a McShane--Whitney argument. We then construct a spectrally truncated approximation of $h^*$ using a Jackson kernel, controlling the resulting approximation error. Fourier inversion and Fubini's theorem allow us to express the integral in terms of characteristic functions, which we split into two contributions via the triangle inequality. Each contribution is then bounded using the super-smoothness assumption and Assumptions~\ref{ass1} and \ref{ass2}. Optimizing over the truncation parameter and rearranging finally yields the claimed inequality. Below is a more detailed discussion of each step.

\emph{Step 1: McShane--Whitney extension.\hspace{1em}} By the Kantorovich--Rubinstein duality theorem \citep{Villani-09}, for every $P$ and $P'$ there exists $h^* \in \mathrm{Lip}_1(\mathcal{X})$ attaining
\[
W_1(P, P') = \int_{\mathcal{X}} h^*(\theta)\, (P(d\theta) - P'(d\theta)).
\]
Replacing $h^*$ with $h^*(\cdot) - h^*(0)$ leaves the integral unchanged and ensures $|h^*(x)| \leq \|x\|$ for all $x \in \mathcal{X}$. Our goal is to analyze $W_1(P,P')$ via Fourier inversion, which requires extending $h^*$ to all of $\mathbb{R}^d$ and working with a sufficiently regular approximation thereof. We first extend $h^*$ to a globally $1$-Lipschitz continuous function on $\mathbb{R}^d$ via the McShane--Whitney infimal convolution \citep{heinonen2001lectures,McShane1934,Whitney1934}
\[
\bar{h}_{\mathcal{X}}^{*}(x) := \inf_{y \in \mathcal{X}} \Big\{ h^*(y) + \|x - y\| \Big\},
\]
and then multiply it by a compactly supported and infinitely differentiable cutoff $\eta \in C_c^\infty(\mathbb{R}^d)$ satisfying $\eta(x)=1$ for $x\in\bar{B}(0,R)$ and $\eta(x)=0$ for $x\not\in B(0,2R)$ \citep{Lee2013SmoothManifolds}. The resulting function $\widetilde{h}^*_{\mathcal{X}}$ belongs to $L^1(\mathbb{R}^d)$, is globally $(1 + 2RM_\eta)$-Lipschitz continuous (where $M_\eta := \sup_{x\in\mathbb{R}^d}\|\nabla\eta(x)\| < \infty$), and preserves the property
\[
W_1(P, P')=\int_{\mathbb{R}^d} \widetilde{h}^*_{\mathcal{X}}(\theta)\, (P(d\theta) - P'(d\theta)).
\]

\emph{Step 2: Spectral truncation.\hspace{1em}} To apply Fourier inversion, we convolve $\widetilde{h}^*_{\mathcal{X}}$ with a rescaled Jackson kernel $K_\Lambda(x) := \Lambda^d K(\Lambda x)$, whose Fourier transform is supported on $\bar{B}(0,\Lambda)$, to obtain $\widetilde{h}^*_{\mathcal{X},\Lambda} := \widetilde{h}^*_{\mathcal{X}} * K_\Lambda$ \citep{devore1993constructive,butzer1971fourier}. The Lipschitz continuity of $\widetilde{h}^*_{\mathcal{X}}$ implies the approximation bound
\[
\|\widetilde{h}^*_{\mathcal{X},\Lambda} - \widetilde{h}^*_{\mathcal{X}}\|_\infty := \sup_{x\in\mathbb{R}^d}\big|\widetilde{h}^*_{\mathcal{X},\Lambda}(x) - \widetilde{h}^*_{\mathcal{X}}(x)\big| \leq (1+2RM_\eta)M_K/\Lambda,
\]
where $M_K := \int_{\mathbb{R}^d}\|x\|K(x)\,dx < \infty$. Combined with $\sup_{A\in\mathscr{B}(\mathbb{R}^d)}|P(A) - P'(A)| \leq 2$, this gives
\begin{equation}
\label{eq:sketch_split}
W_1(P, P') \;\leq\; \frac{C_0}{\Lambda} + \int_{\mathbb{R}^d} \widetilde{h}^*_{\mathcal{X},\Lambda}(\theta)\, (P(d\theta) - P'(d\theta))
\end{equation}
for a constant $C_0 > 0$ depending only on $R$.

\emph{Step 3: Fourier inversion.\hspace{1em}} Since $\widetilde{h}^*_{\mathcal{X},\Lambda} \in L^1(\mathbb{R}^d)$ and $\mathcal{F}[\widetilde{h}^*_{\mathcal{X},\Lambda}]$ is supported in $\bar{B}(0,\Lambda)$ and bounded, the Fourier inversion and Fubini's theorems yield
\[
\int_{\mathbb{R}^d} \widetilde{h}^*_{\mathcal{X},\Lambda}(\theta)\, (P(d\theta) - P'(d\theta)) = \frac{1}{(2\pi)^d} \int_{\bar{B}(0,\Lambda)} \mathcal{F}[\widetilde{h}^*_{\mathcal{X},\Lambda}](\xi)\,\big(\Phi_P(\xi) - \Phi_{P'}(\xi)\big)\, d\xi.
\]
Writing $\Phi_P(\xi) - \Phi_{P'}(\xi) = \Phi_{p_G}(\xi)/\Phi_\Sigma(\xi) - \Phi_{p_{G'}}(\xi)/\Phi_{\Sigma'}(\xi)$ and applying the triangle inequality, the right-hand side is bounded above by $C_R$ times
\begin{equation}
\label{eq:sketch_two_terms}
\int_{\bar{B}(0,\Lambda)} \frac{|\Phi_{p_G}(\xi) - \Phi_{p_{G'}}(\xi)|}{|\Phi_\Sigma(\xi)|}\, d\xi \;+\; \int_{\bar{B}(0,\Lambda)} \frac{|\Phi_{\Sigma'}(\xi) - \Phi_\Sigma(\xi)|}{|\Phi_\Sigma(\xi)|}\, d\xi,
\end{equation}
where $C_R > 0$ depends only on $R$ and $d$.

\emph{Step 4: Bounding via super-smoothness.\hspace{1em}} For the first term in Equation~\eqref{eq:sketch_two_terms}, we use $|\Phi_{p_G}(\xi) - \Phi_{p_{G'}}(\xi)| \leq \|p_G - p_{G'}\|_1$ and the lower-bound $|\Phi_\Sigma(\xi)| \geq \exp(-c_1\|\xi\|^\alpha)$ from the super-smoothness assumption to get
\[
\int_{\bar{B}(0,\Lambda)} \frac{|\Phi_{p_G}(\xi) - \Phi_{p_{G'}}(\xi)|}{|\Phi_\Sigma(\xi)|}\,d\xi \;\lesssim\; \|p_G - p_{G'}\|_1\cdot \Lambda^d \exp(c_1 \Lambda^\alpha).
\]
For the second term, Assumption~\ref{ass2} and the super-smoothness bounds together yield
\[
\int_{\bar{B}(0,\Lambda)} \frac{|\Phi_{\Sigma'}(\xi) - \Phi_\Sigma(\xi)|}{|\Phi_\Sigma(\xi)|}\,d\xi \;\lesssim\; \Xi(\|\Sigma-\Sigma'\|_{\mathrm{op}})\cdot \Lambda^{d+p}\exp((c_1-c_2)\Lambda^\alpha).
\]

\emph{Step 5: Conclusion.\hspace{1em}} Substituting these bounds into Equation~\eqref{eq:sketch_split} and using Assumption~\ref{ass1} to replace $\Xi(\|\Sigma - \Sigma'\|_{\mathrm{op}})$ with $(\Xi \circ \Psi^{-1})(\|p_G - p_{G'}\|_1)$, we obtain
\[
W_1(P,P') \;\lesssim\; \frac{C_0}{\Lambda} + \|p_G - p_{G'}\|_1 \cdot \Lambda^d e^{c_1\Lambda^\alpha} + (\Xi\circ\Psi^{-1})(\|p_G-p_{G'}\|_1)\cdot \Lambda^{d+p}e^{(c_1-c_2)\Lambda^\alpha}.
\]
Setting $\Lambda \asymp W_1(P,P')^{-1}$ and applying the elementary bound $t^{-k} \lesssim \exp(t^{-\alpha})$ for any $k, \alpha > 0$ and $t \in (0,1]$, we rearrange to obtain that either
\[
\|p_G - p_{G'}\|_1 \;\gtrsim\; \exp\!\left(-\frac{C}{W_1(P,P')^\alpha}\right),
\]
or
\[
\|p_G - p_{G'}\|_1 \;\gtrsim\; (\Psi\circ \Xi^{-1})\!\left(\exp\!\left(-\frac{C}{W_1(P,P')^\alpha}\right)\right),
\]
which together with Assumption~\ref{ass1} give Equation~\eqref{eq:lower_bound_super_smooth}.
\end{proof}

\subsection{Ordinary-smooth kernel}

The second setting we consider involves an ordinary-smooth kernel $f$, which is formalized as follows.

\begin{definition}
    The density kernel $f$ is \emph{ordinary-smooth of order $\beta>0$} if
    $$c_{1} (1 + \|\xi\|^\beta)^{-1} \le |\Phi_\Sigma(\xi)| \le c_{2}(1 + \|\xi\|^\beta)^{-1}$$
    for all $\xi \in \mathbb{R}^{d}$ and $\Sigma\in\mathbb{S}^{+}_{d}(\lambda_{\min},\lambda_{\max})$, where $c_{1}$ and $c_{2}$ are constants depending on $\lambda_{\min},\lambda_{\max}$, and $d$. 
\end{definition}

Unlike in the super-smooth case, ordinary-smooth kernels are characterized by a polynomial decay of the characteristic function, corresponding to a substantially lower degree of smoothness. This leads to the following inverse bound.

\begin{theorem}\label{theorem:ordinary_rate}
    Under Assumptions~\ref{ass1}~and~\ref{ass2}, if the kernel $f$ is ordinary-smooth of order $\beta$, there exist positive constants $C$ and $C'$, depending on $R$, $\lambda_{\min}$, $\lambda_{\max}$, $p$, and $d$, such that
    \begin{align}
     \|p_G &- p_{G'}\|_1 \geq C' \Big[\Psi(\|\Sigma - \Sigma'\|_{\mathrm{op}})\nonumber\\
     & +\min\Big( W_1(P, P')^{d+\beta+1}, \;(\Psi\circ \Xi^{-1})\big(C W_1(P, P')^{d+p+1}\big) \Big)\Big] \label{eq:lower_bound_ordinary_smooth}
     \end{align}
     for any $G=P\times \delta_\Sigma$ and $G'=P'\times \delta_{\Sigma'}$ in $\mathscr G$.
\end{theorem}

The structure of the bound in Theorem~\ref{theorem:ordinary_rate} is similar to its super-smooth counterpart in Theorem~\ref{theorem:supersmooth_rate}, in terms of the role played by the functions $\Psi$ and $\Xi$ in mediating the dependence on $\|\Sigma - \Sigma'\|_{\mathrm{op}}$ and $W_1(P, P')$. However, the $W_1$ component of the bound is significantly improved under the ordinary-smooth assumption: the inverse-exponential dependence in the super-smooth case is replaced by a polynomial one, with $\beta$ worsening the rate as it increases. A drawback of Theorem~\ref{theorem:ordinary_rate}, however, is the explicit dependence on the dimension $d$, which may inflate the polynomial degree. Intuitively, this reflects the fact that the ordinary-smooth assumption imposes only mild conditions on the tail behavior of the characteristic function of $f$, whereas in many cases of interest (e.g., Laplace mixtures, as shown in Section~\ref{sec:examples}) the mixture possesses useful additional structure. This motivates, in the next subsection, turning to a more refined condition to obtain dimension-free bounds.

\begin{proof}[Proof sketch of Theorem~\ref{theorem:ordinary_rate}]
The proof follows the same structure as that of Theorem~\ref{theorem:supersmooth_rate}, with Steps~1--3 carrying over verbatim. The key differences arise in Step~4, where the polynomial decay of the characteristic function under ordinary smoothness leads to polynomial rather than exponential bounds, and in Step~5, where the choice of truncation parameter and the subsequent rearrangement differ accordingly. We therefore focus on these two steps.

\emph{Step 4': Bounding via ordinary smoothness.\hspace{1em}} For the first term in Equation~\eqref{eq:sketch_two_terms}, we use $|\Phi_{p_G}(\xi) - \Phi_{p_{G'}}(\xi)| \leq \|p_G - p_{G'}\|_1$ and the lower-bound $|\Phi_\Sigma(\xi)| \geq (1 + c_1\|\xi\|^\beta)^{-1}$ from the ordinary-smoothness assumption to get
\[
\int_{\bar{B}(0,\Lambda)} \frac{|\Phi_{p_G}(\xi) - \Phi_{p_{G'}}(\xi)|}{|\Phi_\Sigma(\xi)|}\,d\xi \;\lesssim\; \|p_G - p_{G'}\|_1 \cdot \Lambda^{d+\beta}.
\]
For the second term, Assumption~\ref{ass2} gives
$$|\Phi_{\Sigma'}(\xi) - \Phi_\Sigma(\xi)| \lesssim \Xi(\|\Sigma - \Sigma'\|_{\mathrm{op}})\|\xi\|^p\max(|\Phi_\Sigma(\xi)|, |\Phi_{\Sigma'}(\xi)|).$$
Under ordinary smoothness, the ratio $\max(|\Phi_\Sigma(\xi)|, |\Phi_{\Sigma'}(\xi)|)/|\Phi_\Sigma(\xi)|$ is bounded above by $c_1/c_2$, uniformly in $\xi$, so that
\[
\int_{\bar{B}(0,\Lambda)} \frac{|\Phi_{\Sigma'}(\xi) - \Phi_\Sigma(\xi)|}{|\Phi_\Sigma(\xi)|}\,d\xi \;\lesssim\; \Xi(\|\Sigma-\Sigma'\|_{\mathrm{op}})\cdot \Lambda^{d+p}.
\]
Note that, in contrast to the super-smooth case, both bounds are now purely polynomial in $\Lambda$.

\emph{Step 5': Conclusion.\hspace{1em}} Substituting these bounds into Equation~\eqref{eq:sketch_split} and using Assumption~\ref{ass1} to replace $\Xi(\|\Sigma - \Sigma'\|_{\mathrm{op}})$ with $(\Xi \circ \Psi^{-1})(\|p_G - p_{G'}\|_1)$, we obtain
\[
W_1(P,P') \;\lesssim\; \frac{C_0}{\Lambda} + \|p_G - p_{G'}\|_1 \cdot \Lambda^{d+\beta} + (\Xi\circ\Psi^{-1})(\|p_G-p_{G'}\|_1)\cdot \Lambda^{d+p}.
\]
Setting $\Lambda \asymp W_1(P,P')^{-1}$ and rearranging, we obtain that either
\[
\|p_G - p_{G'}\|_1 \;\gtrsim\; W_1(P,P')^{d+\beta+1},
\]
or
\[
\|p_G - p_{G'}\|_1 \;\gtrsim\; (\Psi\circ\Xi^{-1})\!\left(C W_1(P,P')^{d+p+1}\right).
\]
These two cases, combined with Assumption~\ref{ass1}, give Equation~\eqref{eq:lower_bound_ordinary_smooth}. 
\end{proof}

\subsection{Kernel with PDE inversion} 

To remedy the dependence on $d$ in the inverse bound based on ordinary-smooth kernels, we now introduce the following PDE inversion assumption. Before proceeding, as we are about to work with differential operators defined in the weak functional-analytic sense of distributions, we refer the reader to, e.g., Chapter 9 of \cite{folland1999real} for a concise introduction to the topic.

\begin{definition}\label{def:PDE_inversion}
    For any $\nu = (\nu_1, \dots, \nu_d) \in \mathbb{N}^d$, set $|\nu| := \sum_{j=1}^d \nu_j$ and define the distributional derivative operator $\partial^\nu := \frac{\partial^{|\nu|}}{\partial x_1^{\nu_1} \dots \partial x_d^{\nu_d}}$. Then the family $\mathscr P$ satisfies a \emph{PDE inversion condition of order $\beta \in \mathbb{N}$} if there exists a linear differential operator
    $$ \mathcal{T}_\Sigma = \sum_{\nu\in\mathbb N^d \,:\, 0 \le |\nu| \le \beta} c_\nu(\Sigma) \partial^\nu$$
    such that
    \begin{enumerate}
        \item there exists $\nu\in\mathbb N^d$ with $|\nu|=\beta$ such that $c_{\nu}(\Sigma)\neq 0$,
        \item there exists a constant $L>0$, independent of $\nu$, such that the coefficients $\Sigma\mapsto c_\nu(\Sigma)$ are $L$-Lipschitz continuous from $\mathbb{S}^{+}_{d}(\lambda_{\min},\lambda_{\max})$ to $\mathbb R$, and
        \item the equation $\mathcal{T}_\Sigma p_{G} = P$ holds in the sense of distributions on $\mathbb{R}^d$. That is, for every test function $\phi \in C_c^\infty(\mathbb{R}^d)$, we have
        $$ \langle \mathcal{T}_\Sigma^* \phi, p_G \rangle = \int_{\mathbb{R}^d} \phi(x) \, P(dx), $$
        where $\mathcal{T}_\Sigma^* := \sum_{\nu\in\mathbb N^d \,:\,0 \le |\nu| \le \beta} (-1)^{|\nu|} c_\nu(\Sigma) \partial^\nu$ is the formal adjoint of $\mathcal{T}_\Sigma$, and $\langle \cdot, \cdot \rangle$ denotes the natural pairing between distributions and test functions. 
    \end{enumerate}
\end{definition}


To the best of our knowledge, the formulation of PDE inversion conditions of the above kind in the context of mixture model theory is new. As it turns out, conditions of this form are especially well-suited to our problem, as they provide a convenient way to invert the convolution $p_G(x)=P*f(x\mid\cdot,\Sigma)$ by means of a differential operator of finite order $\beta\in\mathbb N$, whose coefficients depend smoothly on the scale matrix $\Sigma$. Exploiting this relationship---which, as shown below, holds for Laplace mixtures\footnote{Though we do not pursue this here, the PDE inversion condition of Definition~\ref{def:PDE_inversion} is applicable to other kinds of mixtures, such as those driven by shifted exponential and fixed-shape Gamma kernels.}---allows one to relate differences between the mixture densities $p_G$ and $p_{G'}$ directly to differences between the corresponding mixing measures $P$ and $P'$ and the scale parameters $\Sigma$ and $\Sigma'$. Finally, we emphasize that the use of distribution theory is necessary here, since the mixing measure $P$ may be purely atomic, as is the case almost surely under a Dirichlet process prior, and differentiation must be interpreted in a weak sense via test functions $\phi\in C_c^\infty(\mathbb R^d)$. 

Before presenting an inverse bound based on the PDE inversion condition, we relate it to the ordinary smoothness of the kernel, thereby justifying its use as a substitute for obtaining sharper inequalities.

\begin{proposition}
\label{pro:PDE_implies_ordinary}
Suppose that $f(x \mid 0, \Sigma) = |\Sigma|^{-1/2} g(x^\top \Sigma^{-1} x)$ for some $g: [0,\infty) \to [0,\infty)$. If $\mathscr P$ satisfies the PDE inversion condition of even order $\beta \in \mathbb{N}$, then $f$ is ordinary-smooth of order $\beta$.
\end{proposition}

That is, when the index $\beta$ is even (as it is the case for Laplace mixtures, cf. Theorem~\ref{thm:inequality_laplace} below), the PDE condition is a proper strengthening of ordinary-smoothness of the kernel $f$. This, as we show next, plays a crucial role in sharpening the inverse bound obtained in Theorem~\ref{theorem:ordinary_rate}.

\begin{theorem}
    \label{theorem:sharpen_ordinary_rate}
    Under Assumption~\ref{ass1}, if $\mathscr P$ satisfies a PDE inversion condition of order $\beta$, there exist positive quantities $C$ and $C'$, depending only on $R$, $\lambda_{\min}$, $\lambda_{\max}$, and $d$, such that
    \begin{align}
       \|p_G &- p_{G'}\|_1 \geq C' \left[\Psi(\|\Sigma - \Sigma'\|_{\mathrm{op}}) \;+\; \min\left\{ W_1(P, P')^{\beta}, \;\Psi\big(C W_1(P, P')^{\beta}\big) \right\}\right], \label{eq:sharpen_lower_bound_ordinary_smooth}
    \end{align}
    for any $G=P\times \delta_\Sigma$ and $G'=P'\times \delta_{\Sigma'}$ in $\mathscr G$. 
\end{theorem}

Theorem~\ref{theorem:sharpen_ordinary_rate} improves upon Theorem~\ref{theorem:ordinary_rate} by removing the explicit dependence on the dimension $d$ in the exponent. That is, while the bound in Theorem~\ref{theorem:ordinary_rate} describes a dimension-dependent relationship under simple ordinary-smoothness, the PDE inversion condition yields, remarkably, a dimension-free rate governed solely by the smoothness order $\beta$. This shows that, for mixtures satisfying this condition, the efficiency in recovering the mixing measure is determined by the smoothness properties of the kernel (with higher smoothness leading to weaker bounds) rather than by the ambient dimension.

\begin{proof}[Proof sketch of Theorem~\ref{theorem:sharpen_ordinary_rate}]
The proof shares Step~1 with Theorems~\ref{theorem:supersmooth_rate} and~\ref{theorem:ordinary_rate}, producing a compactly supported, globally Lipschitz continuous function $\widetilde{h}^*_{\mathcal{X}} \in L^1(\mathbb{R}^d)$ satisfying $W_1(P,P') = \int_{\mathbb{R}^d}\widetilde{h}^*_{\mathcal{X}}(\theta)\,(P(d\theta)-P'(d\theta))$. The key departure begins in Step~2'', where we use a smooth mollifier rather than a Jackson kernel to construct the approximation $\widetilde{h}^*_{\mathcal{X},\Lambda}$, which grants additional regularities needed to invoke the PDE inversion condition. Steps~3''--5'' then exploit this condition directly in the spatial domain, bypassing Fourier inversion entirely and leading to dimension-free bounds.

\emph{Step 2'': Mollifier approximation.\hspace{1em}} Rather than convolving with a Jackson kernel as in the previous proofs, we convolve $\widetilde{h}^*_{\mathcal{X}}$ with a scaled mollifier $\psi_\Lambda(x) := \Lambda^d\psi(\Lambda x)$, where $\psi(x) = C_\psi^{-1}\varrho(x)$ is a normalized bump function with $C_\psi := \int_{\mathbb{R}^d}\varrho(x)\,dx$ and
\[
\varrho(x) = \begin{cases} \exp\!\left(-\dfrac{1}{1-\|x\|^2}\right) & \text{if } \|x\| < 1, \\ 0 & \text{if } \|x\| \geq 1, \end{cases}
\]
so that $\psi \in C_c^\infty(\mathbb{R}^d)$ and $\int_{\mathbb{R}^d}\psi(x)\,dx = 1$. We then define
\[
\widetilde{h}^*_{\mathcal{X},\Lambda}(x) := (\widetilde{h}^*_{\mathcal{X}} * \psi_\Lambda)(x) = \int_{\mathbb{R}^d} \widetilde{h}^*_{\mathcal{X}}(x-y)\,\psi_\Lambda(y)\,dy.
\]
The resulting function $\widetilde{h}^*_{\mathcal{X},\Lambda}$ belongs to $C_c^\infty(\mathbb{R}^d)$, which makes it an admissible test function for the distributional PDE inversion condition. The Lipschitz continuity of $\widetilde{h}^*_{\mathcal{X}}$ again implies $\|\widetilde{h}^*_{\mathcal{X},\Lambda} - \widetilde{h}^*_{\mathcal{X}}\|_\infty \leq (1+2RM_\eta)M_\psi/\Lambda$, where $M_\psi := \int_{\mathbb{R}^d}\|x\|\psi(x)\,dx < \infty$, and combined with $\sup_{A\in\mathscr{B}(\mathbb{R}^d)}|P(A)-P'(A)| \leq 2$ this yields
\begin{equation}
\label{eq:sketch_split_pde}
W_1(P, P') \;\leq\; \frac{C_0}{\Lambda} + \int_{\mathbb{R}^d} \widetilde{h}^*_{\mathcal{X},\Lambda}(\theta)\,(P(d\theta) - P'(d\theta))
\end{equation}
for a constant $C_0>0$ depending only on $R$.

\emph{Step 3'': PDE inversion.\hspace{1em}} Since $\mathcal{T}_\Sigma p_G = P$ and $\mathcal{T}_{\Sigma'} p_{G'} = P'$ hold in the sense of distributions, and $\widetilde{h}^*_{\mathcal{X},\Lambda} \in C_c^\infty(\mathbb{R}^d)$ is an admissible test function, we may write
\begin{align*}
\int_{\mathbb{R}^d} \widetilde{h}^*_{\mathcal{X},\Lambda}(\theta)\,(P(d\theta) - P'(d\theta)) &= \langle \widetilde{h}^*_{\mathcal{X},\Lambda},\, \mathcal{T}_\Sigma p_G - \mathcal{T}_{\Sigma'} p_{G'} \rangle \\
&= \langle \mathcal{T}^*_\Sigma \widetilde{h}^*_{\mathcal{X},\Lambda},\, p_G - p_{G'} \rangle + \langle (\mathcal{T}^*_\Sigma - \mathcal{T}^*_{\Sigma'})\widetilde{h}^*_{\mathcal{X},\Lambda},\, p_{G'} \rangle,
\end{align*}
where $\mathcal{T}^*_\Sigma := \sum_{0 \leq |\nu| \leq \beta} (-1)^{|\nu|} c_\nu(\Sigma)\partial^\nu$ is the formal adjoint of $\mathcal{T}_\Sigma$, and we added and subtracted $\mathcal{T}_\Sigma p_{G'}$ to isolate the two sources of discrepancy. We label the two resulting terms $I_1$ (observation penalty) and $I_2$ (scale mismatch).

\emph{Step 4'': Bounding $I_1$ and $I_2$.\hspace{1em}} For $I_1$, Hölder's inequality gives $|I_1| \leq \|p_G - p_{G'}\|_1\|\mathcal{T}^*_\Sigma \widetilde{h}^*_{\mathcal{X},\Lambda}\|_\infty$. To bound $\|\mathcal{T}^*_\Sigma \widetilde{h}^*_{\mathcal{X},\Lambda}\|_\infty$, we use the fact that $\widetilde{h}^*_{\mathcal{X},\Lambda} = \widetilde{h}^*_{\mathcal{X}} * \psi_\Lambda$ to transfer one derivative onto $\widetilde{h}^*_{\mathcal{X}}$ and the remaining ones onto $\psi_\Lambda$. Young's inequality and the scaling identity $\|\partial^\nu \psi_\Lambda\|_1 = \Lambda^{|\nu|-1}\|\partial^\nu\psi\|_1$ then yield $\|\partial^\nu \widetilde{h}^*_{\mathcal{X},\Lambda}\|_\infty \lesssim \Lambda^{|\nu|-1}$ for all $\nu\in\mathbb N^d$ with $1 \leq |\nu| \leq \beta$. Since the coefficients $c_\nu(\Sigma)$ are uniformly bounded over $\Sigma \in \mathbb{S}^+_d(\lambda_{\min},\lambda_{\max})$, we conclude
\[
|I_1| \;\lesssim\; (1 + \Lambda^{\beta-1})\|p_G - p_{G'}\|_1.
\]
For $I_2$, we use $\|p_{G'}\|_1 = 1$ together with the $L$-Lipschitz continuity of $\Sigma\mapsto c_\nu(\Sigma)$ from the PDE inversion condition, which gives $|c_\nu(\Sigma) - c_\nu(\Sigma')| \leq L\|\Sigma - \Sigma'\|_{\mathrm{op}}$. Applying the same derivative bounds as for $I_1$, we obtain
\[
|I_2| \;\lesssim\; (1+\Lambda^{\beta-1})\|\Sigma - \Sigma'\|_{\mathrm{op}}.
\]
Note that both bounds are now dimension-free, in contrast to those obtained in the proof of Theorem~\ref{theorem:ordinary_rate}.

\emph{Step 5'': Conclusion.\hspace{1em}} Substituting the bounds on $I_1$ and $I_2$ into Equation~\eqref{eq:sketch_split_pde} and using Assumption~\ref{ass1} to replace $\|\Sigma - \Sigma'\|_{\mathrm{op}}$ with $\Psi^{-1}(\|p_G - p_{G'}\|_1)$, we obtain
\[
W_1(P,P') \;\lesssim\; \frac{C_0}{\Lambda} + (1+\Lambda^{\beta-1})\|p_G - p_{G'}\|_1 + (1+\Lambda^{\beta-1})\Psi^{-1}(\|p_G-p_{G'}\|_1).
\]
Two cases arise depending on the size of $W_1(P,P')$. If $W_1(P,P')$ is bounded away from zero, setting $\Lambda = 1$ and rearranging yields $\|p_G - p_{G'}\|_1 \gtrsim \min\{W_1(P,P')^\beta,\,\Psi(CW_1(P,P')^\beta)\}$ directly. If instead $W_1(P,P')$ is small, setting $\Lambda \asymp W_1(P,P')^{-1}$ and rearranging leads to the same bound. In both cases, combining with Assumption~\ref{ass1} gives Equation~\eqref{eq:sharpen_lower_bound_ordinary_smooth}.
\end{proof}

\section{Examples}\label{sec:examples}

We now specialize the general inequalities from Section~\ref{sec:general_inequalities} to specific choices of the mixture kernel $f$. In particular, we focus on three canonical cases, namely Gaussian, Cauchy, and Laplace kernels, which exemplify the different assumptions considered in our theory.

\subsection{Gaussian kernel}

The multivariate Gaussian kernel, which is arguably the most widely adopted choice in applied mixture modeling, is given by
$$f(x\mid\theta, \Sigma) = \frac{\exp\left( -\frac{1}{2} (x - \theta)^\top \Sigma^{-1} (x - \theta) \right)}{(2\pi)^{\frac{d}{2}} |\Sigma|^{\frac{1}{2}}}, \quad x \in \mathbb{R}^d,$$
for some $(\theta, \Sigma) \in \mathcal X \times \mathbb S_d^+(\lambda_{\min},\lambda_{\max})$. The following theorem shows that it belongs to the class of super-smooth kernels and provides the corresponding explicit inverse bound.

\begin{theorem}\label{thm:inequality_gaussian}
    The multivariate Gaussian kernel is super-smooth of order $2$. Moreover, if $f$ is Gaussian, there exists a constant $M>0$, depending only on $R$, $\lambda_{\min}$, and $\lambda_{\max}$, such that Assumptions~\ref{ass1} and~\ref{ass2} hold with
    \begin{align*}
        \Psi(t) = \exp\left( -M \frac{\log (1/t)}{t} \right), \quad \Xi(t)=t, \quad\text{and}\quad p=2.
    \end{align*}
    Consequently,
    \begin{align*}
        \|p_G  - p_{G'}\|_1 \;\geq\;C'\left[\exp\left( - \frac{C\exp\big(\tilde C/W_1(P, P')^2\big)}{W_1(P, P')^2} \right) \;+\; \exp\left(-M \frac{\log(1/\|\Sigma - \Sigma'\|_{\mathrm{op}})}{\|\Sigma - \Sigma'\|_{\mathrm{op}}} \right)\right],
    \end{align*}
    for any $G=P\times \delta_\Sigma$ and $G'=P'\times \delta_{\Sigma'}$ in $\mathscr G$, where $C,C',\tilde C>0$ depend on $R$, $\lambda_{\min}$, $\lambda_{\max}$, and $d$. 
\end{theorem}

As a preview of the posterior contraction results in Section~\ref{sec:posterior_rates}, Theorem~\ref{thm:inequality_gaussian} shows how an $L^1$ estimation rate $\omega_n$ for $p_{G_0}$ propagates to $P_0$ and $\Sigma_0$. In particular, one obtains
\begin{equation*}
    W_1(P,P_0) \lesssim \left( \log\log(1/\omega_n) \right)^{-1/2}, \quad \|\Sigma - \Sigma_0\|_{\mathrm{op}} \lesssim \frac{\log\log(1/\omega_n)}{\log(1/\omega_n)}.
\end{equation*}
These rates are quite slow, reflecting the strong smoothing effect induced by the Gaussian kernel. This will serve as a benchmark when comparing with other kernel choices.

Before moving on to other kernel types, we briefly focus on the Gaussian setting where the scale matrices under consideration are of the form $\Sigma=\sigma^2 I_d$, so that $\mathscr G$ consists of mixing measures  $G=P\times\delta_{\sigma^2I_d}$ with $\sigma^2\in[\lambda_{\min},\lambda_{\max}]$. We refer to this as an \emph{isotropic Gaussian mixture}. In this case, the next result improves the $W_1$ rate dependence from double-logarithmic to logarithmic. From a technical point of view, the proof is not significantly different from the general super-smooth case and only exploits the additional semi-group structure of the Gaussian kernel family.

\begin{theorem}\label{thm:inequality_gaussian_isotropic}
    For an isotropic Gaussian mixture, there exist positive constants $M,C$, and $C'$, depending only on $R$, $\lambda_{\min}$, and $\lambda_{\max}$, and $d$, such that
    \begin{align*}
        \|p_G  - p_{G'}\|_1 \;\geq\;C'\left[\exp\left( - \frac{C}{W_1(P, P')^2} \right) \;+\; \exp\left(-M \frac{\log(1/|\sigma^2 - \sigma'^{2}|)}{|\sigma^2 - \sigma'^2|} \right)\right],
    \end{align*}
    for any $G=P\times \delta_{\sigma^2I_d}$ and $G'=P'\times \delta_{\sigma'^2 I_d}$ in $\mathscr G$. 
\end{theorem}

While Theorem~\ref{thm:inequality_gaussian_isotropic} leaves the estimation rate for $\Sigma_0=\sigma_0^2I_d$ unchanged relative to the fully general counterpart in Theorem~\ref{thm:inequality_gaussian}, it significantly improves the implied $W_1$ bound,
$$W_1(P,P_0) \lesssim \left(\log(1/\omega_n) \right)^{-1/2},$$
which will be reflected in our analysis of posterior contraction rates in Section~\ref{sec:posterior_rates}. We also trivially note that, for $d=1$, Proposition~\ref{thm:inequality_gaussian_isotropic} fully replaces Theorem~\ref{thm:inequality_gaussian} as a strictly tighter result.

\subsection{Cauchy kernel}

We next consider the multivariate Cauchy kernel,
$$f(x\mid\theta, \Sigma) = \frac{\Gamma\left(\frac{1+d}{2}\right)}{\pi^{\frac{d+1}{2}} |\Sigma|^{\frac{1}{2}} \left[ 1 + (x - \theta)^\top \Sigma^{-1} (x - \theta) \right]^{\frac{d+1}{2}}},$$
which leads to the following inverse bound.

\begin{theorem}\label{thm:inequality_cauchy}
    The multivariate Cauchy kernel is super-smooth of order $1$. Moreover, if $f$ is Cauchy, Assumptions~\ref{ass1} and~\ref{ass2} hold with
    \begin{align*}
        \Psi(t) = t^2, \quad \Xi(t)=\sqrt t, \quad\text{and}\quad p=1.
    \end{align*}
    Consequently,
    \begin{align*}
    \|p_G - p_{G'}\|_1 \;\geq\; C'\left[\exp\left( - \frac{C}{W_1(P, P')} \right) + \|\Sigma - \Sigma'\|_{\mathrm{op}}^2\right],
    \end{align*}
    for any $G=P\times \delta_\Sigma$ and $G'=P'\times \delta_{\Sigma'}$ in $\mathscr G$, where $C,C'>0$ may depend on $R$, $\lambda_{\min}$, $\lambda_{\max}$, and $d$.
\end{theorem}

In this case, an $L^1$ rate $\omega_n$ for $p_{G_0}$ translates into
\begin{equation*}
    W_1(P,P_0) \lesssim \left(\log(1/\omega_n)\right)^{-1}, \quad \|\Sigma - \Sigma_0\|_{\mathrm{op}} \lesssim \omega_n^{1/2}.
\end{equation*}
Compared to the Gaussian setting, both rates are faster. This is consistent with the lower degree of smoothness of the Cauchy kernel ($\alpha=1$), which reduces the loss of information induced by convolving $G$ with $f$ and makes the inversion problem statistically easier.

\subsection{Laplace kernel}

Finally, we consider the multivariate Laplace kernel
\begin{align*}
    f(x \mid\theta, \Sigma) = \frac{2\left( \frac{(x - \theta)^\top \Sigma^{-1} (x - \theta)}{2} \right)^{\frac{2-d}{2}} B_{\frac{d-2}{2}}\left( \sqrt{2 (x - \theta)^\top \Sigma^{-1} (x - \theta)} \right)}{(2\pi)^{\frac{d}{2}} |\Sigma|^{\frac{1}{2}}},
\end{align*}
where $B_{\kappa}$ denotes the modified Bessel function of the second kind. As previewed in Section~\ref{sec:general_inequalities}, Laplace mixtures are ordinary-smooth and satisfy the PDE inversion condition introduced in Definition~\ref{def:PDE_inversion}, which allows one to obtain the following inverse bound.

\begin{theorem}\label{thm:inequality_laplace}
    The multivariate Laplace kernel is ordinary-smooth of order $2$. Moreover, if $f$ is Laplace, Assumptions~\ref{ass1} and~\ref{ass2} hold with
    \begin{align*}
        \Psi(t) = \Xi(t)=t \quad\text{and}\quad p=2,
    \end{align*}
    and $\mathscr P$ satisfies a PDE inversion condition of order $2$ for the operator
    $$\mathcal{T}_\Sigma = \mathrm{Id} - \frac{1}{2} \sum_{j=1}^d \sum_{k=1}^d \Sigma_{jk} \frac{\partial^2}{\partial x_j \partial x_k},$$
    where $\mathrm{Id}$ denotes the identity operator. Consequently,
    \begin{align*}
       \|p_G - p_{G'}\|_1 \;\geq\; C'\left[W_1(P, P')^{2} + \|\Sigma - \Sigma'\|_{\mathrm{op}}\right],
    \end{align*}
    for any $G=P\times \delta_\Sigma$ and $G'=P'\times \delta_{\Sigma'}$ in $\mathscr G$, where $C'
    >0$ may depend on $R$, $\lambda_{\min}$, $\lambda_{\max}$, and $d$. 
\end{theorem}

In this setting, an $L^1$ rate $\omega_n$ for $p_{G_0}$ yields
\begin{equation*}
    W_1(P,P_0) \lesssim \omega_n^{1/2}, \quad \|\Sigma - \Sigma_0\|_{\mathrm{op}} \lesssim \omega_n.
\end{equation*}
These rates improve upon both the Gaussian and Cauchy cases. This reflects the weaker smoothness of the Laplace kernel (which is only ordinary-smooth) together with the additional structure provided by the PDE inversion property allowing for a more direct recovery of the underlying mixing measure.

\section{Posterior asymptotics for Dirichlet process mixtures}\label{sec:posterior_rates}

Equipped with the general inequalities obtained in Section~\ref{sec:general_inequalities} and their specializations in Section~\ref{sec:examples}, we now derive posterior contraction results for mixing measures in Bayesian nonparametric homoscedastic location-scale mixture models. As discussed earlier, our approach proceeds in two steps:
\begin{enumerate}
    \item We verify an $L^1$ posterior contraction rate $\omega_n$ for density estimation, namely such that
    \begin{equation*}
        \Pi\left(G : \|p_G - p_{G_0}\|_1\geq K\omega_n\mid X_1,\ldots,X_n\right) \to 0
    \end{equation*}
    in $P_{G_0}^n$-probability for all large $K>0$ as $n\to\infty$, for a given kernel $f$ defining the density $p_G(x)=P*f(x\mid\cdot,\Sigma)$ (as usual, with $G=P\times\delta_\Sigma$). Such results are well studied in the literature on Bayesian asymptotics and will be borrowed accordingly;

    \item We use the inverse inequalities developed in the previous sections to translate the $L^1$ rate $\omega_n$ into contraction rates for $W_1(P,P_0)$ and $\|\Sigma-\Sigma_0\|_{\mathrm{op}}$.
\end{enumerate}

We provide explicit results for multivariate Gaussian mixtures and univariate Laplace mixtures under Dirichlet process priors, as these are settings where suitable $L^1$ contraction rates are available. In contrast, we do not include any explicit result for Cauchy mixtures, despite the bounds obtained in Theorem~\ref{thm:inequality_cauchy}. This is because, to the best of our knowledge, no $L^1$ posterior contraction results are currently available for infinite Cauchy mixtures. This gap appears to stem from limitations in the existing techniques used to estimate the entropy numbers of such model classes, which are in turn essential to establish posterior contraction rates (cf.~the discussion in Section~\ref{sec:notation}). In particular, standard approaches, originating from \cite{ghosal2001entropies}, rely on approximating infinite mixtures by finitely supported ones via moment-matching strategies, which fail in their current form with heavy-tailed kernels. As a consequence, entropy bounds and posterior contraction rates for infinite Cauchy mixtures constitute an open problem which is beyond the scope of this article.

In the Gaussian case, combining our inequalities with Theorem~9.9 in \cite{ghosal2017fundamentals} \citep[see also][]{shen2013adaptive}, which establishes the $L^1$ contraction rate $\omega_n=n^{-1/2}(\log n)^{(d+1)/2}$, yields the following result.

\begin{proposition}\label{pro:rate_multivariate_normal}
    Assume that $X_1,X_2,\ldots\overset{\textnormal{i.i.d.}}{\sim}p_{G_0}=\int_{\mathbb R} f(\cdot\mid \theta,\Sigma_0)\,P_0(d\theta)$, where $f$ is the $d$-variate Gaussian kernel, $P_0(\bar B(0,R_0))=1$ for some $R_0\in(0,\infty)$, and $\Sigma_0\in\mathbb{S}^{+}_{d}(\lambda_{\min},\lambda_{\max})$ for some $0<\lambda_{\min}<\lambda_{\max}<\infty$. Consider a homoscedastic Gaussian location-scale mixture prior $\Pi$ on $G=P\times \delta_{\Sigma}$, where $P\sim \mathrm{DP}(a,H)$ (for some $a>0$) with base measure $H$ having a strictly positive and continuous density supported on $[-R,R]^d$ for some $R\in[R_0,\infty)$, and $\Sigma$ has a prior with positive continuous density supported on $\mathbb{S}^{+}_{d}(\lambda_{\min},\lambda_{\max})$. Then the following statements hold true. 
    
    \begin{enumerate}
        \item[(a)] For $K>0$ sufficiently large,
        \begin{align*}
            \Pi\Bigg(G & \;:\; \exp\bigg( - \frac{C\exp\big(\tilde C/W_1(P, P_0)^2\big)}{W_1(P, P_0)^2} \bigg) \\
            &+ \exp\bigg(-M \frac{\log(1/\|\Sigma-\Sigma_0\|_{\mathrm{op}})}{\|\Sigma-\Sigma_0\|_{\mathrm{op}}} \bigg) \;\geq\; Kn^{-1/2}(\log n)^{\frac{d+1}{2}}
            \;\Big|\; X_1,\ldots,X_n\Bigg) \to 0
        \end{align*}
        in $P_{G_0}^n$-probability as $n\to\infty$, where the positive constants $M$, $C$, and $\tilde C$ may depend on $R$, $\lambda_{\min}$, $\lambda_{\max}$, and $d$. Consequently,
        \begin{align*}
            \Pi\Bigg(G \,:\, W_1(P, P_0) \geq K \frac{1}{\sqrt{\log\log n}} \;\Big|\; X_1,\ldots,X_n\Bigg) &\to 0, \\
            \Pi\Bigg(G \,:\, \|\Sigma-\Sigma_0\|_{\mathrm{op}} \geq K \frac{\log \log n}{\log n} \;\Big|\; X_1,\ldots,X_n\Bigg) &\to 0
        \end{align*}
        in $P_{G_0}^n$-probability for all large $K>0$ as $n\to\infty$.

        \item[(b)] In the isotropic Gaussian mixture setting, the above convergence statements may be replaced by
        \begin{align*}
        \Pi\Bigg(G & \;:\; \exp\bigg( - \frac{C}{W_1(P, P_0)^2} \bigg) \\
        &+ \exp\bigg(-M \frac{\log(1/|\sigma^2-\sigma_0^2|)}{|\sigma^2-\sigma^2_0|} \bigg) \;\geq\; Kn^{-1/2}(\log n)^{\frac{d+1}{2}}
        \;\Big|\; X_1,\ldots,X_n\Bigg) \to 0
        \end{align*}
        and
        \begin{align*}
        \Pi\Bigg(G \,:\, W_1(P, P_0) \geq K \frac{1}{\sqrt{\log n}} \;\Big|\; X_1,\ldots,X_n\Bigg) &\to 0,\\
        \Pi\Bigg(G \,:\, |\sigma^2-\sigma_0^2| \geq K \frac{\log \log n}{\log n} \;\Big|\; X_1,\ldots,X_n\Bigg) &\to 0.   
        \end{align*}
        in $P_{G_0}^n$-probability as $n\to\infty$.
    \end{enumerate}

\end{proposition}

Proposition~\ref{pro:rate_multivariate_normal} highlights a key feature of the Gaussian setting. Despite the essentially parametric $L^1$ rate for density estimation, the super-smooth nature of the Gaussian kernel in general leads to slow, double logarithmic rates for recovering the mixing measure $P_0$ and logarithmic rates for the scale matrix $\Sigma_0$. The double logarithmic rate is however turned into a milder (but nonetheless slow) logarithmic one in the isotropic case. On the other hand, the dependence on the dimension $d$, which appears at the exponent of the logarithmic factor in the $L^1$ rate, only affects lower-order terms and does not influence the leading behavior of the $W_1$ and $\|\cdot\|_{\mathrm{op}}$ rates (in the fixed-$d$ regime considered here). Finally, Proposition~\ref{pro:rate_multivariate_normal} highlights that the rate of recovery for $\Sigma_0$ is much faster than the rates for $P_0$, both in the general and isotropic cases, uncovering significant heterogeneity in parameter estimation efficiency with this kind of infinite mixture model.

We now turn to the Laplace case. Using the $L^1$ contraction rate $\omega_n=n^{-1/4}(\log n)^{5/4}$ available for univariate Dirichlet process mixtures of Laplace kernels,\footnote{Adapted from Theorem 4.4 in \cite{scricciolo2011posterior} and setting, in the author's notation, $p=1$.} and combining it with Theorem~\ref{thm:inequality_laplace}, we obtain the following result.

\begin{proposition}\label{pro:rate_univariate_laplace}
    Assume that $X_1,X_2,\ldots\overset{\textnormal{i.i.d.}}{\sim}p_{G_0}=\int_{\mathbb R} f(\cdot\mid \theta,\sigma_0)\,P_0(d\theta)$, where $f$ is the univariate Laplace kernel and $P_0([-R_0,R_0])=1$ for some $R_0\in(0,\infty)$. Consider a homoscedastic Laplace location-scale mixture prior $\Pi$ on $G=P\times \delta_{\sigma}$, where $P\sim \mathrm{DP}(a,H)$ (for some $a>0$) with base measure $H$ having a strictly positive continuous density supported on $[-R,R]\subset\mathbb R$ for some $R\in[R_0,\infty)$, and $\sigma$ has a prior with positive continuous density supported on $[\underline \sigma,\overline \sigma]\subset(0,\infty)$ for some $0<\underline \sigma<\sigma_0<\overline \sigma<\infty$. Then, for $K>0$ sufficiently large,
\begin{align*}
    \Pi\Bigg(G \,:\, W_1(P, P_0)^{2} + |\sigma - \sigma_{0}| \;\geq\; K\frac{(\log n)^{5/4}}{n^{1/4}}
    \;\Big|\; X_1,\ldots,X_n\Bigg) \to 0
\end{align*}
in $P_{G_0}^n$-probability as $n\to\infty$. Consequently,
\begin{align*}
    \Pi\Bigg(G \,:\, W_1(P, P_0) \geq K \frac{(\log n)^{5/8}}{n^{1/8}} \;\Big|\; X_1,\ldots,X_n\Bigg) &\to 0, \\
    \Pi\Bigg(G \,:\, |\sigma-\sigma_0| \geq K \frac{(\log n)^{5/4}}{n^{1/4}} \;\Big|\; X_1,\ldots,X_n\Bigg) &\to 0
\end{align*}
in $P_{G_0}^n$-probability for all large $K>0$ as $n\to\infty$.  
\end{proposition}

Laplace mixtures display a very different behavior compared to the Gaussian case. Although their $L^1$ rate is slower (of order $n^{-1/4}$, up to poly-logarithmic factors, compared to the Gaussian $n^{-1/2}$), the resulting rates of recovery of $P_0$ and $\Sigma_0$ are polynomial rather than logarithmic. This reflects the substantially weaker smoothing induced by the Laplace kernel combined with the PDE inversion structure, which yield a much sharper inverse bound (compare Theorems~\ref{thm:inequality_laplace} and \ref{thm:inequality_gaussian_isotropic}) and compensate for the slower density estimation rate. However, Laplace mixtures exhibit the same qualitative separation in contraction rates as Gaussian mixtures: the location mixing measure $P_0$ is recovered at rate $n^{1/8}(\log n)^{5/8}$, whereas the scale parameter $\Sigma_0$ is recovered at the faster rate $n^{1/4}(\log n)^{5/4}$. Ignoring poly-logarithmic factors, this reflects significant heterogeneity in the form of a polynomial gap in the estimation rates for these two parameters of interest.

\subsection{Mixing measure priors beyond the Dirichlet process}

We conclude this section by noting that, although the results above are stated for Dirichlet process mixtures, the underlying methodology is not restricted to this specific prior choice. The arguments rely on two ingredients: (i) the inverse inequalities developed in Sections~\ref{sec:general_inequalities} and~\ref{sec:examples}, and (ii) available $L^1$ posterior contraction rates for the corresponding mixture models. Since the inequalities are general, they can be combined with contraction results for a wide range of alternative priors on the mixing measure.

Examples include Pitman--Yor processes \citep{perman1992size, pitman1997two}, normalized inverse-Gaussian processes \citep{lijoi2005hierarchical}, more general normalized random measures with independent increments \citep{kingman1967completely,regazzini2003distributional,Nieto-Barajas:2004:NRM, lijoi2010models}, Gibbs-type priors \citep{GnedinPitman2006,deblasi2015gibbs}, and species sampling models \citep{kingman1978representation, pitman1996some}. For instance, the inequalities developed here could be combined with the contraction rates established in \cite{scricciolo2014adaptive} for Pitman--Yor and normalized inverse-Gaussian mixtures. We do not pursue these extensions further, as they do not introduce additional conceptual difficulties given the generality of our framework.

\section{Discussion}\label{sec:discussion}
We studied posterior contraction for mixing measures in homoscedastic location-scale mixture models with infinitely many components. We developed a general approach to derive lower-bounds for the $L^1$ distance between mixture densities in terms of discrepancies between the corresponding mixing measures. Our method relies on the dual formulation of the $W_1$ distance in terms of 1-Lipschitz functions, combined with functional approximation and Fourier inversion techniques, and yields inequalities whose strength reflects the smoothness of the mixture kernel and, in the ordinary-smooth case, the applicability of a novel PDE inversion condition. These results provide a unified route from $L^1$ convergence to contraction at the level of both the mixing measure and the shared scale parameter, and apply to a range of commonly used kernels in the context of Dirichlet process mixtures.

In particular, we point out a connection with deconvolution problems, which have motivated much of the previous work on mixing measure convergence, as they can be framed, in the mixture setting, as the recovery of a signal distributed according to the unknown mixing measure when observations are contaminated by additive noise drawn from the mixture kernel \citep{Gao2016Posterior,scricciolo2018bayes,rousseau2024wasserstein}. As noted earlier, existing contributions focus on the location-only mixture case, which in deconvolution terminology corresponds to the setting where the noise scale is known. Our results can be interpreted as characterizing recovery of the mixing measure in the more challenging scenario where the noise scale is also unknown. Moreover, in contrast to other multivariate deconvolution analyses \citep{rousseau2024wasserstein}, our approach does not rely on independence assumptions across noise coordinates and instead accommodates general covariance structures. New results on $L^1$ density estimation for infinite Cauchy mixtures would also make it possible, when combined with our inequality in Theorem~\ref{thm:inequality_cauchy}, to fully characterize the currently unexplored problem of deconvolution with heavy-tailed Cauchy error distribution. 

The functional-analytic techniques we developed to handle the dual formulation of the $W_1$ metric are quite general and open several directions for future work. A natural extension is to heteroscedastic location-scale mixtures, where the scale parameter varies across components. In addition, our results for different kernel choices reveal significant heterogeneity, with Laplace mixtures achieving substantially faster rates than Gaussian mixtures. This raises the question of whether, when the primary objects of interest are the mixing measure and the scale (as in the deconvolution problems discussed above), it may be advantageous to use a misspecified kernel with lower smoothness, potentially improving recovery rates---perhaps up to an approximation error---when the true kernel is very smooth. A natural setting to investigate this question is the class of exponential power kernels considered for density estimation by \cite{scricciolo2011posterior}, where the smoothness can be tuned continuously and misspecification can be studied in a systematic way. Finally, it would be of interest to extend our results to models with more complex prior structures on the mixing measure, such as those inducing dependence across multiple mixing measures \citep[][among many others]{Teh-et-al06,rodriguez2008nested,lijoi2014bayesian,franzolini2025multivariate} or repulsion across atoms to achieve well-separated components \citep{petralia2012repulsive,xu2016bayesian}.

\clearpage
\begin{center}

{\bf{\LARGE{Appendices}}}
\end{center}

\appendix

In the following appendices, we report the detailed proofs of the theorems and propositions presented in the main text of the paper (Appendix~\ref{appA}), as well as the statements and proofs of several auxiliary lemmas (Appendix~\ref{appB})

\section{Proofs of theorems and propositions in the main text}\label{appA}

\subsection*{Proof of Theorem~\ref{theorem:supersmooth_rate}}\label{proof_thm3.2}
Recall that the probability measures $P$ and $P'$ are  supported on $\mathcal{X} = \bar{B}(0, R)$. By the fundamental Kantorovich-Rubinstein Duality Theorem \citep{Villani-03} for optimal transport on the compact metric space $\mathcal{X}$, the $W_1$ distance can be expressed as
$$ W_1(P, P') = \sup_{h \in \mathrm{Lip}_1(\mathcal{X})} \int_{\mathcal{X}} h(\theta) (P(d\theta) - P'(d\theta)) = \int_{\mathcal{X}} h^*(\theta) (P(d\theta) - P'(d\theta)),$$
%
for a function $h^* : \mathcal{X} \to \mathbb{R}$ with $h^* \in \mathrm{Lip}_1(\mathcal{X})$.

We define $ h_{\mathcal{X}}^{*}(x) = h^*(x) - h^*(0)$ for all $x \in \mathcal{X}$ to guarantee that $h_{\mathcal{X}}^{*}(0) = 0$. Furthermore, $|h_{\mathcal{X}}^{*}(x) - h_{\mathcal{X}}^{*}(y)| = |h^*(x) - h^*(y)| \le \|x - y\|$. Thus, $h_{\mathcal{X}}^{*} \in \mathrm{Lip}_1(\mathcal{X})$. The shift also leaves the optimal transport integral unchanged, which implies that
$$\int_{\mathcal{X}} h_{\mathcal{X}}^{*}(\theta) (P(d\theta) - P'(d\theta)) = \int_{\mathcal{X}} h^*(\theta) (P(d\theta) - P'(d\theta)) = W_1(P, P').$$

\subsubsection*{Extension from compact domain $\mathcal{X}$ to $\mathbb{R}^{d}$} We now extend the domain of the function $h_{\mathcal{X}}^{*}$ from $\mathcal{X}$ to $\mathbb{R}^d$, while preserving its exact 1-Lipschitz property. This can be done via the McShane-Whitney extension, which is recalled in Lemma~\ref{lemma:global_extension} (see \cite{heinonen2001lectures}, \cite{Whitney1934}, or \cite{McShane1934} for reference). Given the extension of the function $h_{\mathcal{X}}^{*}$ to $\bar{h}_{\mathcal{X}}^{*}$, we can rewrite the integral from the compact subset $\mathcal{X}$ to the full Euclidean space $\mathbb{R}^d$ as follows: 
\begin{align*}
\int_{\mathbb{R}^d}\bar{h}_{\mathcal{X}}^{*}(\theta) (P(d\theta) - P'(d\theta)) & = \int_{\mathcal{X}} \bar{h}_{\mathcal{X}}^{*}(\theta) (P(d\theta) - P'(d\theta)) + \int_{\mathbb{R}^d \setminus \mathcal{X}} \bar{h}_{\mathcal{X}}^{*}(\theta)(P(d\theta) - P'(d\theta))\\
& = \int_{\mathcal{X}} h_{\mathcal{X}}^{*}(\theta) (P(d\theta) - P'(d\theta)) \\
& = W_1(P, P') . 
\end{align*}

\subsubsection*{Spatial truncation via smooth cutoff} Since $\bar{h}_{\mathcal{X}}^{*} \in \mathrm{Lip}_1(\mathbb{R}^d)$, we have the global geometric growth bound:
$$\forall x \in \mathbb{R}^d,\quad |\bar{h}_{\mathcal{X}}^{*}(x)| = |\bar{h}_{\mathcal{X}}^{*}(x) - \bar{h}_{\mathcal{X}}^{*}(0)| \le \|x\|.$$
Our next goal is to further modify $\bar{h}_{\mathcal{X}}^{*}(x)$ to obtain an absolutely integrable function that is amenable to Fourier inversion, and we do so via spatial truncation as follows.

Concretely, we aim to construct a spatial cutoff function $\eta: \mathbb{R}^d \to \mathbb{R}$ in $C_c^\infty(\mathbb{R}^d)$ that takes values exactly in $[0,1]$, evaluates identically to $1$ on the closed ball $\bar{B}(0, R)$ (where $P$ and $P'$ are supported), and vanishes completely outside the open ball $B(0, 2R)$. Lemma~\ref{lemma:cutoff_function} provides the detailed construction of such a cutoff function (see also \citep{Lee2013SmoothManifolds}). Hence, given $\eta$ , we denote 
$$\widetilde{h}_{\mathcal{X}}^{*}(x) := \bar{h}_{\mathcal{X}}^{*}(x) \eta(x)$$ 
for all $x \in \mathbb{R}^{d}$. As we show in Lemma~\ref{lemma:cutoff_function_properties}, $\widetilde{h}_{\mathcal{X}}^{*} \in L_{1}(\mathbb{R}^{d}):=\{h\in\mathbb R^{(\mathbb R^d)} : \int_{\mathbb R^d} |h(x)|dx<\infty\}$ and it is globally Lipschitz on $\mathbb{R}^{d}$. By the properties of $\eta$, we also clearly have that
\begin{align}
\int_{\mathbb{R}^d}\widetilde{h}_{\mathcal{X}}^{*}(\theta) (P(d\theta) - P'(d\theta)) & = \int_{\mathcal{X}} \widetilde{h}_{\mathcal{X}}^{*}(\theta) (P(d\theta) - P'(d\theta)) + \int_{\mathbb{R}^d \setminus \mathcal{X}} \widetilde{h}_{\mathcal{X}}^{*}(\theta)(P(d\theta) - P'(d\theta)) \nonumber \\
& = \int_{\mathcal{X}} h_{\mathcal{X}}^{*}(\theta) (P(d\theta) - P'(d\theta)) \nonumber\\
&= W_1(P, P'). \label{eq:key_equality_duality}
\end{align}

\subsubsection*{Spectral bandlimiting via the Jackson kernel} Since $\widetilde{h}_{\mathcal{X}}^{*} \in L_{1}(\mathbb{R}^{d})$, its Fourier transform $\mathcal{F}[\widetilde{h}_{\mathcal{X}}^{*}](\xi)$ is well-defined. Our goal is to analyze $\int_{\mathcal{X}} \widetilde h_{\mathcal{X}}^{*}(\theta) (P(d\theta) - P'(d\theta))$, which by the Fourier inversion theorem can be rewritten as $\int_{\mathbb{R}^{d}} \mathcal{F}[h_{\mathcal{X}}^{*}](\xi) (\Phi_{P}(\xi) - \Phi_{P'}(\xi))d\xi$. However, the integrand $\mathcal{F}[h_{\mathcal{X}}^{*}](\xi) (\Phi_{P}(\xi) - \Phi_{P'}(\xi))$ may decay slowly and fail to be absolutely integrable. To prevent this, we truncate the frequency domain within a neighborhood $\{\xi\in\mathbb R^d :\|\xi\| \leq \Lambda\}$ for some $\Lambda>0$. This can be done via convolving the function $\widetilde{h}_{\mathcal{X}}^{*}$ with the Jackson kernel. 

To that end, consider the multivariate Jackson kernel (see \cite{devore1993constructive}, \cite{butzer1971fourier})
$$K(x) := \bigg( \frac{3}{8\pi\sqrt{d}} \bigg)^d \prod_{i=1}^d \mathrm{sinc}^4\bigg(\frac{x_i}{4\sqrt{d}}\bigg), \quad x \in \mathbb{R}^{d},$$
where $\mathrm{sinc}(u) := \sin u/u$, which by construction satisfies $\int_{\mathbb{R}^d} K(x) dx = 1$. Lemma~\ref{lemma:kernel_properties} establishes several properties of the Fourier transform of $K$ and its moments, which we use in what follows. For any $\Lambda > 0$, define the function $K_{\Lambda}(x) := \Lambda^{d} K(\Lambda x)$. By Lemma~\ref{lemma:kernel_properties}, $\mathrm{supp}(\mathcal{F}[K]) \subseteq \bar{B}(0, 1)$, which implies that $\mathrm{supp}(\mathcal{F}[K_{\Lambda}]) \subseteq \bar{B}(0, \Lambda)$.\footnote{For any function $h:\mathbb R^d\to\mathbb R$, $\mathrm{supp}(h)$ denotes the closure of the set $\{x\in\mathbb R^d : h(x)\neq 0\}$.} Consider the function
$$\widetilde{h}_{\mathcal{X}, \Lambda}^{*}(x) := (\widetilde{h}_{\mathcal{X}}^{*} * K_\Lambda)(x) = \int_{\mathbb{R}^d} \widetilde{h}_{\mathcal{X}}^{*}(x - y) K_\Lambda(y) dy.$$
Direct computation yields that
\begin{align*}
    \widetilde{h}_{\mathcal{X}, \Lambda}^{*}(x) - \widetilde{h}_{\mathcal{X}}^{*}(x) & = \int_{\mathbb{R}^d} \widetilde{h}_{\mathcal{X}}^{*}(x - y) K_\Lambda(y) dy - \int_{\mathbb{R}^d} \widetilde{h}_{\mathcal{X}}^{*}(x) K_\Lambda(y) dy \\
    & = \int_{\mathbb{R}^d} [\widetilde{h}_{\mathcal{X}}^{*}(x - y) - \widetilde{h}_{\mathcal{X}}^{*}(x)] K_\Lambda(y) dy.
\end{align*}
Since the function $\widetilde{h}_{\mathcal{X}}^{*}$ is a globally $(1+2RM_{\eta})$-Lipschitz function (which follows from Lemma~\ref{lemma:cutoff_function_properties}), we further obtain that, for any $x \in \mathbb{R}^{d}$,
\begin{align*}
    |\widetilde{h}_{\mathcal{X}, \Lambda}^{*}(x) - \widetilde{h}_{\mathcal{X}}^{*}(x)| &\leq \int_{\mathbb{R}^d} |\widetilde{h}_{\mathcal{X}}^{*}(x - y) - \widetilde{h}_{\mathcal{X}}^{*}(x)| \Lambda^{d} K(\Lambda y) dy\\
    &\leq (1+2RM_{\eta}) \int_{\mathbb{R}^d} \|y\| \Lambda^{d} K(\Lambda y) dy \\
    & = \frac{(1 + 2RM_{\eta})M_{K}}{\Lambda},
\end{align*}
where the last inequality follows from Lemma~\ref{lemma:kernel_properties}. As a consequence, we arrive at
\begin{align}
    \|\widetilde{h}_{\mathcal{X}, \Lambda}^{*}- \widetilde{h}_{\mathcal{X}}^{*}\|_{\infty} \leq \frac{(1 + 2RM_{\eta})M_{K}}{\Lambda}. \label{eq:key_approximation_inequality_1}
\end{align}
From the dual formulation of $W_{1}(P,P')$ in Equation~\eqref{eq:key_equality_duality}, we have
\begin{align*}W_1(P, P') &= \int_{\mathbb{R}^d}\widetilde{h}_{\mathcal{X}}^{*}(\theta) (P(d\theta) - P'(d\theta)) \\
&=\int_{\mathbb{R}^d} (\widetilde{h}_{\mathcal{X}}^{*}(\theta) - \widetilde{h}_{\mathcal{X}, \Lambda}^{*}(\theta)) (P(d\theta) - P'(d\theta)) + \int_{\mathbb{R}^d}  \widetilde{h}_{\mathcal{X}, \Lambda}^{*}(\theta) d(P(d\theta) - P'(d\theta)).
\end{align*}
A direct application of Hölder's inequality leads to 
$$\int_{\mathbb{R}^d} (\widetilde{h}_{\mathcal{X}}^{*}(\theta) - \widetilde{h}_{\mathcal{X}, \Lambda}^{*}(\theta)) (P(d\theta) - P'(d\theta)) \leq \|\widetilde{h}_{\mathcal{X}, \Lambda}^{*}- \widetilde{h}_{\mathcal{X}}^{*}\|_{\infty} \|P - P'\|_{1} \leq \frac{2(1 + 2RM_{\eta})M_{K}}{\Lambda},$$
where the final inequality is due to $\|P - P'\|_{1} \leq 2$ and the bound of $\|\widetilde{h}_{\mathcal{X}, \Lambda}^{*}- \widetilde{h}_{\mathcal{X}}^{*}\|_{\infty}$ in Equation~\eqref{eq:key_approximation_inequality_1}.\footnote{With a slight abuse of notation, $\|P-P'\|_1:=\sup_{A\in\mathscr{B}(\mathbb{R}^d)}|P(A) - P'(A)|$ denotes the total variation distance between probability measures $P$ and $P'$.} Hence, we obtain the following bound
\begin{align}
W_1(P, P') \leq \frac{2(1 + 2RM_{\eta})M_{K}}{\Lambda} + \int_{\mathbb{R}^d}  \widetilde{h}_{\mathcal{X}, \Lambda}^{*}(\theta) (P(d\theta) - P'(d\theta)). \label{eq:key_inequality_duality_1_2}
\end{align}

To apply the Fourier inversion theorem to $\int_{\mathbb{R}^d}  \widetilde{h}_{\mathcal{X}, \Lambda}^{*}(\theta) (P(d\theta) - P'(d\theta))$, we first need to demonstrate that $\widetilde{h}_{\mathcal{X}, \Lambda}^{*} \in L_{1}(\mathbb{R}^{d})$ and $\mathcal{F}[\widetilde{h}_{\mathcal{X}, \Lambda}^{*}] \in L_{1}(\mathbb{R}^{d})$. Because $\widetilde{h}_{\mathcal{X}, \Lambda}^{*}(x) = (\widetilde{h}_{\mathcal{X}}^{*} * K_\Lambda)(x)$, an application of Young's inequality leads to
$\|\widetilde{h}_{\mathcal{X}, \Lambda}^{*}\|_{1} \leq \|\widetilde{h}_{\mathcal{X}}^{*}\|_{1}\|K_\Lambda\|_{1} = \|\widetilde{h}_{\mathcal{X}}^{*}\|_{1} < \infty$ (thanks to Lemma~\ref{lemma:cutoff_function_properties}). As for $\mathcal{F}[\widetilde{h}_{\mathcal{X}, \Lambda}^{*}]$, by the properties of convolutions we obtain that
$\mathcal{F}[\widetilde{h}_{\mathcal{X}, \Lambda}^{*}](\xi) = \mathcal{F}[\widetilde{h}_{\mathcal{X}}^{*}](\xi) \mathcal{F}[K_\Lambda](\xi)$. As $\mathrm{supp}(\mathcal{F}[K_{\Lambda}]) \subseteq \bar{B}(0, \Lambda)$, we have $\mathrm{supp}(\mathcal{F}[\widetilde{h}_{\mathcal{X}, \Lambda}^{*}]) \subseteq \bar{B}(0, \Lambda)$. Furthermore, $\|\mathcal{F}[K_\Lambda]\|_{\infty} \leq 1$ and $\|\mathcal{F}[\widetilde{h}_{\mathcal{X}}^{*}]\|_{\infty} \leq \frac{\pi^{d/2}}{\Gamma(d/2 + 1)} (2R)^{d + 1}$ (see Lemma~\ref{lemma:cutoff_function_properties}). Therefore, $\|\mathcal{F}[\widetilde{h}_{\mathcal{X}, \Lambda}^{*}]\|_{\infty} \leq \|\mathcal{F}[\widetilde{h}_{\mathcal{X}}^{*}]\|_{\infty} \leq \frac{\pi^{d/2}}{\Gamma(d/2 + 1)} (2R)^{d + 1}$. Putting these results together, we obtain that $\mathcal{F}[\widetilde{h}_{\mathcal{X}, \Lambda}^{*}] \in L_{1}(\mathbb{R}^{d})$. 

Now, from the Fourier inversion theorem, as $\widetilde{h}_{\mathcal{X}, \Lambda}^{*} \in L_{1}(\mathbb{R}^{d})$ and $\mathcal{F}[\widetilde{h}_{\mathcal{X}, \Lambda}^{*}] \in L_{1}(\mathbb{R}^{d})$, we can write
$$\widetilde{h}_{\mathcal{X}, \Lambda}^{*}(x) = \frac{1}{(2\pi)^d} \int_{\mathbb{R}^d} \mathcal{F}[\widetilde{h}_{\mathcal{X}, \Lambda}^{*}](\xi) e^{i \xi^\top x} d\xi,$$
leading to 
$$\int_{\mathbb{R}^d}  \widetilde{h}_{\mathcal{X}, \Lambda}^{*}(\theta) (P(d\theta) - P'(d\theta)) = \frac{1}{(2\pi)^d} \int_{\mathbb{R}^d}  \left(\int_{\mathbb{R}^d} \mathcal{F}[\widetilde{h}_{\mathcal{X}, \Lambda}^{*}](\xi) e^{i \xi^\top \theta} d\xi\right) (P(d\theta) - P'(d\theta)).$$
Since $|\mathcal{F}[\widetilde{h}_{\mathcal{X}, \Lambda}^{*}]|$ and $P-P'$ have compact support, and $|\mathcal{F}[\widetilde{h}_{\mathcal{X}, \Lambda}^{*}]|$ is bounded, we have 
\begin{align*}
\int_{\mathbb{R}^d}\int_{\mathbb{R}^d}|\mathcal{F}[\widetilde{h}_{\mathcal{X}, \Lambda}^{*}](\xi)e^{i \xi^\top \theta}|d\xi d|P-P'| = \int_{\mathbb{R}^d}\int_{\mathbb{R}^d}|\mathcal{F}[\widetilde{h}_{\mathcal{X}, \Lambda}^{*}](\xi)|d\xi d|P-P'|  < \infty.
\end{align*}
Thus, by Fubini's theorem, we have that
\begin{align*}
    \int_{\mathbb{R}^d}  \widetilde{h}_{\mathcal{X}, \Lambda}^{*}(\theta) (P(d\theta) - P'(d\theta)) & = \frac{1}{(2\pi)^d} \int_{\mathbb{R}^d}  \left(\int_{\mathbb{R}^d} e^{i \xi^\top \theta} (P(d\theta) - P'(d\theta)) \right) \mathcal{F}[\widetilde{h}_{\mathcal{X}, \Lambda}^{*}](\xi) d\xi \\
    & = \frac{1}{(2\pi)^d} \int_{\mathbb{R}^d} \mathcal{F}[\widetilde{h}_{\mathcal{X}, \Lambda}^{*}](\xi) \Big( \Phi_P(\xi) - \Phi_{P'}(\xi) \Big) d\xi \\
    & = \frac{1}{(2\pi)^d} \int_{\bar{B}(0, \Lambda)} \mathcal{F}[\widetilde{h}_{\mathcal{X}, \Lambda}^{*}](\xi) \Big( \Phi_P(\xi) - \Phi_{P'}(\xi) \Big) d\xi,
\end{align*}
where the final identity is due to the fact that $\mathrm{supp}(\mathcal{F}[\widetilde{h}_{\mathcal{X}, \Lambda}^{*}]) \subseteq \bar{B}(0, \Lambda)$. 

Putting all of these results together, we arrive at
\begin{align}
    W_1(P, P') \leq \frac{2(1 + 2RM_{\eta})M_{K}}{\Lambda} + \frac{1}{(2\pi)^d} \int_{\bar{B}(0, \Lambda)} \mathcal{F}[\widetilde{h}_{\mathcal{X}, \Lambda}^{*}](\xi) \Big( \Phi_P(\xi) - \Phi_{P'}(\xi) \Big) d\xi. \label{eq:key_inequality_duality_2}
\end{align}
Next, since $p_G = P * f(\cdot\mid0,\Sigma)$ and, therefore, $\Phi_{p_{G}}(\xi) = \Phi_{P}(\xi) \Phi_{\Sigma}(\xi)$ for all $\xi \in \mathbb{R}^{d}$, we obtain that
\begin{align*}
    \Phi_P(\xi) - \Phi_{P'}(\xi) &= \frac{\Phi_{p_{G}}(\xi)}{\Phi_\Sigma(\xi)} - \frac{\Phi_{p_{G'}}(\xi)}{\Phi_{\Sigma'}(\xi)} \\
    &= \frac{\Phi_{p_{G}}(\xi) - \Phi_{p_{G'}}(\xi)}{\Phi_\Sigma(\xi)} + \Phi_{P'}(\xi) \left( \frac{\Phi_{\Sigma'}(\xi) - \Phi_\Sigma(\xi)}{\Phi_\Sigma(\xi) \Phi_{\Sigma'}(\xi)} \right) \Phi_{\Sigma'}(\xi),
\end{align*}
where we used the fact that both $\Phi_{\Sigma'}(\xi)$ and $\Phi_{\Sigma}(\xi)$ are non-zero for all $\xi\in\mathbb R^d$ by assumption.
Given that $|\Phi_{P'}(\xi)| \le 1$ for all $\xi \in \mathbb{R}^{d}$ and $\|\mathcal{F}[\widetilde{h}_{\mathcal{X}, \Lambda}^{*}]\|_{\infty} \leq \frac{\pi^{d/2}}{\Gamma(d/2 + 1)} (2R)^{d + 1}=:C(R,d)$, an application of the triangle inequality leads to
\begin{align}
    &\left|\int_{\bar{B}(0, \Lambda)} \mathcal{F}[\widetilde{h}_{\mathcal{X}, \Lambda}^{*}](\xi) \Big( \Phi_P(\xi) - \Phi_{P'}(\xi) \Big) d\xi\right|  \nonumber \\
    &\leq \frac{C(R,d)}{(2\pi)^{d}} \int_{\bar{B}(0, \Lambda)} \frac{|\Phi_{p_{G}}(\xi) - \Phi_{p_{G'}}(\xi)|}{|\Phi_\Sigma(\xi)|} d\xi  + \frac{C(R,d)}{(2\pi)^{d}}  \int_{\bar{B}(0, \Lambda)} \frac{|\Phi_{\Sigma'}(\xi) - \Phi_\Sigma(\xi)|}{|\Phi_\Sigma(\xi)|} d\xi. \label{eq:key_inequality_duality_3}
\end{align}
By construction, $|\Phi_{p_{G}}(\xi) - \Phi_{p_{G'}}(\xi)| \leq \|p_{G} - p_{G'}\|_{1}$ for any $\xi \in \mathbb{R}^{d}$. Furthermore, applying the super-smooth assumption, we have $\exp(-c_1 \|\xi\|^\alpha) \le |\Phi_\Sigma(\xi)|$. Therefore, we obtain that
\begin{align}
\frac{C(R,d)}{(2\pi)^{d}} \int_{\bar{B}(0, \Lambda)} \frac{|\Phi_{p_{G}}(\xi) - \Phi_{p_{G'}}(\xi)|}{|\Phi_\Sigma(\xi)|} d\xi  & \leq  \frac{C(R,d)}{(2\pi)^{d}} \|p_{G} - p_{G'}\|_{1} \int_{\bar{B}(0, \Lambda)} \exp(c_1 \|\xi\|^\alpha) d\xi \nonumber \\
& \leq C_{1} \|p_{G} - p_{G'}\|_{1} \Lambda^{d} \exp(c_{1} \Lambda^{\alpha}), \label{eq:key_inequality_duality_4}
\end{align}
where $C_{1} =\frac{C(R,d)}{(4\pi)^{d/2}   \Gamma(d/2+1)}$ is a constant depending on $R$ and $d$.

By Assumption~\ref{ass2}, we get
\begin{equation}
\label{eqn:dung_distance_phi_Sigma}
    |\Phi_{\Sigma'}(\xi) - \Phi_\Sigma(\xi)| \lesssim \Xi(\|\Sigma - \Sigma'\|_{\mathrm{op}}) \|\xi\|^p \max\big(|\Phi_\Sigma(\xi)|, |\Phi_{\Sigma'}(\xi)|\big) 
\end{equation}
and thus, by super-smoothness,
\begin{align}
    \int_{\bar{B}(0, \Lambda)} \frac{|\Phi_{\Sigma'}(\xi) - \Phi_\Sigma(\xi)|}{|\Phi_\Sigma(\xi)|} d\xi & \lesssim \Xi(\|\Sigma - \Sigma'\|_{\mathrm{op}}) \int_{\bar{B}(0, \Lambda)} \|\xi\|^p \left[ \frac{\max\big(|\Phi_\Sigma(\xi)|, |\Phi_{\Sigma'}(\xi)|\big)}{|\Phi_\Sigma(\xi)|} \right] d\xi \nonumber \\
    & \leq \Xi(\|\Sigma - \Sigma'\|_{\mathrm{op}}) \int_{\bar{B}(0, \Lambda)} \|\xi\|^p  \exp((c_{1} - c_{2}) \|\xi\|^\alpha)d\xi \nonumber \\
    & \leq \Xi(\|\Sigma - \Sigma'\|_{\mathrm{op}}) \exp((c_{1} - c_{2}) \Lambda^\alpha)\int_{\bar{B}(0, \Lambda)} \|\xi\|^p  d\xi \nonumber \\
    & = C' \Xi(\|\Sigma - \Sigma'\|_{\mathrm{op}}) \Lambda^{d + p} \exp((c_{1} - c_{2}) \Lambda^\alpha)\label{eq:key_inequality_duality_5} 
\end{align}
for some $C'$ depending on $d$. Plugging the bounds in Equations~\eqref{eq:key_inequality_duality_4} and \eqref{eq:key_inequality_duality_5} into the bound in Equation~\eqref{eq:key_inequality_duality_3} yields that
\begin{align}
    \left|\int_{\bar{B}(0, \Lambda)} \mathcal{F}[\widetilde{h}_{\mathcal{X}, \Lambda}^{*}](\xi) \Big( \Phi_P(\xi) - \Phi_{P'}(\xi) \Big) d\xi\right| &\leq C_{2} \|p_{G} - p_{G'}\|_{1} \Lambda^{d} \exp(c_{1} \Lambda^{\alpha})  \nonumber \\
    & + C_{3} \Xi(\|\Sigma - \Sigma'\|_{\mathrm{op}}) \Lambda^{d + p} \exp((c_{1} - c_{2}) \Lambda^\alpha), \label{eqn:dung_key_inequality_duality_6} 
\end{align}
where $C_{2}$ and $C_{3}$ are constants depending on $R$, $\lambda_{\min}$, $\lambda_{\max}$, and $d$.
Since $\|p_G - p_{G'}\|_1 \gtrsim\Psi(\|\Sigma - \Sigma'\|_{\mathrm{op}})$ and $\Psi$ is strictly increasing, we have that $\Psi^{-1}(\|p_G - p_{G'}\|_1/C_0) \geq \|\Sigma - \Sigma'\|_{\mathrm{op}}$ for some $C_0>0$. 

Combining Equations~\eqref{eq:key_inequality_duality_1_2} and~\eqref{eqn:dung_key_inequality_duality_6} and noting that $\Xi$ is also a strictly increasing function, we obtain that
\begin{align*}
     W_1(P, P') &\leq \frac{C_{1}}{\Lambda} + C_{2} \|p_{G} - p_{G'}\|_{1} \Lambda^{d} \exp(c_{1} \Lambda^{\alpha})\\
     &\hspace{2cm }+ C_{3} (\Xi \circ \Psi^{-1})(\|p_G - p_{G'}\|_1/C_0) \Lambda^{d + p} \exp((c_{1} - c_{2}) \Lambda^\alpha),
\end{align*}
where $C_1 := 2(1 + 2RM_{\eta})M_{K}$. Choosing now $\Lambda = 2C_{1}/W_{1}(P,P')$ we arrive at the inequality 
\begin{align}
\nonumber
     1 &\leq \Big[ C_2' V   W_1(P,P')^{-(d + 1)} + C_3' (\Xi \circ \Psi^{-1})(V/C_0) W_1(P,P')^{-(d+p+1)} \Big] \\
     &\times\exp\Big( C_{4}'W_1(P,P')^{-\alpha} \Big),
\label{eqn:dung_thm1_1_less_than_W1_exp}
\end{align}
where we set $V = \|p_{G} - p_{G'}\|_{1}$, $C'_2 = 2C_2(2C_1)^d$, $C'_3 = 2C_3(2C_1)^{d+p}$, and $C'_4 = c_1(2C_1)^\alpha$. Using Lemma \ref{lemma:exp_larger_polynomial} for $p_1 = (d+1)/\alpha$ and $p_2 = (d+p+1)/\alpha$, there exist $M_{p_1,1}$ and $M_{p_2,1}$ depending only on $p_1$ and $p_2$ such that
\begin{align*}
    W_1(P,P')^{-(d+1)} & \leq M_{p_1,1}\exp(W_1(P,P')^{-\alpha}), \\
    W_1(P,P')^{-(d+p+1)} & \leq M_{p_2,1}\exp(W_1(P,P')^{-\alpha}).
\end{align*}

Combining these results with \eqref{eqn:dung_thm1_1_less_than_W1_exp}, we have
\begin{align*}
    1&\leq C'_2 M_{p_1,1} V   \exp((1+C_4')W_1(P,P')^{-\alpha}) \\
    &+ C'_3 M_{p_2,1}(\Xi \circ \Psi^{-1})(V/C_0)  \exp((1+C_4')W_1(P,P')^{-\alpha})\\
    &\leq C_2'' V  \exp(C_{\text{new}}W_1(P,P')^{-\alpha}) + C_3''(\Xi \circ \Psi^{-1})(V/C_0)  \exp(C_{\text{new}}W_1(P,P')^{-\alpha}),
\end{align*}
where $C_2'' = C'_2M_{p_1,1}$, $C_3'' =C_3'M_{p_2,1}$, and $C_{\text{new}} = 1+C_4'$. This inequality implies either $1/2 \leq C_2''V  \exp(C_{\text{new}}W_1(P,P')^{-\alpha})$ or $1/2 \leq C_3''(\Xi \circ \Psi^{-1})(V/C_0)  \exp(C_{\text{new}}W_1(P,P')^{-\alpha})$. The first case implies 
\begin{equation*}
    V \geq \dfrac{1}{2C_2''}\exp\left(-C_{\text{new}}\dfrac{1}{W_1(P,P')^{\alpha}}\right),
\end{equation*}
while the second case implies 
\begin{equation*}
    V \geq C_0 (\Psi\circ \Xi^{-1})\left(\frac{1}{2C_3''}\exp\left(-C_{\text{new}}\frac{1}{W_{1}(P,P')^{\alpha}}\right)\right).
\end{equation*}
As a result, we have 
\begin{equation*}
    V \geq C_{1}\min\left(\exp\Big(-\frac{C}{W_1(P, P')^\alpha} \Big), \;(\Psi \circ \Xi^{-1})\Big(\widetilde C\exp\Big(-\frac{C}{W_1(P, P')^\alpha}\Big)\Big)\right),
\end{equation*}
where $C = C_{\text{new}}$, $C_1 = \min\{1,1/2C''_2, C_0\}$, and $\widetilde C = 1/2C_3''$. Combining with Assumption~\ref{ass1}, this completes the proof.

\subsection*{Proof of Theorem~\ref{theorem:ordinary_rate}}\label{proof_thm_ordinary_rate}

Following the same duality, extension, cutoff, and spectral bandlimiting arguments as in the proof of Theorem~\ref{theorem:supersmooth_rate}, we arrive at
\begin{align}
    W_1(P, P') \leq \frac{C_{1}}{\Lambda} + \frac{1}{(2\pi)^d} \int_{\bar{B}(0, \Lambda)} \mathcal{F}[\widetilde{h}_{\mathcal{X}, \Lambda}^{*}](\xi) \Big( \Phi_P(\xi) - \Phi_{P'}(\xi) \Big) d\xi, \label{eq:key_inequality_duality_ordinary_1}
\end{align}
where $\Lambda > 0$ is to be chosen later, $\widetilde{h}_{\mathcal{X}, \Lambda}^{*}$ is defined in the proof of Theorem~\ref{theorem:supersmooth_rate}, and $C_1 := 2(1 + 2RM_{\eta})M_{K}$. Using the same decomposition as in Equation~\eqref{eq:key_inequality_duality_3}, recalling that $\Phi_{p_G}(\xi) = \Phi_{P}(\xi)\Phi_{\Sigma}(\xi)$ and setting $C(R,d) := \frac{\pi^{d/2}}{\Gamma(d/2 + 1)} (2R)^{d + 1}$, we have
\begin{align}
    &\left|\int_{\bar{B}(0, \Lambda)} \mathcal{F}[\widetilde{h}_{\mathcal{X}, \Lambda}^{*}](\xi) \Big( \Phi_P(\xi) - \Phi_{P'}(\xi) \Big) d\xi\right| \nonumber \\
    &\quad \leq \frac{C(R,d)}{(2\pi)^{d}} \int_{\bar{B}(0, \Lambda)} \frac{|\Phi_{p_{G}}(\xi) - \Phi_{p_{G'}}(\xi)|}{|\Phi_\Sigma(\xi)|} d\xi  + \frac{C(R,d)}{(2\pi)^{d}}  \int_{\bar{B}(0, \Lambda)} \frac{|\Phi_{\Sigma'}(\xi) - \Phi_\Sigma(\xi)|}{|\Phi_\Sigma(\xi)|} d\xi. \label{eq:key_inequality_duality_ordinary_2}
\end{align}
We now bound each term separately.

As for the first term, since $|\Phi_{p_{G}}(\xi) - \Phi_{p_{G'}}(\xi)| \leq \|p_{G} - p_{G'}\|_{1}$ for all $\xi \in \mathbb{R}^{d}$, and the ordinary-smoothness assumption gives $|\Phi_{\Sigma}(\xi)| \geq c_1(1+\|\xi\|^{\beta})^{-1}$, an application of Lemma~\ref{lem:integral_inequality} yields
\begin{align}
\frac{C(R,d)}{(2\pi)^{d}} \int_{\bar{B}(0, \Lambda)} \frac{|\Phi_{p_{G}}(\xi) - \Phi_{p_{G'}}(\xi)|}{|\Phi_\Sigma(\xi)|} d\xi  & \leq  \frac{C(R,d)}{c_1(2\pi)^{d}} \|p_{G} - p_{G'}\|_{1} \int_{\bar{B}(0, \Lambda)} (1+\|\xi\|^{\beta}) d\xi \nonumber \\
& \leq C'_{2} \|p_{G} - p_{G'}\|_{1} \max\{\Lambda^{d+\beta},\Lambda^{d}\},
\label{eq:key_inequality_duality_ordinary_3}
\end{align}
where $C'_2$ is a constant depending on $R$, $d$, and $c_1$.

As for the second term, from Assumption~\ref{ass2},
$$|\Phi_{\Sigma'}(\xi) - \Phi_\Sigma(\xi)| \lesssim \Xi(\|\Sigma - \Sigma'\|_{\mathrm{op}}) \|\xi\|^p \max\big(|\Phi_\Sigma(\xi)|, |\Phi_{\Sigma'}(\xi)|\big).$$
Moreover, by ordinary-smoothness,  $\max\big(|\Phi_\Sigma(\xi)|, |\Phi_{\Sigma'}(\xi)|\big) \leq c_2(1+\|\xi\|^{\beta})^{-1}$, so that
\begin{equation*}
    \frac{\max\big(|\Phi_\Sigma(\xi)|, |\Phi_{\Sigma'}(\xi)|\big)}{|\Phi_\Sigma(\xi)|} \leq \frac{c_2(1+\|\xi\|^{\beta})^{-1}}{c_1(1+\|\xi\|^{\beta})^{-1}} = \frac{c_2}{c_1}.
\end{equation*}
Thus, another application of Lemma~\ref{lem:integral_inequality} yields
\begin{align}
    \int_{\bar{B}(0, \Lambda)} \frac{|\Phi_{\Sigma'}(\xi) - \Phi_\Sigma(\xi)|}{|\Phi_\Sigma(\xi)|} d\xi
    & \lesssim \frac{c_2}{c_1}\Xi(\|\Sigma - \Sigma'\|_{\mathrm{op}}) \int_{\bar{B}(0, \Lambda)} \|\xi\|^p \, d\xi \nonumber \\
    & = C'_3\, \Xi(\|\Sigma-\Sigma'\|_{\mathrm{op}})\Lambda^{d+p},
    \label{eq:key_inequality_duality_ordinary_4}
\end{align}
where $C'_3$ is a constant depending on $R$, $d$, $\lambda_{\min}$, $\lambda_{\max}$, and $p$. Substituting Equations~\eqref{eq:key_inequality_duality_ordinary_3} and \eqref{eq:key_inequality_duality_ordinary_4} into Equation~\eqref{eq:key_inequality_duality_ordinary_2}, and in turn into \eqref{eq:key_inequality_duality_ordinary_1}, we arrive at the bound
\begin{align}
    W_1(P,P') &\leq \frac{C_1}{\Lambda} + C'_{2} \|p_{G} - p_{G'}\|_{1} \max\{\Lambda^{d+\beta},\Lambda^{d}\} \nonumber\\
    &+ C'_{3}\,(\Xi \circ\Psi^{-1})(\|p_G - p_{G'}\|_1/C_0)\, \Lambda^{d + p}, \label{eq:key_inequality_duality_ordinary_5}
\end{align}
where we used $\Xi(\|\Sigma-\Sigma'\|_{\mathrm{op}}) \leq (\Xi\circ\Psi^{-1})(\|p_G-p_{G'}\|_1/C_0)$, which follows from $\|p_G - p_{G'}\|_1 \geq C_0\,\Psi(\|\Sigma - \Sigma'\|_{\mathrm{op}})$ for some $C_0$ (by Assumption~\ref{ass1}) and the fact that both $\Psi$ and $\Xi$ are strictly increasing. We now distinguish two cases, based on the size of $W_1(P,P')$.

\paragraph*{Case 1: $W_1(P,P') \geq 2C_1$.} Setting $\Lambda = 1$ in Equation~\eqref{eq:key_inequality_duality_ordinary_5} and using $W_1(P,P') - C_1 \geq W_1(P,P')/2$, we get
\begin{equation*}
    W_1(P,P') \leq 2C'_2\|p_{G} - p_{G'}\|_{1} + 2C_3'\,(\Xi \circ\Psi^{-1})(\|p_G - p_{G'}\|_1/C_0).
\end{equation*}
Hence either $W_1(P,P')\leq 4C'_2\|p_{G} - p_{G'}\|_{1}$ or $W_1(P,P') \leq 4C_3'\,(\Xi \circ\Psi^{-1})(\|p_G - p_{G'}\|_1/C_0)$. Since $\mathrm{supp}(P), \mathrm{supp}(P')\subseteq \bar{B}(0,R)$ implies $W_1(P,P') \leq 2R$, we have $W_1(P,P') \geq (2R)^{-(d+\beta)} W_1(P,P')^{d+\beta+1}$ and similarly with $d+p$ in place of $d+\beta$. Combining these observations, there exist constants $C'_{\mathrm{I}}$ and $c'_{\mathrm{I}}$ depending on $R$, $d$, $\lambda_{\min}$, $\lambda_{\max}$, and $p$ such that
\begin{align}
    \|p_{G} - p_{G'}\|_{1} \geq C'_{\mathrm{I}}\min\Big\{W_1(P,P')^{d+\beta+1},\,(\Psi\circ \Xi^{-1})\big(c'_{\mathrm{I}}\,W_1(P,P')^{d+p+1}\big)\Big\}.
    \label{eq:ordinary_case1}
\end{align}

\paragraph*{Case 2: $W_1(P,P') < 2C_1$.} Setting $\Lambda = 2C_1/W_1(P,P')$ in Equation~\eqref{eq:key_inequality_duality_ordinary_5} gives
\begin{align*}
    1 &\leq C'_{2,\mathrm{II}}\, \|p_{G} - p_{G'}\|_{1} W_1(P,P')^{-(d+\beta + 1)} \\
    &+ C'_{3,\mathrm{II}}\,(\Xi \circ\Psi^{-1})(\|p_G-p_{G'}\|_1/C_0) W_1(P,P')^{-(d+p+1)},
\end{align*}
where $C'_{2,\mathrm{II}} = 2C'_2(2C_1)^{d+\beta}$ and $C'_{3,\mathrm{II}} = 2C'_3(2C_1)^{d+p}$. Hence either the first or second summand is at least $1/2$, yielding respectively $\|p_G-p_{G'}\|_1\geq  W_1(P,P')^{d+\beta + 1}/2C'_{2,\mathrm{II}}$ or $\|p_G-p_{G'}\|_1\geq C_0\,(\Psi \circ\Xi^{-1})(W_1(P,P')^{d+p+1}/2C'_{3,\mathrm{II}})$. Combining, there exists a constant $C>0$ such that
\begin{align}
    \|p_G - p_{G'}\|_1 \gtrsim\min\Big\{W_1(P,P')^{d+\beta+1},\,(\Psi\circ \Xi^{-1})\big(CW_1(P,P')^{d+p+1}\big)\Big\}.
    \label{eq:ordinary_case2}
\end{align}

\medskip
Combining \eqref{eq:ordinary_case1} and \eqref{eq:ordinary_case2} with Assumption~\ref{ass1}, we conclude that
\begin{equation*}
    \|p_G-p_{G'}\|_1 \gtrsim \min\Big\{W_1(P,P')^{d+\beta+1},\,(\Psi \circ\Xi^{-1})\big(C\,W_1(P,P')^{d+p+1}\big)\Big\} + \Psi(\|\Sigma-\Sigma'\|_{\mathrm{op}}),
\end{equation*}
where the implicit constant and $C$ depend on $R$, $\lambda_{\min}$, $\lambda_{\max}$, and $d$. This completes the proof.

\subsection*{Proof of Proposition~\ref{pro:PDE_implies_ordinary}}
\label{proof_pro_PDE_implies_ordinary}
Let $P = \delta_0$, so that $G = \delta_0 \times \delta_\Sigma \in \mathscr{G}$, and the mixture density $p_G$ corresponds to
$$ p_G(x) = \int_{\mathbb{R}^d} f(x \mid \theta, \Sigma) \,\delta_0(d\theta) = f(x \mid 0, \Sigma), \quad x \in \mathbb{R}^d.$$
By the PDE inversion condition, we have $\mathcal{T}_{\Sigma}f(\cdot\mid 0,\Sigma) = \delta_{0}$. Applying the Fourier transform to both sides yields $\mathcal{F}[\mathcal{T}_{\Sigma}f(\cdot\mid 0,\Sigma)](\xi) = 1$, for all $\xi \in \mathbb{R}^d$. Consequently, we have, for all $\xi \in \mathbb{R}^d$, that
\begin{align*}
    1 &= \mathcal{F}[\mathcal{T}_{\Sigma}f(\cdot\mid 0,\Sigma)](\xi)\\ 
    & = \left(\sum_{\nu\in\mathbb N^d\,:\,0\leq |\nu| \leq \beta} c_{\nu}(\Sigma)(-i\xi)^{\nu}\right)\mathcal{F}[f(\cdot\mid 0,\Sigma)](\xi) \\
    &= \left(\sum_{\nu\in\mathbb N^d\,:\,0\leq |\nu| \leq \beta} c_{\nu}(\Sigma)(-i\xi)^{\nu}\right)\Phi_{\Sigma}(-\xi).
\end{align*}
Let $Q_\Sigma$ denote the characteristic polynomial of the differential operator $\mathcal{T}_\Sigma$, namely,
\begin{equation}
\label{eqn:dung_lemma_PDE_case_of_ordinary_smooth}
    Q_\Sigma(\xi) := \sum_{\nu\in\mathbb N^d\,:\,0\leq |\nu| \leq \beta} c_\nu(\Sigma) (-i\xi)^\nu, \quad \xi \in \mathbb{R}^d.
\end{equation}
Noting that $\Phi_{\Sigma}(\cdot)$ is an even function because, by assumption, $f(\cdot  \mid 0,\Sigma)$ is  even, we conclude that  $Q_\Sigma(\xi) \Phi_\Sigma(\xi) = 1$ for all $\xi \in \mathbb{R}^d$, which in turn implies that $|\Phi_\Sigma(\xi)| = |Q_\Sigma(\xi)|^{-1}$ for all $\xi \in \mathbb{R}^d$.

We first prove the lower bound $|\Phi_\Sigma(\xi)| \ge c_1(1 + \|\xi\|^\beta)^{-1}$ for some $c_1>0$, which is equivalent to showing $|Q_\Sigma(\xi)| \le C_{\text{upper}}(1 + \|\xi\|^\beta)$ for some $C_{\text{upper}}>0$. An application of the triangle inequality to the polynomial $Q_\Sigma(\xi)$ leads to 
$$ |Q_\Sigma(\xi)| \le \sum_{\nu\in\mathbb N^d\,:\,0\leq |\nu| \leq \beta} |c_\nu(\Sigma)|\cdot \|\xi\|^{|\nu|}. $$
Given the PDE inversion condition, the coefficients $c_\nu(\Sigma)$ are continuous mappings from the set of positive-definite matrices $\mathbb{S}_d^+(\lambda_{\min}, \lambda_{\max})$ to $\mathbb{R}$. Because $\mathbb{S}_d^+(\lambda_{\min}, \lambda_{\max})$ is a closed and bounded compact metric space, any continuous function defined on it is globally bounded. Thus, there exists a constant $M > 0$, depending only on $\lambda_{\min}$ and $\lambda_{\max}$, such that $\sup_{\Sigma} |c_\nu(\Sigma)| \le M$ for all $\nu\in\mathbb N^d$.

Furthermore, for any integer degree $0 \le |\nu| \le \beta$, we have
$$ \|\xi\|^{|\nu|} \le 1 + \|\xi\|^\beta \quad \text{for all } \xi \in \mathbb{R}^d. $$
Therefore, we obtain that
$$ |Q_\Sigma(\xi)| \le M \sum_{\nu\in\mathbb N^d\,:\,0\leq |\nu| \leq \beta} (1 + \|\xi\|^\beta) = M \binom{\beta+d}{d} (1 + \|\xi\|^\beta) =: C_{\text{upper}} (1 + \|\xi\|^\beta).$$

Next, we prove the upper bound $|\Phi_\Sigma(\xi)| \le c_2(1 + \|\xi\|^\beta)^{-1}$ for some $c_2>0$, which is equivalent to proving $|Q_\Sigma(\xi)| \ge C_{\text{lower}}(1 + \|\xi\|^\beta)$ for some $C_{\text{lower}}>0$. As $|\Phi_\Sigma(\xi)| \le 1$ for all $\xi \in \mathbb{R}^d$, it follows that
$$|Q_\Sigma(\xi)| = |\Phi_\Sigma(\xi)|^{-1} \ge 1.$$
Recall that $f(x \mid 0, \Sigma) = |\Sigma|^{-1/2} g(x^\top \Sigma^{-1} x)$. Using the change of variables $y = \Sigma^{-1/2}x$, we have that
$$\Phi_\Sigma(\xi) = \int_{\mathbb{R}^d} e^{i\xi^\top (\Sigma^{1/2} y)} g(y^\top y) \, dy = \int_{\mathbb{R}^d} e^{i (\Sigma^{1/2}\xi)^\top y} g(\|y\|^2) \, dy.$$
Define $I(v) := \int_{\mathbb{R}^d} e^{i v^\top y} g(\|y\|^2) \, dy$, where $v = \Sigma^{1/2}\xi$. For any orthogonal matrix $O$, we have
$$I(Ov) = \int_{\mathbb{R}^d} e^{i (Ov)^\top y} g(\|y\|^2) \, dy = \int_{\mathbb{R}^d} e^{i v^\top (O^\top y)} g(\|y\|^2) \, dy = I(v).$$
This rotational invariance means there exists a function $\psi$ such that $I(v) = \psi(\|v\|^2)$. Equivalently,
$$ \Phi_\Sigma(\xi) = \psi\big( \|\Sigma^{1/2}\xi\|^2 \big) = \psi(\xi^\top \Sigma \xi).$$
Therefore, we obtain $Q_\Sigma(\xi) = 1 / \psi(\xi^\top \Sigma \xi)$. By setting $y = \Sigma^{1/2}\xi$, it is clear that the function $\tilde{Q}(y) := Q_\Sigma(\Sigma^{-1/2}y) = 1/\psi(\|y\|^2)$ is a multivariate polynomial in $y$ that is invariant under any orthogonal rotation matrix $O$, since $\|Oy\|^2 = \|y\|^2$. Hence, it is a polynomial strictly in $\|y\|^2$. Denoting $\beta = 2m$, we can represent the polynomial $Q_{\Sigma}(\xi)$ as follows:
$$ Q_\Sigma(\xi) = \sum_{k=0}^{m} a_k(\Sigma) (\xi^\top \Sigma \xi)^k,$$ 
where all odd-degree coefficients disappear. From the formulation of $Q_{\Sigma}$ in Equation~\eqref{eqn:dung_lemma_PDE_case_of_ordinary_smooth}, we have $a_k(\Sigma) \in \mathbb{R}$ for $0\leq k \leq m$. By the first assumption in the PDE inversion condition, $a_m(\Sigma) \neq 0$. Furthermore, the fact that $|Q_{\Sigma}(\xi)| \geq 1$ for all $\xi$ implies that the leading coefficient must be strictly positive, $a_m(\Sigma) > 0$. 

Since the original coefficients $c_\nu(\Sigma)$ are continuous on the compact metric space $\mathbb{S}_d^+(\lambda_{\min}, \lambda_{\max})$, the derived polynomial coefficients $a_k(\Sigma)$ are also continuous on this compact set. Thus, the leading coefficient attains a strictly positive minimum $a_m^* := \min_{\Sigma} a_m(\Sigma) > 0$, and the remaining coefficients attain finite maximum magnitudes $A_k^* := \max_{\Sigma} |a_k(\Sigma)| < \infty$. We now establish a uniform lower bound for $|Q_{\Sigma}|$. An application of the reverse triangle inequality leads to
$$|Q_\Sigma(\xi)| \ge a_{m}(\Sigma) \left|\xi^\top \Sigma \xi\right|^{m} - \sum_{k=0}^{m - 1} |a_k(\Sigma)| |\xi^\top \Sigma \xi|^k.$$
Since $\lambda_{\min} \|\xi\|^2 \leq \xi^{\top} \Sigma \xi \leq \lambda_{\max} \|\xi\|^2$, it holds that
$$|Q_\Sigma(\xi)| \geq a_m^* \lambda_{\min}^{m} \|\xi\|^{2m} - \sum_{k = 0}^{m - 1} A_k^* \lambda_{\max}^{k} \|\xi\|^{2k}.$$
Hence, for $\|\xi\|\geq 1$, we have
$$ |Q_\Sigma(\xi)| \ge M_{1} \|\xi\|^\beta - M_{2} \|\xi\|^{\beta-2} = \|\xi\|^\beta \left(M_{1} - \frac{M_{2}}{\|\xi\|^2} \right),$$
where $M_{1} := a_m^* \lambda_{\min}^{m}$ and $M_{2} := \sum_{k = 0}^{m - 1} A_k^* \lambda_{\max}^{k}$ are strictly positive constants independent of $\Sigma$. By choosing $\|\xi\|  \ge \max\left\{\sqrt{2M_2/M_1},1\right\}$, we obtain that 
$$|Q_\Sigma(\xi)| \geq \frac{M_{1}}{2} \|\xi\|^{\beta}.$$
Combining this bound with the fact that $|Q_{\Sigma}(\xi)| \geq 1$ for all $\xi$, we arrive at
$$ |Q_\Sigma(\xi)| \ge C_{\text{lower}}(1 + \|\xi\|^\beta),$$
for a sufficiently small constant $C_{\text{lower}} > 0$ depending only on $\lambda_{\min}$, $\lambda_{\max}$, and the operator $\mathcal{T}_\Sigma$.
Putting these results together leads to the conclusion that the kernel $f$ is an ordinary-smooth function of order $\beta$.

\subsection*{Proof of Theorem~\ref{theorem:sharpen_ordinary_rate}}
\label{proof_thm_sharpen_ordinary_rate}

Let $\bar{h}^*_{\mathcal{X}}$ be the $1$-Lipschitz continuous function achieving the Kantorovich--Rubinstein supremum, constructed as in the proof of Theorem~\ref{theorem:supersmooth_rate}, and set $\widetilde{h}_{\mathcal{X}}^{*}(x) := \bar{h}_{\mathcal{X}}^{*}(x)\eta(x)$, with $\eta$ denoting the cutoff function constructed in Lemma~\ref{lemma:cutoff_function}. Rather than convolving with the Jackson kernel, we convolve $\widetilde{h}_{\mathcal{X}}^{*}$ with a scaled normalized mollifier in order to obtain a smooth proxy amenable to differentiation under the integral. Concretely, let $\psi$ be the mollifier defined in Lemma~\ref{lemma:cutoff_function}, and set $\psi_{\Lambda}(x) := \Lambda^d\psi(\Lambda x)$ for $\Lambda>0$.\footnote{Compared to Lemma~\ref{lemma:cutoff_function}, this is a slight abuse of notation, as in the Lemma $\psi_\epsilon(x)$ is defined as $\epsilon^{-d}\psi(x/\epsilon)$.} Define
$$\widetilde{h}_{\mathcal{X}, \Lambda}^{*}(x) := (\widetilde{h}_{\mathcal{X}}^{*} * \psi_\Lambda)(x) = \int_{\mathbb{R}^d} \widetilde{h}_{\mathcal{X}}^{*}(x - y)\, \psi_\Lambda(y)\, dy.$$
Since $\widetilde{h}_{\mathcal{X}}^{*}$ has support in $\bar{B}(0,2R)$ and $\psi_\Lambda$ has support in $\bar{B}(0,1/\Lambda)$, we have $\mathrm{supp}(\widetilde{h}_{\mathcal{X},\Lambda}^{*})\subseteq \bar{B}(0,2R+1/\Lambda)$. We split the rest of the proof into several steps.

\subsubsection*{Approximation bound} Direct computation gives
\begin{align*}
    \widetilde{h}_{\mathcal{X}, \Lambda}^{*}(x) - \widetilde{h}_{\mathcal{X}}^{*}(x) = \int_{\mathbb{R}^d} \big[\widetilde{h}_{\mathcal{X}}^{*}(x - y) - \widetilde{h}_{\mathcal{X}}^{*}(x)\big] \psi_\Lambda(y)\, dy, \quad x \in \mathbb{R}^d.
\end{align*}
Since $\widetilde{h}_{\mathcal{X}}^{*}$ is globally $(1+2RM_{\eta})$-Lipschitz continuous (Lemma~\ref{lemma:cutoff_function_properties}), for any $x \in \mathbb{R}^{d}$ we have
\begin{align*}
    |\widetilde{h}_{\mathcal{X}, \Lambda}^{*}(x) - \widetilde{h}_{\mathcal{X}}^{*}(x)| \leq (1+2RM_{\eta}) \int_{\mathbb{R}^d} \|y\| \Lambda^{d} \psi(\Lambda y)\, dy = \frac{(1 + 2RM_{\eta})M'_{\psi}}{\Lambda},
\end{align*}
where $M'_{\psi} := \int_{\mathbb{R}^d} \|y\|\, \psi(y)\, dy < \infty$. Hence,
\begin{align}
    \|\widetilde{h}_{\mathcal{X}, \Lambda}^{*}- \widetilde{h}_{\mathcal{X}}^{*}\|_{\infty} \leq \frac{(1 + 2RM_{\eta})M'_{\psi}}{\Lambda}. \label{eq:key_approximation_inequality_thm3}
\end{align}
Applying the same Hölder argument as in the proof of Theorem~\ref{theorem:supersmooth_rate} and setting $C_1' := 2(1 + 2RM_{\eta})M'_{\psi}$, we obtain
\begin{align}
W_1(P, P') \leq \frac{C_1'}{\Lambda} + \int_{\mathbb{R}^d}  \widetilde{h}_{\mathcal{X}, \Lambda}^{*}(\theta)\, (P(d\theta) - P'(d\theta)). \label{eq:key_inequality_thm3}
\end{align}

\subsubsection*{PDE inversion} Since $f$ satisfies the PDE inversion condition of order $\beta$, we have $\mathcal{T}_\Sigma p_G = P$ and $\mathcal{T}_{\Sigma'} p_{G'} = P'$ in the sense of distributions.\footnote{For standard references on distribution theory, see \citep{schwartz1966theorie,rudin1991functional,folland1999real}.} Therefore,
\begin{align*}
\int_{\mathbb{R}^d}  \widetilde{h}_{\mathcal{X}, \Lambda}^{*}(\theta)\,(P(d\theta) - P'(d\theta))
&= \langle \widetilde{h}_{\mathcal{X}, \Lambda}^{*},\, \mathcal{T}_\Sigma p_G - \mathcal{T}_{\Sigma'} p_{G'} \rangle \\
&= \langle \widetilde{h}_{\mathcal{X}, \Lambda}^{*},\, \mathcal{T}_{\Sigma}(p_G - p_{G'}) \rangle + \langle \widetilde{h}_{\mathcal{X}, \Lambda}^{*},\, (\mathcal{T}_{\Sigma} - \mathcal{T}_{\Sigma'}) p_{G'} \rangle \\
&= \langle \mathcal{T}^*_{\Sigma}\widetilde{h}_{\mathcal{X}, \Lambda}^{*},\, p_G - p_{G'} \rangle + \langle (\mathcal{T}^*_{\Sigma} - \mathcal{T}^*_{\Sigma'})\widetilde{h}_{\mathcal{X}, \Lambda}^{*},\, p_{G'}\rangle \\
&=: I_1 + I_2,
\end{align*}
where the use of the adjoint operators is justified by the fact that $\widetilde{h}_{\mathcal{X},\Lambda}^* \in C_c^\infty(\mathbb{R}^d)$.

\subsubsection*{Bound on $I_1$} By Hölder's inequality,
\begin{equation}
    |I_1| \leq \|p_{G}-p_{G'}\|_1\,\|\mathcal{T}_\Sigma^{*} \widetilde{h}_{\mathcal{X}, \Lambda}^{*}\|_\infty.
    \label{eq:I1_prelim}
\end{equation}
We bound $\|\mathcal{T}_\Sigma^{*} \widetilde{h}_{\mathcal{X}, \Lambda}^{*}\|_\infty$ by the triangle inequality:
\begin{align*}
    \|\mathcal{T}_\Sigma^*\widetilde{h}_{\mathcal{X}, \Lambda}^{*}\|_\infty \leq |c_0(\Sigma)|\,\|\widetilde{h}_{\mathcal{X}, \Lambda}^{*}\|_\infty + \sum_{\nu\in\mathbb N^d\, :\, 1 \le |\nu| \le \beta} |c_\nu(\Sigma)|\,\|\partial^\nu \widetilde{h}_{\mathcal{X}, \Lambda}^{*}\|_\infty.
\end{align*}
Applying the Fourier transform to both sides of $\mathcal{T}_\Sigma p_G = P$ and evaluating at $\xi=0$ (where $\Phi_{p_G}(0) = \Phi_P(0) = 1$) gives $c_0(\Sigma) \equiv 1$. By Young's inequality and $\|\widetilde{h}_{\mathcal{X}}^{*}\|_\infty \leq 2R$ (see the proof of Lemma~\ref{lemma:cutoff_function_properties}), we have $\|\widetilde{h}_{\mathcal{X},\Lambda}^{*}\|_\infty \leq 2R$. For the higher-order terms, we claim that $\|\partial^\nu \widetilde{h}_{\mathcal{X}, \Lambda}^{*}\|_\infty \leq C_{\mathrm{der}} \Lambda^{|\nu|-1}$ for all $\nu\in\mathbb N^d$ such that $1 \leq |\nu| \leq \beta$ and some $C_{\mathrm{der}}<\infty$. Indeed, fixing any canonical direction $e_m$ with $\nu_m \geq 1$ and writing $\nu' = \nu - e_m$, the convolution structure gives
$$\partial^\nu \widetilde{h}_{\mathcal{X}, \Lambda}^{*} = (\partial^{e_m} \widetilde{h}_{\mathcal{X}}^{*}) * (\partial^{\nu'} \psi_\Lambda).$$
By Young's inequality and $\|\nabla \widetilde{h}_{\mathcal{X}}^{*}\|_\infty \leq 1 + 2RM_{\eta}$ (Lemma~\ref{lemma:cutoff_function_properties}),
$$\|\partial^\nu \widetilde{h}_{\mathcal{X}, \Lambda}^{*}\|_\infty \leq (1+2RM_{\eta})\|\partial^{\nu'} \psi_\Lambda\|_1.$$
A change of variables and the chain rule give $\|\partial^{\nu'}\psi_\Lambda\|_1 = \Lambda^{|\nu'|}\|\partial^{\nu'}\psi\|_1 \leq C_{\psi,\beta}\Lambda^{|\nu|-1}$, where $C_{\psi,\beta}:=\max_{\nu\in\mathbb N^d : 1\leq|\nu'|\leq\beta}\|\partial^{\nu'}\psi\|_1 < \infty$. Hence $C_{\mathrm{der}} := C_{\psi,\beta}(1+2RM_\eta)$.

Since the eigenvalues of $\Sigma$ lie in $[\lambda_{\min},\lambda_{\max}]$ and the coefficients $c_\nu(\Sigma)$ depend smoothly on $\Sigma$, we have $\max_{\nu\in\mathbb N^d :1\leq |\nu|\leq\beta}|c_\nu(\Sigma)| \leq M_\beta<\infty$ (where $M_\beta$ is independent of $\Sigma$). Setting $C_1'' := C_{\mathrm{der}}M_\beta\binom{d+\beta}{d} + 2R$, for $\Lambda \geq 1$ we obtain
$$\|\mathcal{T}_\Sigma^* \widetilde{h}_{\mathcal{X}, \Lambda}^{*}\|_\infty \leq C_1''(1 + \Lambda^{\beta - 1}),$$
and therefore
\begin{equation}
    |I_1| \leq C_1''(1 + \Lambda^{\beta - 1})\|p_G - p_{G'}\|_1.
    \label{eq:I1_bound}
\end{equation}

\subsubsection*{Bound on $I_2$} By Hölder's inequality,
\begin{equation}
    |I_2| \leq \|(\mathcal{T}_\Sigma^* - \mathcal{T}_{\Sigma'}^*) \widetilde{h}_{\mathcal{X}, \Lambda}^{*}\|_\infty.
    \label{eq:I2_prelim}
\end{equation}
The Lipschitz continuity of the coefficients $\Sigma\mapsto c_\nu(\Sigma)$ gives $|c_\nu(\Sigma) - c_\nu(\Sigma')| \leq L\|\Sigma - \Sigma'\|_{\mathrm{op}}$ uniformly in $\nu$. Combining with the derivative bound obtained before and setting $C_2'' := CL\binom{d+\beta}{d}$, we get
\begin{align}
    \|(\mathcal{T}_\Sigma^* - \mathcal{T}_{\Sigma'}^*) \widetilde{h}_{\mathcal{X}, \Lambda}^{*}\|_\infty &\leq \sum_{\nu\in\mathbb R^d\,:\,1 \le |\nu| \le \beta} |c_\nu(\Sigma) - c_\nu(\Sigma')|\,\|\partial^\nu \widetilde{h}_{\mathcal{X}, \Lambda}^{*}\|_\infty \nonumber\\
    &\leq C_2''\|\Sigma - \Sigma'\|_{\mathrm{op}}(1+\Lambda^{\beta - 1}),
    \label{eq:I2_bound}
\end{align}
and therefore $|I_2| \leq C_2''\|\Sigma - \Sigma'\|_{\mathrm{op}}(1+\Lambda^{\beta-1})$.

\subsubsection*{Conclusion} Combining Equations~\eqref{eq:key_inequality_thm3}, \eqref{eq:I1_bound}, and \eqref{eq:I2_bound}, and using $\|\Sigma-\Sigma'\|_{\mathrm{op}} \leq \Psi^{-1}(\|p_G-p_{G'}\|_1/C_0)$ for some $C_0>0$ (by Assumption~\ref{ass1}), we obtain, for $\Lambda \geq 1$,
\begin{align}
    W_1(P, P') \leq \frac{C_1'}{\Lambda} + C_1''(1 + \Lambda^{\beta - 1})\|p_G - p_{G'}\|_1 + C_2''\Psi^{-1}(\|p_G - p_{G'}\|_1/C_0)(1+ \Lambda^{\beta - 1}). \label{eq:W1_thm3_combined}
\end{align}
We now distinguish two cases based on the size of $W_1(P,P')$.

\paragraph*{Case 1: $W_1(P,P') \geq 2C_1'$.} Setting $\Lambda=1$ in Equation~\eqref{eq:W1_thm3_combined} and using $W_1(P,P') - C_1' \geq W_1(P,P')/2$, we get
\begin{equation*}
    W_1(P, P') \leq 4C_1''\|p_G - p_{G'}\|_1 + 4C_2''\Psi^{-1}(\|p_G - p_{G'}\|_1/C_0).
\end{equation*}
Hence either $W_1(P,P') \leq 8C_1''\|p_G-p_{G'}\|_1$ or $W_1(P,P') \leq 8C_2''\Psi^{-1}(\|p_G-p_{G'}\|_1/C_0)$. Since $\mathrm{supp}(P),\mathrm{supp}(P')\subseteq\bar{B}(0,R)$ implies $W_1(P,P')\leq 2R$, we have $W_1(P,P') \leq (2R)^{1-\beta}W_1(P,P')^\beta$. Combining, there exist constants $C_{\mathrm{I}}''$ and $c_{\mathrm{I}}''$ such that
\begin{align}
    \|p_G - p_{G'}\|_1 \geq C_{\mathrm{I}}''\min\Big\{W_1(P,P')^{\beta},\,\Psi\big(c_{\mathrm{I}}''\,W_1(P,P')^{\beta}\big)\Big\}.
    \label{eq:thm3_case1}
\end{align}

\paragraph*{Case 2: $W_1(P,P') < 2C_1'$.} Setting $\Lambda = 2C_1'/W_1(P,P') > 1$ in Equation~\eqref{eq:W1_thm3_combined} and letting $C_{1,\mathrm{II}}'' := 4C_1''(2C_1')^{\beta-1}$ and $C_{2,\mathrm{II}}'' := 4C_2''(2C_1')^{\beta-1}$, we obtain
\begin{equation*}
    1 \leq C_{1,\mathrm{II}}'' W_1(P,P')^{-\beta}\|p_G-p_{G'}\|_1 + C_{2,\mathrm{II}}'' W_1(P,P')^{-\beta}\Psi^{-1}(\|p_G-p_{G'}\|_1/C_0).
\end{equation*}
Hence either the first or the second summand is at least $1/2$, yielding respectively $\|p_G-p_{G'}\|_1 \geq W_1(P,P')^\beta / (2C_{1,\mathrm{II}}'')$ or $\|p_G-p_{G'}\|_1 \geq C_0\,\Psi(W_1(P,P')^\beta/(2C_{2,\mathrm{II}}''))$. Combining, there exist constants $C_{\mathrm{II}}''$ and $c_{\mathrm{II}}''$ such that
\begin{align}
    \|p_G - p_{G'}\|_1 \geq C_{\mathrm{II}}''\min\Big\{W_1(P,P')^{\beta},\,\Psi\big(c_{\mathrm{II}}''\,W_1(P,P')^{\beta}\big)\Big\}.
    \label{eq:thm3_case2}
\end{align}

\medskip
Combining Equations~\eqref{eq:thm3_case1} and \eqref{eq:thm3_case2} with Assumption~\ref{ass1}, we conclude that
\begin{equation*}
    \|p_G-p_{G'}\|_1 \gtrsim \min\Big\{W_1(P,P')^{\beta},\,\Psi\big(C\,W_1(P,P')^{\beta}\big)\Big\} + \Psi(\|\Sigma-\Sigma'\|_{\mathrm{op}}),
\end{equation*}
where the implicit constant and $C$ depend on $R$, $\lambda_{\min}$, $\lambda_{\max}$, and $d$. This completes the proof.

\subsection*{Proof of Theorem~\ref{thm:inequality_gaussian}}
\label{proof_thm_inequality_gaussian}
We split the proof in four parts, as follows.

\paragraph*{Verification of Assumption~\ref{ass2}.}

The characteristic function of the Gaussian kernel is $\Phi_\Sigma(\xi) = \exp(-\xi^\top\Sigma\xi/2)$. Setting $\psi(t):=e^{-t/2}$, so that $\psi'(t) = -\psi(t)/2$, the mean value theorem gives a $\tau\in(0,1)$ such that
$$\Phi_{\Sigma'}(\xi) - \Phi_\Sigma(\xi) = \psi(\xi^\top\Sigma'\xi) - \psi(\xi^\top\Sigma\xi) = \psi'(\xi^\top\Sigma(\tau)\xi)\cdot\xi^\top(\Sigma'-\Sigma)\xi,$$
where $\Sigma(\tau) = (1-\tau)\Sigma'+\tau\Sigma$. Hence,
\begin{align*}
    |\Phi_{\Sigma'}(\xi) - \Phi_\Sigma(\xi)| &\leq \tfrac{1}{2}\,\psi(\xi^\top\Sigma(\tau)\xi)\,\|\Sigma'-\Sigma\|_{\mathrm{op}}\|\xi\|^2\\
    &\leq \tfrac{1}{2}\max\big\{|\Phi_\Sigma(\xi)|,|\Phi_{\Sigma'}(\xi)|\big\}\|\Sigma'-\Sigma\|_{\mathrm{op}}\|\xi\|^2,
\end{align*}
so Assumption~\ref{ass2} holds with $p=2$ and $\Xi(t)=t$.

\paragraph*{Verification of Assumption~\ref{ass1}: univariate case.}

Consider two univariate mixture densities
$$p_G(x) = \int f(x\mid\theta,\sigma^2)\,P(d\theta) \qquad\text{and}\qquad p_{G'}(x) = \int f(x\mid\theta,\sigma'^2)\,P'(d\theta),$$
where $f(\cdot\mid\theta,\sigma^2)$ denotes the univariate Gaussian kernel with mean $\theta$ and variance $\sigma^2$. Without loss of generality, suppose $\sigma^2\geq\sigma'^2$ and set $\gamma^2:=\sigma^2-\sigma'^2$. By the semigroup property of Gaussian convolutions, $f(\cdot\mid0,\sigma^2) = f(\cdot\mid0,\gamma^2)*f(\cdot\mid0,\sigma'^2)$, so that
$$p_G = P*f(\cdot\mid0,\sigma^2) = \big(P*f(\cdot\mid0,\gamma^2)\big)*f(\cdot\mid0,\sigma'^2) =: \tilde{P}*f(\cdot\mid0,\sigma'^2),$$
where $\tilde{P}$ is the distribution of $X+\gamma Z$ with $X\sim P$ and $Z\sim f(\cdot\mid0,1)$, independent of each other.

For a positive integer $k$ to be specified later, consider the Hermite polynomial
$$H_{2k}(x;\sigma'^2) := \mathbb{E}_{Z\sim f(\cdot \mid 0,1)}\big[(x+i\sigma' Z)^{2k}\big].$$
By Lemma~\ref{lemma:gaussian_property}, $\int_{\mathbb{R}} H_{2k}(x;\sigma'^2)\,f(x\mid\theta,\sigma'^2)\,dx = \theta^{2k}$, so that
\begin{align*}
    \Delta M_{2k} &:= \int\theta^{2k}\,\tilde{P}(d\theta) - \int\theta^{2k}\,P'(d\theta) = \int_{\mathbb{R}} H_{2k}(x;\sigma'^2)\big(p_G(x)-p_{G'}(x)\big)\,dx.
\end{align*}
We now lower-bound $\Delta M_{2k}$. Since $\tilde{P}$ is the law of $X+\gamma Z$,
\begin{align*}
    \int\theta^{2k}\,\tilde{P}(d\theta) = \sum_{j=0}^k\binom{2k}{2j}\mathbb{E}_P\big[X^{2k-2j}\big]\gamma^{2j}\mathbb{E}\big[Z^{2j}\big] \geq \gamma^{2k}(2k-1)!!.
\end{align*}
Using $(2k)!\geq(2k/e)^{2k}$ and $k!\leq k^k$, we get $(2k-1)!!=\frac{(2k)!}{2^k k!}\geq(2k/e^2)^k$. Since $P'$ is supported on $[-R,R]$, we have $\int\theta^{2k}\,P'(d\theta)\leq R^{2k}$, and therefore
\begin{equation}
    \Delta M_{2k} \geq \left(\frac{2k}{e^2}\right)^k\gamma^{2k} - R^{2k}. \label{eq:lower_DeltaM}
\end{equation}
To  obtain an upper-bound on $\Delta M_{2k}$, set $V:=\|p_G-p_{G'}\|_1$ and $\Lambda:=R+L$, for a $L>0$ to be specified. From the definition of $H_{2k}$,
$$|H_{2k}(x;\sigma'^2)|\leq 2^{k-1}\big(x^{2k}+\sigma'^{2k}(2k-1)!!\big)\leq 2^{k-1}\big(x^{2k}+\lambda_{\max}^k 2^k k^k\big).$$
For any choice of $|x|>\Lambda$,  $\theta\in[-R,R]$ and $\sigma\in[\lambda_{\min},\lambda_{\max}]$, the Gaussian kernel satisfies $f(x\mid\theta,\sigma^2)\leq C_0\exp(-(|x|-R)^2/2\lambda_{\max}^2)$ with $C_0=(2\pi\lambda_{\min}^2)^{-1/2}$, so that $|p_G(x)|,|p_{G'}(x)|\leq C_0\exp(-(|x|-R)^2/2\lambda_{\max}^2)$. Splitting the integral at $|x|=\Lambda$ and applying Lemma~\ref{lemma:dung_tail_gaussian_estimation} to bound the tail, we obtain
\begin{align}
    \Delta M_{2k} \leq C_2^k\big(V+e^{-L^2/2\lambda_{\max}^2}\big)\big(R^{2k}+L^{2k}+k^k\big), \label{eq:upper_DeltaM}
\end{align}
where $C_2>0$ depends only on $\lambda_{\min}$, $\lambda_{\max}$, and $R$. Combining Equations~\eqref{eq:lower_DeltaM} and \eqref{eq:upper_DeltaM} gives
\begin{equation}
    k^k\gamma^{2k} \leq R_1^k + C_3^k\big(V+e^{-L^2/2\lambda_{\max}^2}\big)\big(L^{2k}+k^k+R^{2k}\big), \label{eq:to_tune}
\end{equation}
where $R_1:=(Re)^2/2$ and $C_3:=C_2 e^2/2$.

We now tune $L$ and $k$. Set $L:=\lambda_{\max}(2\log(1/V))^{1/2}$, so that $e^{-L^2/2\lambda_{\max}^2}=V$ and
\begin{equation}
    k^k\gamma^{2k}\leq R_1^k + 2C_3^k V\big(R^{2k}+(2\lambda_{\max}^2)^k(\log(1/V))^k+k^k\big). \label{eq:tuned_k}
\end{equation}
Set $k:=\lfloor\log(1/V)/\log\log(1/V)\rfloor$. If $V<e^{-e}$, one has $k\leq\log(1/V)$, so that $V(\log(1/V))^k\leq 1$ and $V k^k\leq 1$. Substituting into \eqref{eq:tuned_k} and absorbing all terms into a geometric bound yields
$$\gamma^2\leq C_4\frac{\log\log(1/V)}{\log(1/V)},$$
where $C_4:=R_1+4C_3R^2+2C_3+4C_3\lambda_{\max}^2$. If $V\leq e^{-e^{8C_4}}$ (so that $\log\log(1/V)\geq 8C_4$) then, using $e^t\geq t^2/2$, we have that $C_4\log\log(1/V)\leq(\log(1/V))^{1/2}$, and thus  $\gamma^2\leq(\log(1/V))^{-1/2}$. Hence and
\begin{equation}\label{eq:small_V}
    V \geq \exp\!\left(-2C_4\frac{\log(1/\gamma^2)}{\gamma^2}\right),
\end{equation}
whenever $V\leq C_5:=\min\!\left\{e^{-e},e^{-e^{8C_4}}\right\}. $

On the other hand, if $V\geq C_5$, then, noting that the function $t\mapsto t\log t$ attains its minimum $-e^{-1}$ at $t=e^{-1}$, we have that
\begin{equation}
    V\geq C_5 \geq C_6\exp\!\left(-2C_4\frac{\log(1/\gamma^2)}{\gamma^2}\right), \label{eq:large_V}
\end{equation}
where $C_6:=C_5 e^{2C_4/e}$.
Combining Equations~\eqref{eq:small_V} and \eqref{eq:large_V} yields, with $M:=2C_4$ and $C':=\min\{1,C_6\}$,
\begin{equation}
    \|p_G-p_{G'}\|_1\geq C'\exp\!\left(-M\frac{\log(1/\gamma^2)}{\gamma^2}\right). \label{eq:univariate_conclusion}
\end{equation}
Thus, Assumption~\ref{ass1} is verified in the univariate case with $\Psi(t)=\exp(-M\log(1/t)/t)$.

\paragraph*{Verification of Assumption~\ref{ass1}: multivariate case.}

Let $u\in\mathbb{R}^d$ with $\|u\|=1$ be the principal unit eigenvector of $\Sigma-\Sigma'$, so that $u^\top(\Sigma-\Sigma')u = \|\Sigma-\Sigma'\|_{\mathrm{op}}$. Define the projected mixture densities
$$p_{G,u}(y):=\int f\big(y\mid u^\top\theta,\,u^\top\Sigma u\big)\,P(d\theta), \qquad p_{G',u}(y):=\int f\big(y\mid u^\top\theta,\,u^\top\Sigma' u\big)\,P'(d\theta),$$
which represent the marginal densities of $u^\top X$ under $p_G$ and $p_{G'}$ respectively. By Lemma~\ref{lemma:dung_total_variance_distance_estimation},
$$\|p_G-p_{G'}\|_1\geq\|p_{G,u}-p_{G',u}\|_1.$$
Setting $\sigma^2:=u^\top\Sigma u$, $\sigma'^2:=u^\top\Sigma'u$, and $\gamma^2:=|\sigma^2-\sigma'^2|=\|\Sigma-\Sigma'\|_{\mathrm{op}}$, the univariate bound \eqref{eq:univariate_conclusion} gives
$$\|p_{G,u}-p_{G',u}\|_1\geq C'\exp\!\left(-M\frac{\log(1/\gamma^2)}{\gamma^2}\right) = C'\exp\!\left(-M\frac{\log(1/\|\Sigma-\Sigma'\|_{\mathrm{op}})}{\|\Sigma-\Sigma'\|_{\mathrm{op}}}\right),$$
which establishes Assumption~\ref{ass1} with $\Psi(t)=\exp(-M\log(1/t)/t)$.

\paragraph*{Conclusion.}
With Assumptions 1 and 2 verified for $\Psi(t)=\exp(-M\log(1/t)/t)$, $\Xi(t)=t$, $p=2$, and super-smoothness of order $\alpha=2$ being obvious given the form of the Gaussian characteristic function, an application of Theorem~\ref{theorem:supersmooth_rate} completes the proof.

\subsection*{Proof of Theorem~\ref{thm:inequality_gaussian_isotropic}}
\label{proof_thm_inequality_gaussian_isotropic}

Before delving into the details of the proof, we remark that the key challenge in generalizing the argument below to the case of general $\Sigma$ and $\Sigma'$ matrices is due to the fact that, if we were to move from $\Sigma$ to $\Sigma'$ using the semi-group property, we would need to guarantee that $\Sigma - \Sigma'$ is a positive definite matrix, which is not always true unless we have further knowledge on the structure of $\Sigma$. For instance, as we do below, $\Sigma = \sigma^2 I_{d}$ works to this end, as positive-definiteness of $\Sigma - \Sigma'$ can be guaranteed by assuming, without loss of generality, that $\sigma > \sigma'$.

Therefore, without loss of generality, assume that $\sigma > \sigma'$. Define the variance gap $\Delta := \sigma^2 - \sigma'^2 > 0$ and $\widetilde{P} = P * f(\cdot\mid0, \Delta \cdot I_d)$, so that $\widetilde{P}$ is a continuous probability measure with a smooth density given by
$$\widetilde{P}(dy) = \left( \int_{\mathbb{R}^d} \frac{1}{(2\pi\Delta)^{d/2}} \exp\left( - \frac{\|y - \theta\|^2}{2\Delta} \right) P(d\theta) \right) dy.$$
Due to the Gaussian semigroup property, we have $p_G(x) \equiv \int f(x\mid\theta, \sigma'^2 I_d) \widetilde{P}(d\theta)$. To estimate $W_1(P,\tilde{P})$, we consider the coupling $(X,Y)$ such that $X \sim P$, $Y = X+Z$ where $Z \sim \mathcal{N}(0,\Delta \cdot I_d)$ and independent from $X$, which means that
$$ W_1(P, \widetilde{P})\leq \mathbb{E}[\|Y-X\|_2] \le \mathbb{E}[\|Z\|_2] \le \sqrt{\mathbb{E}[\|Z\|_2^2]} = \sqrt{\mathrm{Tr}(\Delta I_d)} = \sqrt{d\Delta}.$$
Therefore, an application of the triangle inequality leads to
\begin{align}
W_{1}(P, P') \leq W_{1}(P, \widetilde{P}) + W_{1}(\widetilde{P}, P') \leq \sqrt{d\Delta} + W_{1}(\widetilde{P}, P'). \label{eq:key_equality_Gaussian_0}
\end{align}

Using Lemma \ref{lemma:dung_ascoli_for_non_compact_measure}, there exists a function $h^*_{\mathcal{X}}\in \mathrm{Lip}_1(\mathbb{R}^{d})$, with $h^*_{\mathcal{X}}(0) = 0$, which is a solution of Kantorovich-Rubinstein dual problem: 
$$ W_1(\widetilde{P}, P') =  \int_{\mathbb{R}^{d}} h^*_{\mathcal{X}}(\theta) (\widetilde{P}(d\theta) - P'(d\theta)).$$

Lemma~\ref{lemma:cutoff_function} guarantees the existence of a cutoff function $\eta: \mathbb{R}^d \to \mathbb{R}$ in $C_c^\infty(\mathbb{R}^d)$ which takes values in $[0,1]$, evaluates identically to $1$ on the closed ball $\bar{B}(0, R)$ (where $P'$ is supported), and vanishes outside the open ball $B(0, 2R)$. Hence, given $\eta$, we denote
$$\widetilde{h}^*_{\mathcal{X}}(x) := h^*_{\mathcal{X}}(x) \eta(x)$$
for all $x \in \mathbb{R}^{d}$. As we show in Lemma~\ref{lemma:cutoff_function_properties}, $\widetilde{h}^*_{\mathcal{X}} \in L_{1}(\mathbb{R}^{d})$ and it is globally $(1 + 2RM_{\eta})$-Lipschitz continuous on $\mathbb{R}^{d}$, where $M_{\eta} = \|\nabla \eta\|_\infty$. By the properties of $\eta$ and noting that the support of $P$ and $P'$ are in $B(0,R)$, we also clearly have that
\begin{align*}
&\int_{\mathbb{R}^d} h^*_{\mathcal{X}}(\theta) (\widetilde{P}(d\theta) - P'(d\theta))\\
& = \int_{\mathbb{R}^{d}} (h^*_{\mathcal{X}}(\theta) - \widetilde{h}^*_{\mathcal{X}}(\theta)) (\widetilde{P}(d\theta) - P'(d\theta)) + \int_{\mathbb{R}^{d}} \widetilde{h}^*_{\mathcal{X}}(\theta) (\widetilde{P}(d\theta) - P'(d\theta)) \nonumber \\
& = \int_{\mathbb{R}^{d} \backslash B(0,R)} (h^*_{\mathcal{X}}(\theta) - \widetilde{h}^*_{\mathcal{X}}(\theta)) (\widetilde{P}(d\theta) - P'(d\theta)) +  \int_{\mathbb{R}^{d}} \widetilde{h}^*_{\mathcal{X}}(\theta) (\widetilde{P}(d\theta) - P'(d\theta)) \nonumber \\
& = \int_{\mathbb{R}^{d} \backslash B(0,R)} (h^*_{\mathcal{X}}(\theta) - \widetilde{h}^*_{\mathcal{X}}(\theta)) \widetilde{P}(d\theta) +  \int_{\mathbb{R}^{d}} \widetilde{h}^*_{\mathcal{X}}(\theta) (\widetilde{P}(d\theta) - P'(d\theta)) \nonumber \\
& = \int_{\mathbb{R}^{d} \backslash B(0,R)} (h^*_{\mathcal{X}}(\theta) - \widetilde{h}^*_{\mathcal{X}}(\theta)) (\widetilde{P}(d\theta) - P(d\theta)) +  \int_{\mathbb{R}^{d}} \widetilde{h}^*_{\mathcal{X}}(\theta) (\widetilde{P}(d\theta) - P'(d\theta)).
\end{align*}
It is clear that $h^*_{\mathcal{X}}(\theta) - \widetilde{h}^*_{\mathcal{X}}(\theta) = h^*_{\mathcal{X}}(\theta)(1 -  \eta(\theta))$. Therefore, $\nabla (h^*_{\mathcal{X}}(\theta) - \widetilde{h}^*_{\mathcal{X}}(\theta)) = (1 - \eta(\theta))\nabla h^*_{\mathcal{X}}(\theta) - h^*_{\mathcal{X}}(\theta) \nabla \eta(\theta)$, so that
$$\|\nabla (h^*_{\mathcal{X}} - \widetilde{h}^*_{\mathcal{X}})\|_{\infty} \leq \|(1 - \eta)\nabla h^*_{\mathcal{X}}\|_{\infty} + \|h^*_{\mathcal{X}} \nabla \eta\|_{\infty}.$$
Since $h^*_{\mathcal{X}} \in \mathrm{Lip}_1(\mathbb{R}^{d})$, we have $|h^*_{\mathcal{X}}(\theta)| \leq \|\theta\|_{2}$. Furthermore, as $\eta$ is stricly supported on $B(0,2R)$, we have
$$\|h^*_{\mathcal{X}} \nabla \eta\|_{\infty} \leq 2R M_{\eta}.$$
Furthermore, as $\|\nabla h^*_{\mathcal{X}}\|_{\infty} \leq 1$ and $\eta \in [0,1]$, we have 
$$\|(1 - \eta)\nabla h^*_{\mathcal{X}}\|_{\infty} \leq 1.$$
Putting these bounds together, we arrive at
$$\|\nabla (h^*_{\mathcal{X}} - \widetilde{h}^*_{\mathcal{X}})\|_{\infty} \leq 1 + 2RM_{\eta}.$$
Therefore,
\begin{align}
\nonumber
\int_{\mathbb{R}^{d} \backslash B(0,R)} ({h}^*_{\mathcal{X}}(\theta) - \widetilde{h}^*_{\mathcal{X}}(\theta)) (\widetilde{P}(d\theta) - P(d\theta)) &= \int_{\mathbb{R}^{d}} ({h}^*_{\mathcal{X}}(\theta) - \widetilde{h}^*_{\mathcal{X}}(\theta)) (\widetilde{P}(d\theta) - P(d\theta)) \\
&\leq (1 + 2RM_{\eta}) W_{1}(\widetilde{P}, P), \label{eq:key_equality_Gaussian_2}
\end{align}
where the final inequality is due to the Kantorovich-Rubinstein duality theorem while the equality is due to ${h}^*_{\mathcal{X}}(\theta) = \widetilde{h}^*_{\mathcal{X}}(\theta)$ when $\theta \in B(0,R)$. 

Moving to $\int_{\mathbb{R}^{d}} \widetilde{h}^{*}(\theta) (\widetilde{P}(d\theta) - P'(d\theta))$, we can bound this term using the spectral bandlimiting via the Jackson kernel as in the proof of Theorem~\ref{theorem:supersmooth_rate}. In particular, we consider the multivariate Jackson kernel \citep{devore1993constructive,butzer1971fourier}.
$$K(x) := \bigg( \frac{3}{8\pi\sqrt{d}} \bigg)^d \prod_{i=1}^d \mathrm{sinc}^4\bigg(\frac{x_i}{4\sqrt{d}}\bigg), \quad x \in \mathbb{R}^{d},$$
where $\mathrm{sinc}(u) := \sin u/u$, which by construction satisfies $\int_{\mathbb{R}^d} K(x) dx = 1$. For any $\Lambda > 0$, define the function $K_{\Lambda}(x) := \Lambda^{d} K(\Lambda x)$. By Lemma~\ref{lemma:kernel_properties}, $\mathrm{supp}(\mathcal{F}[K]) \subseteq \bar{B}(0, 1)$, which implies that $\mathrm{supp}(\mathcal{F}[K_{\Lambda}]) \subseteq \bar{B}(0, \Lambda)$. Consider the function
$$\widetilde{h}_{\mathcal{X},\Lambda}^{*}(x) := (\widetilde{h}^{*}_{\mathcal{X}} * K_\Lambda)(x) = \int_{\mathbb{R}^d} \widetilde{h}^{*}_{\mathcal{X}}(x - y) K_\Lambda(y) dy.$$
Direct computation yields that
\begin{align*}
    \widetilde{h}_{\mathcal{X},\Lambda}^{*}(x) - \widetilde{h}_{\mathcal{X}}^{*}(x) & = \int_{\mathbb{R}^d} \widetilde{h}^{*}_{\mathcal{X}}(x - y) K_\Lambda(y) dy - \int_{\mathbb{R}^d} \widetilde{h}_{\mathcal{X}}^{*}(x) K_\Lambda(y) dy \\
    & = \int_{\mathbb{R}^d} [\widetilde{h}^{*}_{{\mathcal{X}}}(x - y) - \widetilde{h}^{*}_{\mathcal{X}}(x)] K_\Lambda(y) dy.
\end{align*}
Since $\widetilde{h}^{*}_{\mathcal{X}}$ is a globally $(1+2RM_{\eta})$-Lipschitz continuous function, we further obtain for any $x \in \mathbb{R}^{d}$ that
\begin{align*}
    |\widetilde{h}_{\mathcal{X},\Lambda}^{*}(x) - \widetilde{h}^{*}_{\mathcal{X}}(x)| &\leq \int_{\mathbb{R}^d} |\widetilde{h}^{*}_{\mathcal{X}}(x - y) - \widetilde{h}^{*}_{\mathcal{X}}(x)| \Lambda^{d} K(\Lambda y) dy\\
    &\leq (1+2RM_{\eta}) \int_{\mathbb{R}^d} \|y\| \Lambda^{d} K(\Lambda y) dy \\
    & = \frac{(1 + 2RM_{\eta})M_{K}}{\Lambda},
\end{align*}
where the last inequality follows from Lemma~\ref{lemma:kernel_properties}. As a consequence, we arrive at
\begin{align}
\|\widetilde{h}_{\mathcal{X},\Lambda}^{*}- \widetilde{h}^{*}_{\mathcal{X}}\|_{\infty} \leq \frac{(1 + 2RM_{\eta})M_{K}}{\Lambda}. \label{eq:key_equality_Gaussian_3}
\end{align}
Now, we have
\begin{align*}
\int_{\mathbb{R}^{d}} \widetilde{h}^{*}_{\mathcal{X}}(\theta) &(\widetilde{P}(d\theta) - P'(d\theta)) \\
&= \int_{\mathbb{R}^d} (\widetilde{h}^{*}_{\mathcal{X}}(\theta) - \widetilde{h}_{\mathcal{X},\Lambda}^{*}(\theta)) (\widetilde{P}(d\theta) - P'(d\theta)) + \int_{\mathbb{R}^d}  \widetilde{h}_{\mathcal{X},\Lambda}^{*}(\theta) d(\widetilde{P}(d\theta) - P'(d\theta)).
\end{align*}
A direct application of Hölder's inequality leads to 
$$\int_{\mathbb{R}^d} (\widetilde{h}^{*}_{\mathcal{X}}(\theta) - \widetilde{h}_{\mathcal{X},\Lambda}^{*}(\theta)) (\widetilde{P}(d\theta) - P'(d\theta)) \leq \|\widetilde{h}_{\mathcal{X},\Lambda}^{*}- \widetilde{h}^{*}_{\mathcal{X}}\|_{\infty} \|\widetilde{P} - P'\|_{1} \leq \frac{2(1 + 2RM_{\eta})M_{K}}{\Lambda},$$
where the final inequality is due to $\|\widetilde{P} - P'\|_{1} \leq 2$ and the bound of $\|\widetilde{h}_{\Lambda}^{*}- \widetilde{h}^{*}\|_{\infty}$ in Equation~\eqref{eq:key_equality_Gaussian_3}. Hence, we obtain the following bound
\begin{align*}
\int_{\mathbb{R}^d} (\widetilde{h}^{*}_{\mathcal{X}}(\theta) - \widetilde{h}_{\mathcal{X},\Lambda}^{*}(\theta)) &(\widetilde{P}(d\theta) - P'(d\theta)) \\
&\leq \frac{2(1 + 2RM_{\eta})M_{K}}{\Lambda} + \int_{\mathbb{R}^d}  \widetilde{h}_{\mathcal{X},\Lambda}^{*}(\theta) d(\widetilde{P}(d\theta) - P'(d\theta)).
\end{align*}

To apply the Fourier inversion theorem to $\int_{\mathbb{R}^d}  \widetilde{h}_{\mathcal{X}, \Lambda}^{*}(\theta) (P(d\theta) - P'(d\theta))$, we first need to demonstrate that $\widetilde{h}_{\mathcal{X}, \Lambda}^{*} \in L_{1}(\mathbb{R}^{d})$ and $\mathcal{F}[\widetilde{h}_{\mathcal{X}, \Lambda}^{*}] \in L_{1}(\mathbb{R}^{d})$. Because $\widetilde{h}_{\mathcal{X}, \Lambda}^{*}(x) = (\widetilde{h}_{\mathcal{X}}^{*} * K_\Lambda)(x)$, an application of Young's inequality leads to
$\|\widetilde{h}_{\mathcal{X}, \Lambda}^{*}\|_{1} \leq \|\widetilde{h}_{\mathcal{X}}^{*}\|_{1}\|K_\Lambda\|_{1} = \|\widetilde{h}_{\mathcal{X}}^{*}\|_{1} < \infty$ (thanks to Lemma~\ref{lemma:cutoff_function_properties}). As for $\mathcal{F}[\widetilde{h}_{\mathcal{X}, \Lambda}^{*}]$, by the properties of convolutions we obtain that
$\mathcal{F}[\widetilde{h}_{\mathcal{X}, \Lambda}^{*}](\xi) = \mathcal{F}[\widetilde{h}_{\mathcal{X}}^{*}](\xi) \mathcal{F}[K_\Lambda](\xi)$. As $\mathrm{supp}(\mathcal{F}[K_{\Lambda}]) \subseteq \bar{B}(0, \Lambda)$, we have $\mathrm{supp}(\mathcal{F}[\widetilde{h}_{\mathcal{X}, \Lambda}^{*}]) \subseteq \bar{B}(0, \Lambda)$. Furthermore, $\|\mathcal{F}[K_\Lambda]\|_{\infty} \leq 1$ and $\|\mathcal{F}[\widetilde{h}_{\mathcal{X}}^{*}]\|_{\infty} \leq \frac{\pi^{d/2}}{\Gamma(d/2 + 1)} (2R)^{d + 1}$ (see Lemma~\ref{lemma:cutoff_function_properties}). Therefore, $\|\mathcal{F}[\widetilde{h}_{\mathcal{X}, \Lambda}^{*}]\|_{\infty} \leq \|\mathcal{F}[\widetilde{h}_{\mathcal{X}}^{*}]\|_{\infty} \leq \frac{\pi^{d/2}}{\Gamma(d/2 + 1)} (2R)^{d + 1}$. Putting these results together, we obtain that $\mathcal{F}[\widetilde{h}_{\mathcal{X}, \Lambda}^{*}] \in L_{1}(\mathbb{R}^{d})$. 
Now, from the Fourier inversion theorem, as $\widetilde{h}_{\mathcal{X},\Lambda}^{*} \in L_{1}(\mathbb{R}^{d})$ and $\mathcal{F}[\widetilde{h}_{ \mathcal{X},\Lambda}^{*}] \in L_{1}(\mathbb{R}^{d})$, we can write
$$\widetilde{h}_{\mathcal{X},\Lambda}^{*}(x) = \frac{1}{(2\pi)^d} \int_{\mathbb{R}^d} \mathcal{F}[\widetilde{h}_{\mathcal{X},\Lambda}^{*}](\xi) e^{i \xi^\top x} d\xi,$$
leading to 
$$\int_{\mathbb{R}^d}  \widetilde{h}_{\mathcal{X},\Lambda}^{*}(\theta) (\widetilde{P}(d\theta) - P'(d\theta)) = \frac{1}{(2\pi)^d} \int_{\mathbb{R}^d}  \left(\int_{\mathbb{R}^d} \mathcal{F}[\widetilde{h}_{\mathcal{X},\Lambda}^{*}](\xi) e^{i \xi^\top \theta} d\xi\right) (\widetilde{P}(d\theta) - P'(d\theta)).$$
By Fubini's theorem, we have that
\begin{align*}
    \int_{\mathbb{R}^d}  \widetilde{h}_{\mathcal{X},\Lambda}^{*}(\theta) (\widetilde{P}(d\theta) - P'(d\theta)) & = \frac{1}{(2\pi)^d} \int_{\mathbb{R}^d}  \left(\int_{\mathbb{R}^d} e^{i \xi^\top \theta} (\widetilde{P}(d\theta) - P'(d\theta)) \right) \mathcal{F}[\widetilde{h}_{\mathcal{X},\Lambda}^{*}](\xi) d\xi \\
    & = \frac{1}{(2\pi)^d} \int_{\mathbb{R}^d} \mathcal{F}[\widetilde{h}_{\mathcal{X},\Lambda}^{*}](\xi) \Big( \Phi_{\widetilde{P}}(\xi) - \Phi_{P'}(\xi) \Big) d\xi \\
    & = \frac{1}{(2\pi)^d} \int_{\bar{B}(0, \Lambda)} \mathcal{F}[\widetilde{h}_{\mathcal{X},\Lambda}^{*}](\xi) \Big( \Phi_{\widetilde{P}}(\xi) - \Phi_{P'}(\xi) \Big) d\xi,
\end{align*}
where the final identity is due to the fact that $\mathrm{supp}(\mathcal{F}[\widetilde{h}_{\mathcal{X},\Lambda}^{*}]) \subseteq \bar{B}(0, \Lambda)$. 

Putting all of these results together, we arrive at
\begin{align}
    \int_{\mathbb{R}^d} (\widetilde{h}^{*}_{\mathcal{X}}(\theta)& - \widetilde{h}_{\mathcal{X},\Lambda}^{*}(\theta)) (\widetilde{P}(d\theta) - P'(d\theta)) \nonumber \\
    &\leq \frac{2(1 + 2RM_{\eta})M_{K}}{\Lambda} + \frac{1}{(2\pi)^d} \int_{\bar{B}(0, \Lambda)} \mathcal{F}[\widetilde{h}_{\mathcal{X},\Lambda}^{*}](\xi) \Big( \Phi_{\widetilde{P}}(\xi) - \Phi_{P'}(\xi) \Big) d\xi. \label{eq:key_equality_Gaussian_4}
\end{align}
Next, since $p_G = P * f(\cdot\mid0,\sigma^2 I_{d})$, it is clear that $p_G = \widetilde{P} * f(\cdot\mid0,\sigma'^2 I_{d})$.
Therefore, for all $\xi \in \mathbb{R}^{d}$ we obtain that
\begin{align*}
    \Phi_{\widetilde{P}}(\xi) - \Phi_{P'}(\xi) &= \frac{\Phi_{p_{G}}(\xi)}{\Phi_{\sigma'^2 I_{d}}(\xi)} - \frac{\Phi_{p_{G'}}(\xi)}{\Phi_{\sigma'^2 I_{d}}(\xi)} \\
    &= \frac{\Phi_{p_{G}}(\xi) - \Phi_{p_{G'}}(\xi)}{\Phi_{\sigma'^2 I_{d}}(\xi)}.
\end{align*}
Since $\|\mathcal{F}[\widetilde{h}_{\mathcal{X},\Lambda}^*]\|_{\infty} \leq \frac{\pi^{d/2}}{\Gamma(d/2 + 1)} (2R)^{d + 1}=:C(R,d)$, we obtain that
\begin{align}
    &\left|\int_{\bar{B}(0, \Lambda)} \mathcal{F}[\widetilde{h}_{\mathcal{X},\Lambda}^{*}](\xi) \Big( \Phi_P(\xi) - \Phi_{P'}(\xi) \Big) d\xi\right|  \leq \frac{C(R,d)}{(2\pi)^{d}} \int_{\bar{B}(0, \Lambda)} \frac{|\Phi_{p_{G}}(\xi) - \Phi_{p_{G'}}(\xi)|}{|\Phi_{\sigma'^2 I_{d}}(\xi)|} d\xi. \label{eq:key_equality_Gaussian_5}
\end{align}
By construction, $|\Phi_{p_{G}}(\xi) - \Phi_{p_{G'}}(\xi)| \leq \|p_{G} - p_{G'}\|_{1}$ for any $\xi \in \mathbb{R}^{d}$. Furthermore,  $\Phi_{\sigma'^2 I_{d}}(\xi) = \exp \left(- \frac{\sigma'^2}{2} \|\xi\|^2\right)$. Therefore, we obtain that
\begin{align}
\frac{C(R,d)}{(2\pi)^{d}} \int_{\bar{B}(0, \Lambda)} \frac{|\Phi_{p_{G}}(\xi) - \Phi_{p_{G'}}(\xi)|}{|\Phi_{\sigma'^2 I_{d}}(\xi)|} d\xi  & \leq  \frac{C(R,d)}{(2\pi)^{d}} \|p_{G} - p_{G'}\|_{1} \int_{\bar{B}(0, \Lambda)} \exp \left( \frac{\sigma'^2}{2} \|\xi\|^2\right) d\xi \nonumber \\
& \leq C_{1} \|p_{G} - p_{G'}\|_{1} \Lambda^{d} \exp \left(\frac{\sigma'^2}{2} \Lambda^2\right), \label{eq:key_equality_Gaussian_6}
\end{align}
where $C_{1} =\frac{C(R,d)}{(4\pi)^{d/2}   \Gamma(d/2+1)}$ is a constant depending on $R$ and $d$. 

Collecting the results from equations~\eqref{eq:key_equality_Gaussian_0},~\eqref{eq:key_equality_Gaussian_2},~\eqref{eq:key_equality_Gaussian_4},~\eqref{eq:key_equality_Gaussian_5}, and~\eqref{eq:key_equality_Gaussian_6}, we arrive at
\begin{align}
\nonumber
    W_{1}(P,P') &\leq (2 + 2RM_{\eta}) \sqrt{d} \sqrt{\Delta} + \frac{2(1 + 2RM_{\eta}) M_{K}}{\Lambda}\\
 &\hspace{1cm}+ C_{1} \|p_{G} - p_{G'}\|_{1} \Lambda^{d} \exp \left( \frac{\sigma'^2}{2} \Lambda^2\right).  \label{eq:key_equality_Gaussian_7}
\end{align}
Let $V := \|p_{G} - p_{G'}\|_{1}$, $\bar{C}_2 = (2 + 2RM_{\eta}) \sqrt{d}$, $C_3 = 2(1 + 2RM_{\eta}) M_{K}$. Here we consider two situations. 

\paragraph*{Case 1: small $V$.} Recall the following inequality for $\Delta = \sigma^2 - \sigma'^2$ from Theorem~\ref{thm:inequality_gaussian}:
$$\|p_G  - p_{G'}\|_1 \;\geq\; \exp\left(-M \frac{\log(1/\Delta)}{\Delta} \right),$$
which implies 
\begin{equation*}
    \log(1/V) \leq M\frac{\log(1/\Delta)}{\Delta}. 
\end{equation*}
For $V\leq \exp(-e^{1/M})$, using Lemma \ref{lemma:dung_a_small_estimation_from_xlogx}, we have 
\begin{equation*}
    \frac{1}{\Delta} \geq \frac{1}{M}\frac{\log(1/V)}{\log\log(1/V)}. 
\end{equation*}
Replacing this inequality into 
the inequality~\eqref{eq:key_equality_Gaussian_7}, we obtain that
$$W_{1}(P,P') \leq C_{2} \sqrt{\frac{\log \log(1/V)}{\log(1/V)}} + \frac{C_{3}}{\Lambda} + C_{4} V \Lambda^{d} \exp \left(\frac{\sigma'^2}{2} \Lambda^2\right),$$
where $C_2 = \sqrt{M}\cdot \bar{C}_2$. Let $N = \log(1/V)$. When $V \leq e^{-(d+1)^2}$, then we have $N > (d+1)\log N$, we can choose 
$$\Lambda = \frac{\sqrt{2}}{\sigma'} \sqrt{N - \frac{d+1}{2} \log N}.$$
This choice of $\Lambda$ gives us 
\begin{equation*}
    \frac{1}{\sigma'}\sqrt{N}\leq \Lambda \leq \frac{\sqrt{2}}{\sigma'}\sqrt{N}.
\end{equation*}
Under this scenario, $$ \frac{C_3}{\Lambda} \leq \frac{C_3 \sigma'}{\sqrt{N}} = \frac{C_3 \sigma'}{\sqrt{\log(1/V)}}. $$
Furthermore, we have
$$ \exp \left( \frac{\sigma'^2}{2} \Lambda^2 \right) = \exp\left(N - \frac{d+1}{2} \log N\right) = \frac{e^N}{N^{(d+1)/2}}.$$
Therefore, 
$$C_{4} V \Lambda^{d} \exp \left(\frac{\sigma'^2}{2} \Lambda^2\right) = C_4 e^{-N} \Big( N^{d/2} \Big) \left(\frac{e^N}{N^{(d+1)/2}}\right) = C_4N^{-1/2} = \frac{C_4}{\sqrt{\log(1/V)}}.$$
As a result, we have
$$ W_1(P,P') \leq C_2 \sqrt{\frac{\log \log(1/V)}{\log(1/V)}} + \frac{C_6}{\sqrt{\log(1/V)}} \leq C_7 \sqrt{\frac{\log \log(1/V)}{\log(1/V)}},$$
where $C_6 = C_4 + C_3\sigma'$, $C_7 = C_2 + C_6$, or 
$$W_1^2(P,P') \leq C_7^2\frac{\log \log(1/V)}{\log(1/V)}.$$ 

If $V\leq e^{-e^{8C^2_7}}$ (so that $\log\log(1/V)\geq 8C^2_7$) then, using $e^t\geq t^2/2$, we have that $C^2_7\log\log(1/V)\leq(\log(1/V))^{1/2}$, and thus  $W^2_1(P,P') \leq(\log(1/V))^{-1/2}$, which implies $2\log(1/W^2_1(P,P'))\geq\log\log(1/V)$. Thus, for
$$V\leq C_5 := \min\{\exp(-e^{1/M}), e^{-(d+1)^2}, e^{-e^{8C_7^2}}\}$$
we achieve
\begin{equation}\label{eq:small_V_gaussian}
    V \geq \exp\left(-2C_7^2\frac{\log(1/W^2_1(P,P'))}{W^2_1(P,P')}\right) =\exp\!\left(-4C_7^2\frac{\log(1/W_1(P,P'))}{W^2_1(P,P')}\right) 
\end{equation}

\paragraph*{Case 2: $V \geq C_5$.} then noting that the function $t\mapsto t\log t$ attains its minimum $-e^{-1}$ at $t=e^{-1}$, we have that
\begin{equation}
    V\geq C_5 \geq C_8\exp\!\left(-4C_7^2\frac{\log(1/W_1(P,P'))}{W^2_1(P,P')}\right), \label{eq:large_V_gaussian}
\end{equation}
where $C_8:=C_5 e^{-2C_7^2/e}$.
Combining Equations~\eqref{eq:small_V_gaussian} and \eqref{eq:large_V_gaussian} yields, with $M:=4C_7^2$ and $C':=\min\{1,C_8\}$,
\begin{equation}
    \|p_G-p_{G'}\|_1\geq C'\exp\!\left(-M\frac{\log(1/W_1(P,P'))}{W^2_1(P,P')}\right). \label{eq:univariate_conclusion_gaussian}
\end{equation}
uniformly for all $P$ and $P'$. This concludes the proof.

\subsection*{Proof of Theorem~\ref{thm:inequality_cauchy}}
\label{proof_thm_inequality_cauchy}
We split the proof in four parts, as follows.
\paragraph*{Verification of Assumption~\ref{ass2}.}

The characteristic function of the Cauchy kernel is $\Phi_\Sigma(\xi) = \exp(-\sqrt{\xi^\top\Sigma\xi})$. Setting $\psi(t) := e^{-t}$, so that $\psi'(t) = -\psi(t)$, the mean value theorem gives a $\tau\in(0,1)$ such that
$$\Phi_{\Sigma'}(\xi) - \Phi_\Sigma(\xi) = \psi'\!\left((\xi^\top\Sigma(\tau)\xi)^{1/2}\right)\left((\xi^\top\Sigma'\xi)^{1/2} - (\xi^\top\Sigma\xi)^{1/2}\right),$$
where $\Sigma(\tau) := (1-\tau)\Sigma'+\tau\Sigma$. Using the inequality $|\sqrt{a}-\sqrt{b}|\leq\sqrt{|a-b|}$, we get
\begin{align*}
    |\Phi_{\Sigma'}(\xi) - \Phi_\Sigma(\xi)| &\leq \psi\!\left((\xi^\top\Sigma(\tau)\xi)^{1/2}\right)|\xi^\top(\Sigma'-\Sigma)\xi|^{1/2}\\
    &\leq \max\big\{|\Phi_\Sigma(\xi)|,|\Phi_{\Sigma'}(\xi)|\big\}\|\Sigma'-\Sigma\|_{\mathrm{op}}^{1/2}\|\xi\|,
\end{align*}
so Assumption~\ref{ass2} holds with $C=1$, $p=1$, and $\Xi(t)=\sqrt{t}$.

\paragraph*{Verification of Assumption~\ref{ass1}: univariate case.}

Consider two univariate mixture densities
$$p_G(x) = \int f(x\mid\theta,\sigma)\,P(d\theta) \qquad\text{and}\qquad p_{G'}(x) = \int f(x\mid\theta,\sigma')\,P'(d\theta),$$
where $f(\cdot\mid\theta,\sigma)$ is the univariate Cauchy kernel with location $\theta$ and scale $\sigma$. Without loss of generality, suppose $\sigma'>\sigma$. Setting $V:=\|p_G-p_{G'}\|_1$, since $\mathcal{F}[p_G](\xi) = \Phi_P(\xi)\exp(-\sigma|\xi|)$,
\begin{align*}
    V &\geq \mathrm{Re}\big(\mathcal{F}[p_G](\xi)\big) - \mathrm{Re}\big(\mathcal{F}[p_{G'}](\xi)\big)\\
    & = \exp(-\sigma|\xi|)\int\cos(\xi\theta)\,P(d\theta) - \exp(-\sigma'|\xi|)\int\cos(\xi\theta)\,P'(d\theta).
\end{align*}
Using $1-\xi^2\theta^2/2\leq\cos(\xi\theta)\leq 1$ and $|\theta|\leq R$ for $\theta\in\mathrm{supp}(P)\cup\mathrm{supp}(P')$,
\begin{equation}
    1 - \tfrac{1}{2}R^2\xi^2 \leq \int\cos(\xi\theta)\,P(d\theta) \leq 1. \label{eq:cauchy_cos_bound}
\end{equation}
Hence, for all $\xi\in\mathbb{R}$,
\begin{align}
    V &\geq \exp(-\sigma|\xi|)\left(1-\tfrac{1}{2}R^2\xi^2\right) - \exp(-\sigma'|\xi|) \nonumber\\
    &= \exp(-\sigma|\xi|)\big(1-\exp(-(\sigma'-\sigma)|\xi|)\big) - \tfrac{1}{2}R^2\xi^2\exp(-\sigma|\xi|). \label{eq:cauchy_V_lower}
\end{align}
Applying $1-e^{-x}\geq x - x^2/2$ for $x\geq 0$ and $e^{-\sigma|\xi|}\geq 1-\sigma|\xi|$ to the first term, and $e^{-\sigma|\xi|}\leq 1$ to the second, we obtain for $|\xi|\leq 2/(\sigma'-\sigma)$,
\begin{equation*}
    V \geq |\xi|(\sigma'-\sigma) - \xi^2\left(\sigma(\sigma'-\sigma) + \tfrac{1}{2}(\sigma'-\sigma)^2 + R^2\right).
\end{equation*}
The right-hand side is maximized at
$$|\xi| = \xi_0 := \frac{\sigma'-\sigma}{2\sigma(\sigma'-\sigma)+(\sigma'-\sigma)^2+2R^2},$$
which satisfies $\xi_0\leq 2/(\sigma'-\sigma)$. Substituting and simplifying,
$$V \geq \frac{(\sigma'-\sigma)^2}{4\sigma(\sigma'-\sigma)+2(\sigma'-\sigma)^2+4R^2} \geq \frac{(\sigma'-\sigma)^2}{6\lambda_{\max}^2+4R^2} =: C_0(\sigma'-\sigma)^2.$$
This confirms Assumption~\ref{ass1} in the univariate case with $\Psi(t) = t^2$.

\paragraph*{Verification of Assumption~\ref{ass1}: multivariate case.}

Let $u\in\mathbb{R}^d$ be the principal  eigenvector (thus, of unit norm) of $\Sigma-\Sigma'$, so that $u^\top(\Sigma-\Sigma')u = \|\Sigma-\Sigma'\|_{\mathrm{op}}$. Define the projected mixture densities
\begin{align*}
    p_{G,u}(y) &:= \int f\big(y\mid u^\top\theta,\,(u^\top\Sigma u)^{1/2}\big)\,P(d\theta),\\
    p_{G',u}(y) &:= \int f\big(y\mid u^\top\theta,\,(u^\top\Sigma' u)^{1/2}\big)\,P'(d\theta),
\end{align*}
which correspond to the marginal densities of $u^\top X$ under $p_G$ and $p_{G'}$ respectively. By Lemma~\ref{lemma:dung_total_variance_distance_estimation},
$$\|p_G-p_{G'}\|_1 \geq \|p_{G,u}-p_{G',u}\|_1.$$
Setting $\sigma := (u^\top\Sigma u)^{1/2}$ and $\sigma' := (u^\top\Sigma' u)^{1/2}$, the univariate bound gives $\|p_{G,u}-p_{G',u}\|_1 \geq C_0(\sigma'-\sigma)^2$. Then,
$$|\sigma'-\sigma| =  \frac{|\sigma'-\sigma|(\sigma'+\sigma)}{\sigma'+\sigma} = \frac{|u^\top(\Sigma'-\Sigma)u|}{\sigma'+\sigma} \geq \frac{\|\Sigma'-\Sigma\|_{\mathrm{op}}}{2\lambda_{\max}^{1/2}},$$
from which it follows that
$$\|p_G-p_{G'}\|_1 \geq \frac{C_0}{4\lambda_{\max}}\|\Sigma-\Sigma'\|_{\mathrm{op}}^2,$$
thus establishing Assumption~\ref{ass1} with $\Psi(t) = t^2$.

\paragraph*{Conclusion.}
With Assumptions 1 and 2 verified for $\Psi(t)=t^2$, $\Xi(t)=\sqrt{t}$ and $p=1$, and since the super-smoothness condition of order $\alpha=1$ is automatically verified given the form of the Cauchy characteristic function, an application of Theorem~\ref{theorem:supersmooth_rate} completes the proof.

\subsection*{Proof of Theorem~\ref{thm:inequality_laplace}}
\label{proof_thm_inequality_laplace}
We split the proof into five parts, as follows.

\paragraph*{PDE inversion condition.}

The characteristic function of the multivariate Laplace kernel is $\xi \in \mathbb{R}^d \mapsto \Phi_\Sigma(\xi) = (1+\frac{1}{2}\xi^\top\Sigma\xi)^{-1}$, so $\Phi_\Sigma(\xi)^{-1} = 1+\frac{1}{2}\xi^\top\Sigma\xi$. By the linearity of the Fourier transform and the identity $\mathcal{F}[\partial^\nu p_G](\xi) = (-i)^{|\nu|}\xi^\nu\mathcal{F}[p_G](\xi)$,
\begin{align*}
    \mathcal{F}[\mathcal{T}_\Sigma p_G](\xi)
    &= \mathcal{F}[p_G](\xi) - \frac{1}{2}\sum_{j=1}^d\sum_{k=1}^d\Sigma_{jk}(-\xi_j\xi_k)\mathcal{F}[p_G](\xi)
    = \left(1+\frac{1}{2}\xi^\top\Sigma\xi\right)\mathcal{F}[p_G](\xi)\\
    &= \Phi_\Sigma(\xi)^{-1}\mathcal{F}[p_G](\xi) = \mathcal{F}[P](\xi),
\end{align*}
where the last equality uses $\mathcal{F}[p_G](\xi) = \Phi_P(\xi)\Phi_\Sigma(\xi)$. By uniqueness of the Fourier transform, $\mathcal{T}_\Sigma p_G = P$ in the sense of distributions, and similarly $\mathcal{T}_{\Sigma'}p_{G'} = P'$. The Lipschitz continuity condition on the coefficients is immediate, since the only non-constant coefficients are $c_\nu(\Sigma) = -\Sigma_{jk}/2$ for $\nu\in\mathbb N^d$ with $|\nu|=2$, which depend linearly on $\Sigma$.

\paragraph*{Verification of Assumption~\ref{ass2}.}

Setting $\psi(t) := (1+t/2)^{-1}$, so that $|\psi'(t)| = (1/2)|\psi(t)|^2 \leq (1/2)|\psi(t)|$, the mean value theorem implies that 
\begin{align*}
    |\Phi_{\Sigma'}(\xi) - \Phi_\Sigma(\xi)|
    &= |\psi'(\xi^\top\Sigma(\tau)\xi)|\cdot|\xi^\top(\Sigma'-\Sigma)\xi|
    \leq \tfrac{1}{2}\,\psi(\xi^\top\Sigma(\tau)\xi)\,\|\Sigma'-\Sigma\|_{\mathrm{op}}\|\xi\|^2\\
    &\leq \tfrac{1}{2}\max\big\{|\Phi_\Sigma(\xi)|,|\Phi_{\Sigma'}(\xi)|\big\}\|\Sigma'-\Sigma\|_{\mathrm{op}}\|\xi\|^2,
\end{align*}
for some  $\tau\in(0,1)$, where we set $\Sigma(\tau):=(1-\tau)\Sigma'+\tau\Sigma$. Hence Assumption~\ref{ass2} holds with $p=2$ and $\Xi(t)=t$.

\paragraph*{Verification of Assumption~\ref{ass1}: univariate case.}

Consider two univariate mixtures
$$p_G(x) = \int f(x\mid\theta,\sigma)\,P(d\theta) \qquad\text{and}\qquad p_{G'}(x) = \int f(x\mid\theta,\sigma')\,P'(d\theta),$$
where $f(\cdot\mid\theta,\sigma)$ is the univariate Laplace kernel with location $\theta$ and scale $\sigma$. The PDE inversion condition gives $\mathcal{T}_\sigma p_G = P$ and $\mathcal{T}_{\sigma'} p_{G'} = P'$, where $\mathcal{T}_\sigma := \mathrm{Id} - \frac{\sigma}{2}\frac{d^2}{dx^2}$.

Let $h\in C_c^\infty(\mathbb{R})$ be a non-negative bump function supported on $(R+1,R+2)$ with $\|h\|_1 = 1$; the key property here is that $\mathrm{supp}(h)$ is disjoint from $[-R,R]\supseteq\mathrm{supp}(P)\cup\mathrm{supp}(P')$. Therefore $\langle P, h\rangle = \langle P', h\rangle = 0$, and the PDE inversion gives
$$\int_{\mathbb{R}} p_G(x)\,\big(\mathcal{T}_\sigma h(x)\big)\,dx = \langle \mathcal{T}_\sigma p_G, h\rangle = \langle P, h\rangle = 0.$$
On the other hand, decomposing $\mathcal{T}_\sigma = \mathcal{T}_{\sigma'} + (\mathcal{T}_\sigma - \mathcal{T}_{\sigma'})$ and using $\langle p_{G'}, \mathcal{T}_{\sigma'}h\rangle = \langle P', h\rangle = 0$,
$$\int_{\mathbb{R}} p_{G'}(x)\,\big(\mathcal{T}_\sigma h(x)\big)\,dx = \langle p_{G'},(\mathcal{T}_{\sigma}-\mathcal{T}_{\sigma'})h\rangle = -\tfrac{1}{2}(\sigma-\sigma')\int_{\mathbb{R}} p_{G'}(x)\,h''(x)\,dx,$$
since $\mathcal{T}_\sigma - \mathcal{T}_{\sigma'} = -\frac{\sigma-\sigma'}{2}\frac{d^2}{dx^2}$. Subtracting and applying Hölder's inequality,
\begin{equation}
    V\,\|\mathcal{T}_\sigma h\|_\infty \geq \tfrac{1}{2}|\sigma-\sigma'|\left|\int_{\mathbb{R}} h(x)\,p_{G'}''(x)\,dx\right|. \label{eq:laplace_holder}
\end{equation}
For $x\in(R+1,R+2)$ and $\theta'\in[-R,R]$, we have $x-\theta'\geq 1 > 0$, so the Laplace kernel simplifies to $f(x\mid\theta',\sigma') = (2\sigma')^{-1/2}\exp(-(x-\theta')/\sqrt{\sigma'/2})$, giving
$$f''(x\mid\theta',\sigma') = \frac{2}{\sigma'}\,f(x\mid\theta',\sigma').$$
Hence $p_{G'}''(x) = \frac{2}{\sigma'}p_{G'}(x)$ for $x\in(R+1,R+2)$, and since $\mathrm{supp}(h)\subseteq(R+1,R+2)$,
$$\left|\int_{\mathbb{R}} h(x)\,p_{G'}''(x)\,dx\right| = \frac{2}{\sigma'}\int_{\mathbb{R}} h(x)\,p_{G'}(x)\,dx.$$
For $x\in[R+1,R+2]$ and $\theta'\in[-R,R]$, the Laplace kernel is bounded below by
$$f(x\mid\theta',\sigma') \geq \frac{1}{\sqrt{2\lambda_{\max}}}\exp\!\left(-\frac{2R+2}{\sqrt{\lambda_{\min}/2}}\right) =: c_{\min} > 0,$$
so $\int_{\mathbb{R}} h(x)\,p_{G'}(x)\,dx \geq c_{\min}\|h\|_1 = c_{\min}$. For the upper bound,
$$\|\mathcal{T}_\sigma h\|_\infty \leq \|h\|_\infty + \tfrac{\lambda_{\max}}{2}\|h''\|_\infty =: C_h < \infty.$$
Substituting into Equation~\eqref{eq:laplace_holder},
$$V\geq \frac{c_{\min}}{C_h\lambda_{\max}}|\sigma-\sigma'|,$$
which confirms Assumption~\ref{ass1} in the univariate case with $\Psi(t) = t$.

\paragraph*{Verification of Assumption~\ref{ass1}: multivariate case.}

Let $u\in\mathbb{R}^d$ with $\|u\|=1$ be the principal unit eigenvector of $\Sigma-\Sigma'$. The characteristic function of $u^\top X$ for $X\sim p_G$ at $t\in\mathbb{R}$ is
$$\mathbb{E}\big[e^{it\,u^\top X}\big] = \Phi_{p_G}(tu) = \Phi_P(tu)\,\Phi_\Sigma(tu) = \Phi_{P_u}(t)\cdot\frac{1}{1+\frac{1}{2}t^2\,u^\top\Sigma u},$$
where $P_u$ is the pushforward of $P$ under $\theta\mapsto u^\top\theta$. By uniqueness of characteristic functions, this identifies the projection $u^\top X$ as a univariate Laplace mixture with scale $\sigma := u^\top\Sigma u$ and location mixing measure $P_u$, and similarly for $u^\top X'$ with scale $\sigma' := u^\top\Sigma' u$. By Lemma~\ref{lemma:dung_total_variance_distance_estimation},
$$\|p_G-p_{G'}\|_1 \geq \|p_{G,u}-p_{G',u}\|_1 \geq \frac{c_{\min}}{C_h\lambda_{\max}}|\sigma-\sigma'| = \frac{c_{\min}}{C_h\lambda_{\max}}\|\Sigma-\Sigma'\|_{\mathrm{op}},$$
where the last equality uses $|\sigma-\sigma'| = |u^\top(\Sigma-\Sigma')u| = \|\Sigma-\Sigma'\|_{\mathrm{op}}$ by the choice of $u$. This establishes Assumption~\ref{ass1} with $\Psi(t) = t$.

\paragraph*{Conclusion.}
With the PDE inversion condition verified, Assumptions 1 and 2 established for $\Psi(t)=\Xi(t)=t$, $p=2$, and ordinary smoothness of order $\beta=2$ being obvious given the form of the Laplace characteristic function, an application of Theorem\ref{theorem:sharpen_ordinary_rate} completes the proof.

\subsection*{Proof of Proposition~\ref{pro:rate_multivariate_normal}}\label{proof_pro5.1}

The validity of both parts of the statement is a consequence of the inequalities established in our Theorems~\ref{thm:inequality_gaussian} and \ref{thm:inequality_gaussian_isotropic}, applied in conjunction with Theorem~9.9 of \citep{ghosal2017fundamentals} (specifically, their results for the super-smooth setting, which yield a $n^{-1/2}(\log n)^{(d+1)/2}$ Hellinger\footnote{The Hellinger distance $h$ between densities $f$ and $f'$ is defined as $h^2(f,f'):=\int(\sqrt{f}-\sqrt{p_{f}})^2dx$. Because $\|f-f'\|_1\lesssim h(f,f')$, any Hellinger contraction rate is also an $L^1$ contraction rate. In particular, a general reverse inequality also exists, so the two metrics are topologically equivalent.} contraction rate). In particular, we rely on the fact if the base measure $H$ is supported on $[-R,R]^d$, then the posterior is supported on mixing measures $G=P\times \delta_\Sigma$ with $P$ having support included in $[-R,R]^d$ \citep{majumdar1992topological}. Moreover, assumptions (9.5), (9.6), and (9.7) adopted in Theorem 9.9 of \cite{ghosal2017fundamentals} are trivially satisfied by our model due to the assumed compact support conditions. Furthermore, their assumption (9.8) is irrelevant in our context, as it is not required in their proof for the super-smooth case.

\subsection*{Proof of Proposition~\ref{pro:rate_univariate_laplace}}

The validity of the statement follows directly from Theorem 4.4 in \cite{scricciolo2011posterior}, which under our assumptions establishes a $n^{-1/4}(\log n)^{5/4}$ Hellinger contraction rate, together with the inequality established in our Theorem~\ref{thm:inequality_laplace}. In particular, we remark that the lower-bound
\begin{equation*}
    \frac{dH}{dx}(\theta)\gtrsim e^{-b|\theta|^\delta} \quad \text{for some } c,\delta>0,
\end{equation*}
as required in the statement of Theorem 4.4 of \cite{scricciolo2011posterior},\footnote{The bound itself is stated in Equation~(4.12) of that same paper.} is irrelevant in our case where $P_0$ is compactly supported, as it is only employed in the portion of the proof of the result dealing with the more general case where only a weaker exponential-power upper-bound on the tails of $P_0$ is assumed.

\section{Auxiliary lemmas with proofs}\label{appB}

\begin{lemma}
\label{lemma:global_extension}
Define the extended function $\bar{h}_{\mathcal{X}}^{*} : \mathbb{R}^d \to \mathbb{R}$ via the infimal convolution operator over the domain $\mathcal{X}$:
$$\bar{h}_{\mathcal{X}}^{*}(x) := \inf_{y \in \mathcal{X}} \Big\{ h_{\mathcal{X}}^{*}(y) + \|x - y\| \Big\}, \quad x \in \mathbb{R}^d.$$
The following properties hold:

\begin{enumerate}
    \item[(a)] $\bar{h}_{\mathcal{X}}^{*}(x) = h_{\mathcal{X}}^{*}(x)$ for all $x \in \mathcal{X}$;
    \item[(b)] $\bar{h}_{\mathcal{X}}^{*}$ is a globally 1-Lipschitz continuous function on $\mathbb{R}^d$, meaning $\bar{h}_{\mathcal{X}}^{*} \in \mathrm{Lip}_1(\mathbb{R}^d)$.
\end{enumerate}

\end{lemma}

\begin{proof}
Since $h_{\mathcal{X}}^{*}\in\mathrm{Lip}_1(\mathcal{X})$ and $h_{\mathcal{X}}^{*}(0)=0$, we have $|h_{\mathcal{X}}^{*}(y)|\leq\|y\|\leq R$ for all $y\in\mathcal{X}$, so $h_{\mathcal{X}}^{*}(y)+\|x-y\|\geq -R$ uniformly in $y$. Hence the infimum defining $\bar{h}_{\mathcal{X}}^{*}(x)$ is finite for every $x\in\mathbb{R}^d$.

\medskip\noindent\textit{Part (a).}
For any $x\in\mathcal{X}$, taking $y=x$ in the infimum gives $\bar{h}_{\mathcal{X}}^{*}(x)\leq h_{\mathcal{X}}^{*}(x)$. Conversely, the $1$-Lipschitz property of $h_{\mathcal{X}}^{*}$ on $\mathcal{X}$ gives $h_{\mathcal{X}}^{*}(x)-h_{\mathcal{X}}^{*}(y)\leq\|x-y\|$ for all $y\in\mathcal{X}$, so $h_{\mathcal{X}}^{*}(x)\leq h_{\mathcal{X}}^{*}(y)+\|x-y\|$ for all $y\in\mathcal{X}$, and taking the infimum over $y$ yields $h_{\mathcal{X}}^{*}(x)\leq\bar{h}_{\mathcal{X}}^{*}(x)$. Combining both inequalities, $\bar{h}_{\mathcal{X}}^{*}(x)=h_{\mathcal{X}}^{*}(x)$ for all $x\in\mathcal{X}$.

\medskip\noindent\textit{Part (b).}
For any $u,v\in\mathbb{R}^d$ and any $y\in\mathcal{X}$, the triangle inequality gives $\|u-y\|\leq\|u-v\|+\|v-y\|$. Adding $h_{\mathcal{X}}^{*}(y)$ to both sides and taking the infimum over $y\in\mathcal{X}$,
$$\bar{h}_{\mathcal{X}}^{*}(u) \leq \|u-v\| + \bar{h}_{\mathcal{X}}^{*}(v).$$
Swapping the roles of $u$ and $v$ and combining yields $|\bar{h}_{\mathcal{X}}^{*}(u)-\bar{h}_{\mathcal{X}}^{*}(v)|\leq\|u-v\|$ for all $u,v\in\mathbb{R}^d$, so $\bar{h}_{\mathcal{X}}^{*}\in\mathrm{Lip}_1(\mathbb{R}^d)$.
\end{proof}

\begin{lemma}
\label{lemma:cutoff_function}
Define the isotropic bump function $\varrho: \mathbb{R}^d \to \mathbb{R}$ as
$$\varrho(x) := \begin{cases} \exp\left(-\frac{1}{1 - \|x\|^2}\right) & \text{if } \|x\| < 1 \\ 0 & \text{if } \|x\| \ge 1. \end{cases}$$
Then define the normalized mollifier $\psi(x) = \frac{1}{C_\varrho} \varrho(x)$, where $C_\varrho := \int_{\mathbb{R}^d} \varrho(x) dx \in (0, \infty)$, and its $\epsilon = \frac{R}{2}$-scaled mollifier as
$$ \psi_\epsilon(x) = \frac{1}{\epsilon^d} \psi\left(\frac{x}{\epsilon}\right).$$
Then define the cutoff function $\eta: \mathbb{R}^d \to \mathbb{R}$ as
$$\eta(x) := (\mathbf{1}_A * \psi_\epsilon)(x) = \int_{\mathbb{R}^d} \mathbf{1}_A(x - y) \psi_\epsilon(y)\, dy,$$
where $A = \overline{B}(0, R + \epsilon) = \overline{B}\left(0, \frac{3R}{2}\right)$. Then the following properties hold:
\begin{enumerate}
    \item[(a)] $\eta \in C_c^\infty(\mathbb{R}^d)$;
    \item[(b)] $\eta(x) = 1$ for all $x \in \overline{B}(0, R)$, $\eta(x) = 0$ for all $x \notin B(0, 2R)$, and $0 \le \eta(x) \le 1$ for all $x \in \mathbb{R}^d$;
    \item[(c)] $M_\eta:= \|\nabla \eta\|_\infty < \infty$.
\end{enumerate}
\end{lemma}

\begin{proof}
The function $g:\mathbb{R}\to\mathbb{R}$ defined by $g(t)=e^{-1/t}$ for $t>0$ and $g(t)=0$ for $t\leq 0$ belongs to $C^\infty(\mathbb{R})$ by Lemma~2.20 in \cite{Lee2013SmoothManifolds}. Since $\varrho(x)=g(1-\|x\|^2)$ is a composition of $g$ with a smooth polynomial, $\varrho\in C^\infty(\mathbb{R}^d)$. It is strictly positive on $B(0,1)$ and vanishes on $\|x\|\geq 1$, so $\mathrm{supp}(\varrho)=\bar{B}(0,1)$ is compact and $\varrho\in C_c^\infty(\mathbb{R}^d)$. Hence $C_\varrho\in(0,\infty)$ and $\psi$ is well-defined with $\int_{\mathbb{R}^d}\psi(x)\,dx=1$. The scaled mollifier satisfies $\psi_\epsilon\in C_c^\infty(\mathbb{R}^d)$, $\mathrm{supp}(\psi_\epsilon)=\bar{B}(0,\epsilon)$, and $\int_{\mathbb{R}^d}\psi_\epsilon(x)\,dx=1$ by the substitution $y=x/\epsilon$.

\medskip\noindent\textit{Part (a).}
For any multi-index $\nu$, $\partial^\nu\eta(x)=\int_{\mathbb{R}^d}\mathbf{1}_A(y)\,(\partial^\nu\psi_\epsilon)(x-y)\,dy$, which is continuous since $\partial^\nu\psi_\epsilon$ is smooth, bounded, and compactly supported. Hence $\eta\in C^\infty(\mathbb{R}^d)$. The support satisfies
$$\mathrm{supp}(\eta)\subseteq\mathrm{supp}(\mathbf{1}_A)+\mathrm{supp}(\psi_\epsilon)=\bar{B}\!\left(0,\tfrac{3R}{2}\right)+\bar{B}\!\left(0,\tfrac{R}{2}\right)=\bar{B}(0,2R),$$
so $\eta\in C_c^\infty(\mathbb{R}^d)$.

\medskip\noindent\textit{Part (b).}
Since $\mathbf{1}_A\geq 0$ and $\psi_\epsilon\geq 0$, we have $\eta(x)\geq 0$. Since $\mathbf{1}_A\leq 1$, we have $\eta(x)\leq\int_{\mathbb{R}^d}\psi_\epsilon(y)\,dy=1$.

For $x\in\bar{B}(0,R)$ and any $y\in\mathrm{supp}(\psi_\epsilon)$, the triangle inequality gives $\|x-y\|\leq\|x\|+\|y\|\leq R+R/2=3R/2$, so $x-y\in A$ and $\mathbf{1}_A(x-y)=1$. Hence $\eta(x)=\int_{\mathbb{R}^d}\psi_\epsilon(y)\,dy=1$.

For $x\notin B(0,2R)$ and any $y\in\mathrm{supp}(\psi_\epsilon)$, $\|x-y\|\geq\|x\|-\|y\|\geq 2R-R/2=3R/2$, so $x-y\notin A$ and $\mathbf{1}_A(x-y)=0$. Hence $\eta(x)=0$. By continuity, $\eta(x)=0$ on $\|x\|=2R$ as well.

\medskip\noindent\textit{Part (c).}
Since $\eta(x)=1$ for $\|x\|<R$ and $\eta(x)=0$ for $\|x\|>2R$, we have $\nabla\eta(x)=0$ outside the compact annulus $\bar{B}(0,2R)\setminus B(0,R)$. As $\eta\in C^\infty(\mathbb{R}^d)$, the function $x\mapsto\|\nabla\eta(x)\|$ is continuous and attains its maximum on this compact set, giving $M_\eta:=\|\nabla\eta\|_\infty=\sup_{x\in\mathbb{R}^d}\|\nabla\eta(x)\|<\infty$.
\end{proof}

\begin{lemma}\label{lemma:cutoff_function_properties}
For $\widetilde{h}_{\mathcal{X}}^{*}(x) := \bar{h}_{\mathcal{X}}^{*}(x)\eta(x)$, the following properties hold:
\begin{enumerate}
    \item[(a)] $\widetilde{h}_{\mathcal{X}}^{*} \in L_{1}(\mathbb{R}^{d})$;
    \item[(b)] $\widetilde{h}_{\mathcal{X}}^{*}$ is globally $(1 + 2RM_{\eta})$-Lipschitz continuous, namely
    $$\forall x,y \in \mathbb{R}^d,\quad|\widetilde{h}_{\mathcal{X}}^{*}(x) - \widetilde{h}_{\mathcal{X}}^{*}(y)| \le (1 + 2RM_{\eta}) \|x - y\|;$$
    \item[(c)] $\|\mathcal{F}[\widetilde{h}_{\mathcal{X}}^{*}]\|_{\infty} \leq \frac{\pi^{d/2}}{\Gamma(d/2 + 1)} (2R)^{d + 1}$.
\end{enumerate}
\end{lemma}

\begin{proof}
\medskip\noindent\textit{Part (a).}
Since $\mathrm{supp}(\widetilde{h}_{\mathcal{X}}^{*})\subseteq\mathrm{supp}(\eta)\subseteq\bar{B}(0,2R)$, and $\bar{h}_{\mathcal{X}}^{*}$ is $1$-Lipschitz with $\bar{h}_{\mathcal{X}}^{*}(0)=0$, we have $|\widetilde{h}_{\mathcal{X}}^{*}(x)|\leq|\bar{h}_{\mathcal{X}}^{*}(x)|\cdot\eta(x)\leq\|x\|\leq 2R$ for all $x\in\mathbb{R}^d$. Therefore,
$$\int_{\mathbb{R}^d}|\widetilde{h}_{\mathcal{X}}^{*}(x)|\,dx \leq \int_{\bar{B}(0,2R)}2R\,dx = \frac{\pi^{d/2}}{\Gamma(d/2+1)}(2R)^{d+1} < \infty,$$
so $\widetilde{h}_{\mathcal{X}}^{*}\in L_1(\mathbb{R}^d)$.

\medskip\noindent\textit{Part (b).}
Since $\bar{h}_{\mathcal{X}}^{*}\in\mathrm{Lip}_1(\mathbb{R}^d)$, Rademacher's theorem gives $\|\nabla\bar{h}_{\mathcal{X}}^{*}(x)\|\leq 1$ for almost every $x\in\mathbb{R}^d$. Since $\eta\in C_c^\infty(\mathbb{R}^d)$, the product rule gives $\nabla\widetilde{h}_{\mathcal{X}}^{*}(x)=\eta(x)\nabla\bar{h}_{\mathcal{X}}^{*}(x)+\bar{h}_{\mathcal{X}}^{*}(x)\nabla\eta(x)$ almost everywhere. By the triangle inequality,
$$\|\nabla\widetilde{h}_{\mathcal{X}}^{*}(x)\| \leq |\eta(x)|\|\nabla\bar{h}_{\mathcal{X}}^{*}(x)\| + |\bar{h}_{\mathcal{X}}^{*}(x)|\|\nabla\eta(x)\| \leq 1 + 2R\,M_\eta,$$
where we used $\eta(x)\leq 1$, $\|\nabla\bar{h}_{\mathcal{X}}^{*}(x)\|\leq 1$, $|\bar{h}_{\mathcal{X}}^{*}(x)|\leq\|x\|\leq 2R$ on $\mathrm{supp}(\nabla\eta)\subseteq\bar{B}(0,2R)$, and $\|\nabla\eta\|_\infty=M_\eta$. Taking the essential supremum, $\widetilde{h}_{\mathcal{X}}^{*}$ is globally $(1+2RM_\eta)$-Lipschitz.

\medskip\noindent\textit{Part (c).}
By the triangle inequality and part (a),
$$|\mathcal{F}[\widetilde{h}_{\mathcal{X}}^{*}](\xi)| \leq \int_{\mathbb{R}^d}|\widetilde{h}_{\mathcal{X}}^{*}(x)|\,dx \leq \frac{\pi^{d/2}}{\Gamma(d/2+1)}(2R)^{d+1}.$$
\end{proof}

\begin{lemma}
\label{lemma:kernel_properties}
Let $M_K := \int_{\mathbb{R}^d} \|x\| K(x)\, dx$. Then the following properties hold:
\begin{enumerate}
    \item[(a)] $\mathcal{F}[K](0) = 1$ and $\mathrm{supp}(\mathcal{F}[K]) \subseteq \bar{B}(0, 1)$;
    \item[(b)] $M_K \leq 3\sqrt{d}(1 + 256d^2)/(8\pi)$.
\end{enumerate}
\end{lemma}

\begin{proof}
\medskip\noindent\textit{Part (a).}
Since $\int_{\mathbb{R}^d}K(x)\,dx=1$, we have $\mathcal{F}[K](0)=1$. For the support bound, the one-dimensional Fourier transform satisfies $\mathcal{F}[\mathrm{sinc}](\xi)=\pi\mathbf{1}_{[-1,1]}(\xi)$, so
$$\mathcal{F}[\mathrm{sinc}^4](\xi) = \frac{1}{(2\pi)^3}\big(\mathcal{F}[\mathrm{sinc}]*\mathcal{F}[\mathrm{sinc}]*\mathcal{F}[\mathrm{sinc}]*\mathcal{F}[\mathrm{sinc}]\big)(\xi),$$
and $\mathrm{supp}(\mathcal{F}[\mathrm{sinc}^4])\subseteq[-4,4]$ by the Minkowski sum property of convolution supports. Setting $k(x):=\frac{3}{8\pi\sqrt{d}}\mathrm{sinc}^4\!\left(\frac{x}{4\sqrt{d}}\right)$, we have $\mathcal{F}[k](\xi)=4\sqrt{d}\,\mathcal{F}[\mathrm{sinc}^4](4\sqrt{d}\,\xi)$, giving $\mathrm{supp}(\mathcal{F}[k])\subseteq[-1/\sqrt{d},\,1/\sqrt{d}]$. Since $K(x)=\prod_{i=1}^d k(x_i)$, the Fourier transform factorizes as $\mathcal{F}[K](\xi)=\prod_{i=1}^d\mathcal{F}[k](\xi_i)$, so
$$\mathrm{supp}(\mathcal{F}[K])\subseteq\prod_{i=1}^d\!\left[-\tfrac{1}{\sqrt{d}},\tfrac{1}{\sqrt{d}}\right].$$
For any $\xi$ in this box, $\|\xi\|\leq\sqrt{d}\cdot\frac{1}{\sqrt{d}}=1$, hence $\mathrm{supp}(\mathcal{F}[K])\subseteq\bar{B}(0,1)$.

\medskip\noindent\textit{Part (b).}
Using $\|x\|\leq\sum_{j=1}^d|x|$ and the product structure of $K$,
$$M_K = \int_{\mathbb{R}^d}\|x\|\,K(x)\,dx \leq \sum_{i=1}^d\frac{3}{8\pi\sqrt{d}}\int_{-\infty}^\infty|x_i|\,\mathrm{sinc}^4\!\left(\frac{x_i}{4\sqrt{d}}\right)dx_i.$$
Splitting each integral at $|x_i|=1$ and using $|\mathrm{sinc}(x_i/4\sqrt{d})|\leq 1$ on $[0,1]$ and $|\sin(x_i/4\sqrt{d})|\leq 1$ on $[1,\infty)$,
$$\int_{-\infty}^\infty|x_i|\,\mathrm{sinc}^4\!\left(\frac{x_i}{4\sqrt{d}}\right)dx_i \leq 2\int_0^1 x_i\,dx_i + 2(4\sqrt{d})^4\int_1^\infty x_i^{-3}\,dx_i = 1+256d^2.$$
Therefore $M_K\leq\frac{3\sqrt{d}(1+256d^2)}{8\pi}<\infty$.
\end{proof}

\begin{lemma}\label{lem:integral_inequality}
In $\mathbb{R}^d$, the following identities and inequalities hold:
\begin{enumerate}
    \item[(a)] $\int_{\bar{B}(0,R)}(1+\|\xi\|^\beta)\,d\xi \leq \frac{2(2\pi)^{d/2}}{d\,\Gamma(d/2)}\max\{R^d,R^{d+\beta}\}$;
    \item[(b)] $\int_{\bar{B}(0,R)}\|\xi\|^p\,d\xi = \frac{(2\pi)^{d/2}}{(d+p)\,\Gamma(d/2)}R^{d+p}.$
\end{enumerate}
\end{lemma}

\begin{proof}
Moving to polar coordinates gives
\begin{align}
    \int_{\bar{B}(0,R)}(a+b\|\xi\|^\beta)\,d\xi 
    = \frac{(2\pi)^{d/2}}{\Gamma(d/2)}\int_0^R r^{d-1}(a+br^\beta)\,dr
    = \frac{(2\pi)^{d/2}}{\Gamma(d/2)}\left(\frac{aR^d}{d}+\frac{bR^{d+\beta}}{d+\beta}\right).
    \label{eq:ball_integral}
\end{align}
Setting $a=b=1$ yields part~(a):
\begin{equation*}
    \int_{\bar{B}(0,R)}(1+\|\xi\|^\beta)\,d\xi 
    = \frac{(2\pi)^{d/2}}{\Gamma(d/2)}\left(\frac{R^d}{d}+\frac{R^{d+\beta}}{d+\beta}\right)
    \leq \frac{2(2\pi)^{d/2}}{d\,\Gamma(d/2)}\max\{R^d,R^{d+\beta}\}.
\end{equation*}
Part~(b) follows from Equation~\eqref{eq:ball_integral} by setting $a=0$, $b=1$, and $\beta=p$.
\end{proof}

\begin{lemma}\label{lemma:gaussian_property}
    Let $X, Y \sim f(\cdot\mid 0,1)$ be independent standard Gaussian random variables. Then $\mathbb{E}[(X+iY)^n] = 0$ for all $n \geq 1$.
\end{lemma}

\begin{proof}
From the moment generating function and characteristic function of the standard Gaussian,
\begin{align*}
    \mathbb{E}\big[e^{t(X+iY)}\big] = \mathbb{E}\big[e^{tX}\big]\mathbb{E}\big[e^{itY}\big] = e^{t^2/2}\cdot e^{-t^2/2} = 1,\quad t\in\mathbb{R}.
\end{align*}
Expanding the left-hand side as a power series and matching coefficients of $t^n/n!$ for each $n\geq 1$ yields the result.
\end{proof}

\begin{lemma}\label{lemma:dung_tail_gaussian_estimation}
    For $\sigma, R, L > 0$ and $k$ a positive integer, the following bounds hold:
    \begin{enumerate}
        \item[(a)] $\displaystyle\int_{x \geq R+L} e^{-(x-R)^2/2\sigma^2}\,dx \leq (\pi/2)^{1/2}\,\sigma\, e^{-L^2/2\sigma^2}$;
        \item[(b)] $\displaystyle\int_{x \geq R+L} x^{2k} e^{-(x-R)^2/2\sigma^2}\,dx \leq 3^{2k-1}\,e^{-L^2/2\sigma^2}\,(\pi/2)^{1/2}\,\sigma \cdot \big(R^{2k} + L^{2k} + \sigma^{2k}(2k-1)!!\big)$.
    \end{enumerate}
\end{lemma}

\begin{proof}
    For part~(a), observe that for $x \geq R+L$ we have $(x-R)^2 = ((x-R-L)+L)^2\geq (x-R-L)^2 + L^2$, so
    \begin{align*}
        \int_{x \geq R+L} e^{-(x-R)^2/2\sigma^2}\,dx \leq e^{-L^2/2\sigma^2}\int_0^\infty e^{-t^2/2\sigma^2}\,dt = (\pi/2)^{1/2}\,\sigma\, e^{-L^2/2\sigma^2}.
    \end{align*}
    For part~(b), use $x^{2k} \leq 3^{2k-1}\big((x-R-L)^{2k} + R^{2k} + L^{2k}\big)$ for $x\geq R+L$. Setting $t = x-R-L$ and applying the same exponential bound as in part~(a),
    \begin{align*}
        \int_{x \geq R+L} x^{2k} e^{-(x-R)^2/2\sigma^2}\,dx
        &\leq 3^{2k-1}\,e^{-L^2/2\sigma^2}\int_0^\infty (t^{2k} + R^{2k} + L^{2k})\,e^{-t^2/2\sigma^2}\,dt \\
        &= 3^{2k-1}\,e^{-L^2/2\sigma^2}\,(\pi/2)^{1/2}\,\sigma\cdot\big(R^{2k} + L^{2k} + \sigma^{2k}(2k-1)!!\big),
    \end{align*}
    where the last equality uses $\int_0^\infty t^{2k} e^{-t^2/2\sigma^2}\,dt = (\pi/2)^{1/2}\sigma^{2k+1}(2k-1)!!$.
\end{proof}

\begin{lemma}\label{lemma:dung_total_variance_distance_estimation}
    Let $X$ and $Y$ be continuous random vectors in $\mathbb{R}^d$ with densities $p_X$ and $p_Y$, and let $u\in\mathbb{R}^d$ with $\|u\|=1$. Denoting by $X_u := u^\top X$ and $Y_u := u^\top Y$ the projections onto $u$,
    \begin{equation*}
        \|p_X - p_Y\|_1 \geq \|p_{X_u} - p_{Y_u}\|_1.
    \end{equation*}
\end{lemma}

\begin{proof}
    For any Borel set $A\subseteq\mathbb{R}$, the preimage $A^u := \{x\in\mathbb{R}^d : u^\top x \in A\}$ is a Borel set in $\mathbb{R}^d$. Therefore,
    \begin{align*}
        \tfrac{1}{2}\|p_{X_u} - p_{Y_u}\|_1
        &= \sup_{A\in\mathscr{B}(\mathbb{R})}\big|\mathbb{P}(X_u\in A) - \mathbb{P}(Y_u\in A)\big|
        = \sup_{A\in\mathscr{B}(\mathbb{R})}\big|\mathbb{P}(X\in A^u) - \mathbb{P}(Y\in A^u)\big|\\
        &\leq \sup_{B\in\mathscr{B}(\mathbb{R}^d)}\big|\mathbb{P}(X\in B) - \mathbb{P}(Y\in B)\big|
        = \tfrac{1}{2}\|p_X - p_Y\|_1,
    \end{align*}
    which completes the proof.
\end{proof}

\begin{lemma}\label{lemma:exp_larger_polynomial}
    For any $p\geq 0$ and $\varepsilon > 0$, there exists a constant $M_{p,\varepsilon}>0$, depending only on $p$ and $\varepsilon$, such that $u^p \leq M_{p,\varepsilon}\,e^{\varepsilon u}$ for all $u > 0$.
\end{lemma}

\begin{proof}
    Let $k := \lceil p\rceil$. From the Taylor expansion of $e^{\varepsilon u}$, we have
    \begin{equation*}
        e^{\varepsilon u} \geq 1 + \frac{\varepsilon^k u^k}{k!} \geq \min\!\left(1,\frac{\varepsilon^k}{k!}\right)\max\{1, u^k\} \geq \min\!\left(1,\frac{\varepsilon^k}{k!}\right) u^p.
    \end{equation*}
    Dividing both sides by $\min\!\left(1,\frac{\varepsilon^k}{k!}\right)$ yields
    \begin{equation*}
        u^p \leq \max\!\left(1,\frac{k!}{\varepsilon^k}\right) e^{\varepsilon u}.
    \end{equation*}
    The result follows by setting $M_{p,\varepsilon} := \max\!\left(1,\frac{k!}{\varepsilon^k}\right)$.
\end{proof}

\begin{lemma}
\label{lemma:dung_ascoli_for_non_compact_measure}
Let $\mu$ and $\nu$ be two probability measures on $\mathbb{R}^d$ with finite first moments, i.e.,
\begin{equation*}
    \int_{\mathbb{R}^d}\|x\|\,\mu(dx) < \infty, \quad \int_{\mathbb{R}^d}\|x\|\,\nu(dx) < \infty.
\end{equation*}
Then there exists a function $h^* \in \mathrm{Lip}_1(\mathbb{R}^d)$ with $h^*(0) = 0$ that attains the Kantorovich-Rubinstein dual problem:
\begin{equation*}
\sup_{h \in \mathrm{Lip}_1(\mathbb{R}^d)} \left( \int_{\mathbb{R}^d} h(x)\, \mu(dx) - \int_{\mathbb{R}^d} h(x)\, \nu(dx) \right) = \int_{\mathbb{R}^d} h^*(x)\, \mu(dx) - \int_{\mathbb{R}^d} h^*(x)\, \nu(dx).
\end{equation*}
\end{lemma}

\begin{proof}
Let $(h_n)_{n\in\mathbb N}$ be a maximizing sequence for the Kantorovich-Rubinstein dual problem. Since adding a constant to any $h_n$ does not affect $\int h_n\, d\mu - \int h_n\, d\nu$ (as $\mu$ and $\nu$ are both probability measures), we may assume without loss of generality that $h_n(0) = 0$ for all $n$. Consequently, $|h_n(x)| = |h_n(x) - h_n(0)| \leq \|x\|$ for all $x \in \mathbb{R}^d$, so the family $(h_n)_{n\in\mathbb N}$ is uniformly bounded and equicontinuous on every compact subset of $\mathbb{R}^d$.

We now apply a diagonal argument with the Arzel\`a-Ascoli theorem. Applying Arzel\`a-Ascoli on $\bar{B}(0,1)$, we extract a subsequence $(h^1_n)_{n\in\mathbb N} \subset (h_n)_{n\in\mathbb N}$ converging uniformly on $\bar{B}(0,1)$. Iteratively, for each $k \geq 1$, by Arzel\`a-Ascoli on $\bar{B}(0,k+1)$, we extract a subsequence $(h^{k+1}_n)_{n\in\mathbb N} \subset (h^k_n)_{n\in\mathbb N}$ converging uniformly on $\bar{B}(0,k+1)$. The diagonal sequence $(h^n_n)_{n\in\mathbb N}$ then converges uniformly on every compact subset of $\mathbb{R}^d$ to some pointwise limit $h^*$. Since uniform-on-compacts limits of $1$-Lipschitz functions are $1$-Lipschitz, $h^* \in \mathrm{Lip}_1(\mathbb{R}^d)$ and clearly $h^*(0) = 0$.

It remains to show that $h^*$ attains the supremum. We have $h^n_n(x) \to h^*(x)$ pointwise, and the bound $|h^n_n(x)| \leq \|x\|$ holds uniformly in $n$. Since
\begin{equation*}
    \int_{\mathbb{R}^d} \|x\|\, \mu(dx) + \int_{\mathbb{R}^d} \|x\|\, \nu(dx) < \infty,
\end{equation*}
the Dominated Convergence Theorem (applied separately to $\mu$ and $\nu$) yields
\begin{equation*}
    \int_{\mathbb{R}^d} h^n_n(x)\, \mu(dx) - \int_{\mathbb{R}^d} h^n_n(x)\, \nu(dx) \;\longrightarrow\; \int_{\mathbb{R}^d} h^*(x)\, \mu(dx) - \int_{\mathbb{R}^d} h^*(x)\, \nu(dx).
\end{equation*}
Since $(h_n)_{n\in\mathbb N}$ is a maximizing sequence, the left-hand side converges to $W_1(\mu,\nu)$, so $h^*$ attains the supremum. This completes the proof.
\end{proof}

\begin{lemma}
\label{lemma:dung_a_small_estimation_from_xlogx}
Let $x, y, M > 0$ satisfy $x \leq M y \log y$ and $x \geq e^{1/M}$. Then $y \geq x/(M \log x)$.
\end{lemma}

\begin{proof}
Define $f(t) := t \log t$ for $t > 0$. Since $f'(t) = \log t + 1$, the function $f$ is decreasing on $(0, e^{-1})$ and increasing on $(e^{-1}, \infty)$, with global minimum $f(e^{-1}) = -e^{-1}$. From $x > 0$ and $x \leq M y \log y$, we have $y \log y > 0$, hence $y > 1$.

The hypothesis $x \geq e^{1/M}$ gives $\log x \geq 1/M$, hence $M \log x \geq 1$, which implies $x/(M \log x) \leq x$. Therefore
\begin{equation*}
    f\!\left(\frac{x}{M \log x}\right) = \frac{x}{M \log x} \log\!\left(\frac{x}{M \log x}\right) \leq \frac{x}{M \log x} \log x = \frac{x}{M}.
\end{equation*}
Combined with $x/M \leq y \log y = f(y)$, this gives $f(x/(M \log x)) \leq f(y)$.

We conclude with a brief case analysis. If $x/(M \log x) \leq 1$, then since $y > 1$, the conclusion $y \geq x/(M \log x)$ holds trivially. Otherwise, $x/(M \log x) > 1$, and both $x/(M \log x)$ and $y$ lie in $(1, \infty) \subset (e^{-1}, \infty)$, where $f$ is strictly increasing; hence $f(x/(M \log x)) \leq f(y)$ implies $x/(M \log x) \leq y$.
\end{proof}

\bibliographystyle{apalike} 
\bibliography{arxiv}       

\begin{thebibliography}{}

\bibitem[Ascolani et~al., 2023]{ascolani2023clustering}
Ascolani, F., Lijoi, A., Rebaudo, G., and Zanella, G. (2023).
\newblock {Clustering consistency with Dirichlet process mixtures}.
\newblock {\em Biometrika}, 110(2):551--558.

\bibitem[Bariletto et~al., 2025]{bariletto2025posterior}
Bariletto, N., Flores, B., and Walker, S.~G. (2025).
\newblock {Posterior Consistency in Parametric Models via a Tighter Notion of Identifiability}.
\newblock {\em arXiv preprint arXiv:2504.11360}.

\bibitem[Barron et~al., 1999]{Barron-Shervish-Wasserman-99}
Barron, A., Schervish, M., and Wasserman, L. (1999).
\newblock The consistency of posterior distributions in nonparametric problems.
\newblock {\em The Annals of Statistics}, 27:536--561.

\bibitem[Butzer and Nessel, 1971]{butzer1971fourier}
Butzer, P.~L. and Nessel, R.~J. (1971).
\newblock {\em Fourier analysis and approximation, Vol. 1}, volume~7.
\newblock Birkh{\"a}user Basel.

\bibitem[Caillerie et~al., 2011]{Caillerie-etal-11}
Caillerie, C., Chazal, F., Dedecker, J., and Michel, B. (2011).
\newblock Deconvolution for the {W}asserstein metric and geometric inference.
\newblock {\em Electronic Journal of Statistics}, 5:1394--1423.

\bibitem[Canale and De~Blasi, 2017]{canale2017posterior}
Canale, A. and De~Blasi, P. (2017).
\newblock Posterior asymptotics of nonparametric location-scale mixtures for multivariate density estimation.
\newblock {\em Bernoulli}, 23(1):379--404.

\bibitem[Castillo, 2024]{castillo2024bnp}
Castillo, I. (2024).
\newblock {\em {Bayesian Nonparametric Statistics}}.
\newblock École d’Été de Probabilités de Saint-Flour LI - 2023. Springer Nature.

\bibitem[Catalano and Lavenant, 2025]{catalano2025measures}
Catalano, M. and Lavenant, H. (2025).
\newblock Measures of dependence based on {W}asserstein distances.
\newblock {\em arXiv preprint arXiv:2510.06034}.

\bibitem[Catalano et~al., 2024]{catalano2024wasserstein}
Catalano, M., Lavenant, H., Lijoi, A., and Pr{\"u}nster, I. (2024).
\newblock A {W}asserstein index of dependence for random measures.
\newblock {\em Journal of the American Statistical Association}, 119(547):2396--2406.

\bibitem[Catalano et~al., 2021]{catalano2021measuring}
Catalano, M., Lijoi, A., and Pr{\"u}nster, I. (2021).
\newblock {Measuring dependence in the Wasserstein distance for Bayesian nonparametric models}.
\newblock {\em The Annals of Statistics}, 49(5):2916--2947.

\bibitem[Chae and Walker, 2017]{chae2017}
Chae, M. and Walker, S.~G. (2017).
\newblock {A novel approach to Bayesian consistency}.
\newblock {\em Electronic Journal of Statistics}, 11(2):4723--4745.

\bibitem[De~Blasi et~al., 2015]{deblasi2015gibbs}
De~Blasi, P., Favaro, S., Lijoi, A., Mena, R.~H., Pr{\"u}nster, I., and Ruggiero, M. (2015).
\newblock Are {Gibbs}-type priors the most natural generalization of the {Dirichlet} process?
\newblock {\em IEEE Transactions on Pattern Analysis and Machine Intelligence}, 37(2):212--229.

\bibitem[Dedecker et~al., 2015]{dedecker2015improved}
Dedecker, J., Fischer, A., and Michel, B. (2015).
\newblock Improved rates for {W}asserstein deconvolution with ordinary smooth error in dimension one.
\newblock {\em Electronic Journal of Statistics}, 9(1):234--265.

\bibitem[Dedecker and Michel, 2013]{dedecker2013minimax}
Dedecker, J. and Michel, B. (2013).
\newblock Minimax rates of convergence for {W}asserstein deconvolution with supersmooth errors in any dimension.
\newblock {\em Journal of Multivariate Analysis}, 122:278--291.

\bibitem[DeVore and Lorentz, 1993]{devore1993constructive}
DeVore, R.~A. and Lorentz, G.~G. (1993).
\newblock {\em Constructive approximation}, volume 303.
\newblock Springer Science \& Business Media.

\bibitem[Escobar and West, 1995]{escobar1995bayesian}
Escobar, M.~D. and West, M. (1995).
\newblock Bayesian density estimation and inference using mixtures.
\newblock {\em Journal of the American Statistical Association}, 90(430):577--588.

\bibitem[Ferguson, 1973]{Ferguson}
Ferguson, T. (1973).
\newblock A {B}ayesian analysis of some nonparametric problems.
\newblock {\em The Annals of Statistics}, 1:209--230.

\bibitem[Ferguson, 1983]{ferguson1983bayesian}
Ferguson, T.~S. (1983).
\newblock Bayesian density estimation by mixtures of normal distributions.
\newblock In {\em Recent advances in statistics}, pages 287--302. Elsevier.

\bibitem[Folland, 1999]{folland1999real}
Folland, G.~B. (1999).
\newblock {\em Real analysis: modern techniques and their applications}.
\newblock John Wiley \& Sons.

\bibitem[Fraley and Raftery, 2002]{fraley2002model}
Fraley, C. and Raftery, A.~E. (2002).
\newblock Model-based clustering, discriminant analysis, and density estimation.
\newblock {\em Journal of the American Statistical Association}, 97(458):611--631.

\bibitem[Franzolini et~al., 2025]{franzolini2025multivariate}
Franzolini, B., Lijoi, A., Pr{\"u}nster, I., and Rebaudo, G. (2025).
\newblock Multivariate species sampling models.
\newblock {\em arXiv preprint arXiv:2503.24004}.

\bibitem[Gao and van~der Vaart, 2016]{Gao2016Posterior}
Gao, F. and van~der Vaart, A. (2016).
\newblock Posterior contraction rates for deconvolution of {Dirichlet-Laplace} mixtures.
\newblock {\em Electronic Journal of Statistics}, 10(1):608--627.

\bibitem[Genovese and Wasserman, 2000]{genovese2000rates}
Genovese, C.~R. and Wasserman, L. (2000).
\newblock {Rates of convergence for the Gaussian mixture sieve}.
\newblock {\em The Annals of Statistics}, 28(4):1105--1127.

\bibitem[Ghosal et~al., 1999]{Ghosal-1999}
Ghosal, S., Ghosh, J.~K., and Ramamoorthi, R.~V. (1999).
\newblock Posterior consistency of {D}irichlet mixtures in density estimation.
\newblock {\em The Annals of Statistics}, 27:143--158.

\bibitem[Ghosal et~al., 2000]{Ghosal-2000}
Ghosal, S., Ghosh, J.~K., and van~der Vaart, A. (2000).
\newblock Convergence rates of posterior distributions.
\newblock {\em The Annals of Statistics}, 28:500--531.

\bibitem[Ghosal and van~der Vaart, 2007a]{ghosal2007noniid}
Ghosal, S. and van~der Vaart, A. (2007a).
\newblock {Convergence rates of posterior distributions for noniid observations}.
\newblock {\em The Annals of Statistics}, 35(1):192 -- 223.

\bibitem[Ghosal and van~der Vaart, 2007b]{Ghosal-2007}
Ghosal, S. and van~der Vaart, A. (2007b).
\newblock Posterior convergence rates of {D}irichlet mixtures at smooth densities.
\newblock {\em The Annals of Statistics}, 35:697--723.

\bibitem[Ghosal and Van~der Vaart, 2017]{ghosal2017fundamentals}
Ghosal, S. and Van~der Vaart, A. (2017).
\newblock {\em {Fundamentals of nonparametric Bayesian inference}}, volume~44.
\newblock Cambridge University Press.

\bibitem[Ghosal and van~der Vaart, 2001]{ghosal2001entropies}
Ghosal, S. and van~der Vaart, A.~W. (2001).
\newblock Entropies and rates of convergence for maximum likelihood and {B}ayes estimation for mixtures of normal densities.
\newblock {\em The Annals of Statistics}, 29(5):1233--1263.

\bibitem[Gin{\'e} and Nickl, 2011]{gine2011rates}
Gin{\'e}, E. and Nickl, R. (2011).
\newblock {Rates of contraction for posterior distributions in $L^r$-metrics, $1 \leq r \leq\infty$}.
\newblock {\em The Annals of Statistics}, 39(6):2883 -- 2911.

\bibitem[Gnedin and Pitman, 2006]{GnedinPitman2006}
Gnedin, A. and Pitman, J. (2006).
\newblock {Exchangeable Gibbs partitions and Stirling triangles}.
\newblock {\em Journal of Mathematical Sciences}, 138(3):5674--5685.

\bibitem[Guha et~al., 2021]{guha2021Bernoulli}
Guha, A., Ho, N., and Nguyen, X. (2021).
\newblock On posterior contraction of parameters and interpretability in {{Bayesian}} mixture modeling.
\newblock {\em Bernoulli}, 27(4):2159--2188.

\bibitem[Heinonen, 2001]{heinonen2001lectures}
Heinonen, J. (2001).
\newblock {\em Lectures on analysis on metric spaces}.
\newblock Springer Science \& Business Media.

\bibitem[Ho and Nguyen, 2016a]{Ho-Nguyen-Ann-16}
Ho, N. and Nguyen, X. (2016a).
\newblock Convergence rates of parameter estimation for some weakly identifiable finite mixtures.
\newblock {\em The Annals of Statistics}, 44:2726--2755.

\bibitem[Ho and Nguyen, 2016b]{Ho-Nguyen-EJS-16}
Ho, N. and Nguyen, X. (2016b).
\newblock On strong identifiability and convergence rates of parameter estimation in finite mixtures.
\newblock {\em Electronic Journal of Statistics}, 10:271--307.

\bibitem[Kingman, 1978]{kingman1978representation}
Kingman, J.~F. (1978).
\newblock The representation of partition structures.
\newblock {\em Journal of the London Mathematical Society}, 2(2):374--380.

\bibitem[Kingman, 1967]{kingman1967completely}
Kingman, J. F.~C. (1967).
\newblock Completely random measures.
\newblock {\em Pacific Journal of Mathematics}, 21(1):59--78.

\bibitem[Lee, 2013]{Lee2013SmoothManifolds}
Lee, J.~M. (2013).
\newblock {\em Introduction to Smooth Manifolds}, volume 218 of {\em Graduate Texts in Mathematics}.
\newblock Springer, New York, 2 edition.

\bibitem[Lijoi et~al., 2005a]{lijoi2005hierarchical}
Lijoi, A., Mena, R.~H., and Pr{\"u}nster, I. (2005a).
\newblock Hierarchical mixture modeling with normalized inverse-{Gaussian} priors.
\newblock {\em Journal of the American Statistical Association}, 100(472):1278--1291.

\bibitem[Lijoi et~al., 2014]{lijoi2014bayesian}
Lijoi, A., Nipoti, B., and Pr{\"u}nster, I. (2014).
\newblock Bayesian inference with dependent normalized completely random measures.
\newblock {\em Bernoulli}, pages 1260--1291.

\bibitem[Lijoi and Pr{\"u}nster, 2010]{lijoi2010models}
Lijoi, A. and Pr{\"u}nster, I. (2010).
\newblock Models beyond the {Dirichlet} process.
\newblock In Hjort, N.~L., Holmes, C., M{\"u}ller, P., and Walker, S.~G., editors, {\em Bayesian Nonparametrics}, pages 80--136. Cambridge University Press.

\bibitem[Lijoi et~al., 2005b]{lijoi2005consistency}
Lijoi, A., Pr{\"u}nster, I., and Walker, S.~G. (2005b).
\newblock On consistency of nonparametric normal mixtures for {B}ayesian density estimation.
\newblock {\em Journal of the American Statistical Association}, 100(472):1292--1296.

\bibitem[Lo, 1984]{lo1984class}
Lo, A.~Y. (1984).
\newblock {On a Class of Bayesian Nonparametric Estimates: I. Density Estimates}.
\newblock {\em The Annals of Statistics}, 12(1):351--357.

\bibitem[MacEachern and M{\"u}ller, 1998]{maceachern1998estimating}
MacEachern, S.~N. and M{\"u}ller, P. (1998).
\newblock Estimating mixture of {D}irichlet process models.
\newblock {\em Journal of Computational and Graphical Statistics}, 7(2):223--238.

\bibitem[Majumdar, 1992]{majumdar1992topological}
Majumdar, S. (1992).
\newblock On topological support of {D}irichlet prior.
\newblock {\em Statistics \& Probability Letters}, 15(5):381--384.

\bibitem[McLachlan and Peel, 2000]{mclachlan2000finite}
McLachlan, G.~J. and Peel, D. (2000).
\newblock {\em Finite Mixture Models}.
\newblock John Wiley \& Sons.

\bibitem[McShane, 1934]{McShane1934}
McShane, E.~J. (1934).
\newblock Extension of range of functions.
\newblock {\em Bulletin of the American Mathematical Society}, 40:837--842.

\bibitem[Miller and Harrison, 2013]{miller2013simple}
Miller, J.~W. and Harrison, M.~T. (2013).
\newblock A simple example of {D}irichlet process mixture inconsistency for the number of components.
\newblock In {\em Advances in Neural Information Processing Systems}, volume~26.

\bibitem[Miller and Harrison, 2014]{miller2014inconsistency}
Miller, J.~W. and Harrison, M.~T. (2014).
\newblock Inconsistency of {P}itman--{Y}or process mixtures for the number of components.
\newblock {\em Journal of Machine Learning Research}, 15:3333--3370.

\bibitem[Neal, 2000]{neal2000markov}
Neal, R.~M. (2000).
\newblock {Markov chain sampling methods for Dirichlet process mixture models}.
\newblock {\em Journal of Computational and Graphical Statistics}, 9(2):249--265.

\bibitem[Nguyen, 2013]{Nguyen-13}
Nguyen, X. (2013).
\newblock Convergence of latent mixing measures in finite and infinite mixture models.
\newblock {\em The Annals of Statistics}, 4(1):370--400.

\bibitem[Nieto-Barajas et~al., 2004]{Nieto-Barajas:2004:NRM}
Nieto-Barajas, L.~E., Pr{\"u}nster, I., and Walker, S.~G. (2004).
\newblock Normalized random measures driven by increasing additive processes.
\newblock {\em The Annals of Statistics}, 32(6):2343--2360.

\bibitem[Perman et~al., 1992]{perman1992size}
Perman, M., Pitman, J., and Yor, M. (1992).
\newblock Size-biased sampling of {Poisson} point processes and excursions.
\newblock {\em Probability Theory and Related Fields}, 92(1):21--39.

\bibitem[Petralia et~al., 2012]{petralia2012repulsive}
Petralia, F., Rao, V., and Dunson, D. (2012).
\newblock Repulsive mixtures.
\newblock {\em Advances in neural information processing systems}, 25.

\bibitem[Pitman, 1996]{pitman1996some}
Pitman, J. (1996).
\newblock Some developments of the {Blackwell-MacQueen} urn scheme.
\newblock In {\em Statistics, Probability and Game Theory: Papers in Honor of David Blackwell}, volume~30, pages 245--267. Institute of Mathematical Statistics.

\bibitem[Pitman and Yor, 1997]{pitman1997two}
Pitman, J. and Yor, M. (1997).
\newblock The two-parameter {Poisson-Dirichlet} distribution derived from a stable subordinator.
\newblock {\em The Annals of Probability}, 25(2):855--900.

\bibitem[Regazzini et~al., 2003]{regazzini2003distributional}
Regazzini, E., Lijoi, A., and Pr{\"u}nster, I. (2003).
\newblock Distributional results for means of normalized random measures with independent increments.
\newblock {\em The Annals of Statistics}, 31(2):560--585.

\bibitem[Rodriguez et~al., 2008]{rodriguez2008nested}
Rodriguez, A., Dunson, D.~B., and Gelfand, A.~E. (2008).
\newblock The nested {D}irichlet process.
\newblock {\em Journal of the American Statistical Association}, 103(483):1131--1154.

\bibitem[Rousseau and Mengersen, 2011]{rousseau2011asymptotic}
Rousseau, J. and Mengersen, K. (2011).
\newblock Asymptotic behaviour of the posterior distribution in overfitted mixture models.
\newblock {\em Journal of the Royal Statistical Society Series B: Statistical Methodology}, 73(5):689--710.

\bibitem[Rousseau and Scricciolo, 2024]{rousseau2024wasserstein}
Rousseau, J. and Scricciolo, C. (2024).
\newblock {Wasserstein convergence in Bayesian and frequentist deconvolution models}.
\newblock {\em The Annals of Statistics}, 52(4):1691--1715.

\bibitem[Rudin, 1991]{rudin1991functional}
Rudin, W. (1991).
\newblock {\em Functional Analysis}.
\newblock McGraw-Hill, New York, second edition.

\bibitem[Schwartz, 1966]{schwartz1966theorie}
Schwartz, L. (1966).
\newblock {\em Th{\'e}orie des distributions}.
\newblock Hermann, Paris.

\bibitem[Scricciolo, 2011]{scricciolo2011posterior}
Scricciolo, C. (2011).
\newblock {Posterior rates of convergence for Dirichlet mixtures of exponential power densities}.
\newblock {\em Electronic Journal of Statistics}, 5:270--308.

\bibitem[Scricciolo, 2014]{scricciolo2014adaptive}
Scricciolo, C. (2014).
\newblock {Adaptive Bayesian density estimation in Lp-metrics with Pitman-Yor or normalized inverse-Gaussian process kernel mixtures}.
\newblock {\em Bayesian Analysis}, 9(2):475--520.

\bibitem[Scricciolo, 2018]{scricciolo2018bayes}
Scricciolo, C. (2018).
\newblock {Bayes and maximum likelihood for $L1$-Wasserstein deconvolution of {L}aplace mixtures}.
\newblock {\em Statistical Methods \& Applications}, 27(2):333--362.

\bibitem[Sethuraman, 1994]{Sethuraman}
Sethuraman, J. (1994).
\newblock A constructive definition of {D}irichlet priors.
\newblock {\em Statistica Sinica}, 4:639--650.

\bibitem[Shen et~al., 2013]{shen2013adaptive}
Shen, W., Tokdar, S.~T., and Ghosal, S. (2013).
\newblock Adaptive {B}ayesian multivariate density estimation with {D}irichlet mixtures.
\newblock {\em Biometrika}, 100(3):623--640.

\bibitem[Shen and Wasserman, 2001]{shen2001rates}
Shen, X. and Wasserman, L. (2001).
\newblock Rates of convergence of posterior distributions.
\newblock {\em The Annals of Statistics}, 29(3):687--714.

\bibitem[Teh et~al., 2006]{Teh-et-al06}
Teh, Y., Jordan, M., Beal, M., and Blei, D. (2006).
\newblock Hierarchical {D}irichlet processes.
\newblock {\em Journal of the American Statistical Association}, 101:1566--1581.

\bibitem[Tokdar, 2006]{tokdar2006posterior}
Tokdar, S.~T. (2006).
\newblock Posterior consistency of {D}irichlet location-scale mixture of normals in density estimation and regression.
\newblock {\em Sankhy{\=a}: The Indian Journal of Statistics}, 67(4):90--110.

\bibitem[Villani, 2003]{Villani-03}
Villani, C. (2003).
\newblock {\em Topics in Optimal Transportation}.
\newblock American Mathematical Society.

\bibitem[Villani, 2008]{Villani-09}
Villani, C. (2008).
\newblock {\em Optimal transport: Old and New}.
\newblock Springer.

\bibitem[Wade, 2023]{wade2023bayesian}
Wade, S. (2023).
\newblock Bayesian cluster analysis.
\newblock {\em Philosophical Transactions of the Royal Society A}, 381(2247):20220149.

\bibitem[Walker, 2004]{walker2004squarerootsum}
Walker, S.~G. (2004).
\newblock {New approaches to Bayesian consistency}.
\newblock {\em The Annals of Statistics}, 32(5):2028 -- 2043.

\bibitem[Walker, 2007]{walker2007sampling}
Walker, S.~G. (2007).
\newblock Sampling the {D}irichlet mixture model with slices.
\newblock {\em Communications in Statistics—Simulation and Computation{\textregistered}}, 36(1):45--54.

\bibitem[Walker and Hjort, 2001]{walker2001}
Walker, S.~G. and Hjort, N.~L. (2001).
\newblock On {B}ayesian consistency.
\newblock {\em Journal of the Royal Statistical Society, Series B}, 63:811--821.

\bibitem[Walker et~al., 2007]{walker2007rates}
Walker, S.~G., Lijoi, A., and Pr{\"u}nster, I. (2007).
\newblock {On rates of convergence for posterior distributions in infinite-dimensional models}.
\newblock {\em The Annals of Statistics}, 35(2):738--746.

\bibitem[Whitney, 1934]{Whitney1934}
Whitney, H. (1934).
\newblock Analytic extensions of differentiable functions defined in closed sets.
\newblock {\em Transactions of the American Mathematical Society}, 36:63.

\bibitem[Wu and Ghosal, 2010]{wu2010l1}
Wu, Y. and Ghosal, S. (2010).
\newblock {The $L_1$-consistency of Dirichlet mixtures in multivariate Bayesian density estimation}.
\newblock {\em Journal of Multivariate Analysis}, 101(10):2411--2419.

\bibitem[Xu et~al., 2016]{xu2016bayesian}
Xu, Y., M{\"u}ller, P., and Telesca, D. (2016).
\newblock Bayesian inference for latent biologic structure with determinantal point processes (dpp).
\newblock {\em Biometrics}, 72(3):955--964.

\end{thebibliography}

\end{document}